\documentclass{article}
\usepackage{amsmath}
\usepackage{amssymb}
\usepackage{slashed}
\usepackage[letterpaper]{hyperref}

\newtheorem{theorem}{Theorem}[section]

\newtheorem{corollary}[theorem]{Corollary}

\newtheorem{definition}[theorem]{Definition}
\newtheorem{example}[theorem]{Example}

\newtheorem{lemma}[theorem]{Lemma}

\newtheorem{proposition}[theorem]{Proposition}
\newtheorem{remark}[theorem]{Remark}

\newenvironment{proof}[1][Proof]{\noindent\textbf{#1.} }{\ \rule{0.5em}{0.5em}}
\renewcommand{\theequation}{\thesection.\arabic{equation}}
\input{tcilatex}
\begin{document}

\title{$G_{2}$-structures and octonion bundles}
\author{Sergey Grigorian \\
School of Mathematical and Statistical Sciences\\
University of Texas Rio Grande Valley\\
1201 W. University Drive\\
Edinburg, TX 78539}
\maketitle

\begin{abstract}
We use a $G_{2}$-structure on a $7$-dimensional Riemannian manifold with a
fixed metric to define an octonion bundle with a fiberwise non-associative
product. We then define a metric-compatible octonionic covariant derivative
on this bundle that is compatible with the octonion product. The torsion of
the $G_{2}$-structure is then shown to be an octonionic connection for this
covariant derivative with curvature given by the component of the Riemann
curvature that lies in the $7$-dimensional representation of $G_{2}.$ We
also interpret the choice of a particular $G_{2}$-structure within the same
metric class as a choice of gauge and show that under a change of this
gauge, the torsion does transform as an octonion-valued connection $1$-form.
Finally, we also show an explicit relationship between the octonion bundle
and the spinor bundle, define an octonionic Dirac operator and explore an
energy functional for octonion sections. We then prove that critical points
correspond to divergence-free torsion, which is shown to be an octonionic
analog of the Coulomb gauge.
\end{abstract}

\tableofcontents

\section{Introduction}

Seven-dimensional manifolds with a $G_{2}$-structure have been of great
interest in differential geometry and theoretical physics ever since Alfred
Gray studied vector cross products on orientable manifolds in 1969 \cite%
{Gray-VCP}. It turns out that a $2$-fold vector cross product - that is, one
that takes two vectors and outputs another one, exists only in $3$
dimensions and in $7$ dimensions. The $3$-dimensional vector cross product
is very well known in $\mathbb{R}^{3}$ and on a general oriented $3$%
-manifold it comes from the volume $3$-form, so it is a special case of a $%
\left( n-1\right) $-fold vector cross product in a $n$-dimensional space,
where it also comes from the volume form. In $7$ dimensions, however, the
vector cross product structure is even more special, since it is not part of
an infinite sequence. The $3$-dimensional and $7$-dimensional vector cross
products do however have something in common since they are closely related
to the normed division algebras - the quaternions and octonions, which are $%
4 $ and $8$ dimensional, respectively. In fact, the $3$-dimensional vector
cross product can be obtained by restricting the quaternion product to the
purely imaginary quaternions and then taking the projection to the imaginary
part. The $7$-dimensional vector cross product is induced from the octonion
product in a similar way. Note that the only other normed division algebras
are $\mathbb{R}\ $and $\mathbb{C},$ so the only non-trivial vector cross
products obtained this way are in $3$ and $7$ dimensions. In $3$ dimensions,
the group that preserves the vector product is $SO\left( 3\right) ,$
therefore, on a $3$-dimensional manifold, given an oriented orthonormal
frame bundle, i.e. an $SO\left( 3\right) $-structure, we can always define a
vector cross product, and moreover in will be parallel with respect to the
Levi-Civita connection. On a $7$-manifold, the group that preserves the
vector cross product is now $G_{2}$ - this is the automorphism group of the
octonion algebra, which is in particular a $14$-dimensional exceptional Lie
group. Therefore, in order to be able to define a vector cross product
globally on a $7$-manifold, we need to introduce a $G_{2}$\emph{-structure},
which is now a reduction of the frame bundle to $G_{2}.$ There is now a
further topological obstruction for such a reduction - apart from the first
Stiefel-Whitney class $w_{1}$ vanishing (which gives orientability), we also
need the second Stiefel-Whitney class $w_{2}$ to vanish \cite{FernandezGray,
FriedrichNPG2}. This is the condition for the manifold to admit a spin
structure. Once we have a $7$-manifold with both $w_{1}$ and $w_{2}$
vanishing, any Riemannian metric will give rise to an $SO\left( 7\right) $%
-structure, and this could then be reduced to a $G_{2}$-structure. By
specifying a $G_{2}$-principal bundle, we are effectively also defining a $%
G_{2}$-invariant $3$-form $\varphi ,$ which gives rise to the structure
constants for the vector cross product. A good review of vector cross
product geometries can also be found in \cite{ConanChapter}.

In general, the $3$-form $\varphi $ will not be parallel with respect to the
Levi-Civita connection, and thus the $G_{2}$-structure will have torsion.
The different torsion classes have originally been classified by Fern\'{a}%
ndez and Gray \cite{FernandezGray}. Understanding the existence properties
of different torsion classes is of particular interest to theoretical
physics, because in a compactification of $11$-dimensional $M$-theory to an
observable $4$-dimensional space, it is necessary to use a $7$-dimensional
manifold which will necessarily admit a $G_{2}$-structure. The torsion of
this $G_{2}$-structure will then affect the physical properties of the
theory \cite{KasteMinasianFlux}. Of even greater interest, both in
mathematics and physics, are \emph{torsion-free }$G_{2}$-structures. A
torsion-free $G_{2}$-structure then corresponds to a Riemannian metric with
a reduced holonomy group. In particular, the holonomy group would have to be
a subgroup of $G_{2}.$ In even dimensions, thanks to Yau's Theorem \cite%
{CalabiYau}, we have necessary and sufficient conditions for the existence
of $SU\left( n\right) $ holonomy metrics - the Calabi-Yau metrics. For $%
G_{2} $ holonomy manifolds currently there is even no conjecture as to what
the conditions could be.

Due to the close relationship between $G_{2}$ and octonions, it is natural
to introduce an octonionic structure on a $7$-manifold with a $G_{2}$%
-structure. The aim of this paper is to develop the properties of an
octonion bundle on a $7$-manifold with $G_{2}$-structure. A number of
properties of $G_{2}$-structures are re-expressed in a very natural form
using the octonion formalism, and we believe that further progress in the
study of $G_{2}$-structures could be made using this approach.

In Section \ref{secg2struct} we give a brief introduction to $G_{2}$%
-structures and some of their basic properties. More detailed accounts of
properties of $G_{2}$-structures can be found in \cite%
{bryant-2003,GrigorianG2Review,GrigorianG2Torsion1,karigiannis-2005-57,karigiannis-2007}%
. In Section \ref{secOctobundle} we then introduce the octonion bundle -
which is a rank $8$ bundle with fibers $\mathbb{R}\oplus T_{p}M$. The scalar
part corresponds to the real part of an octonion and the vector part
corresponds to the imaginary part. The $G_{2}$-structure is then used to
define a fiberwise nonassociative normed division algebra. We then also give
some properties of the associator that will be used later. The subbundle of
unit octonion sections then has a fiberwise \emph{Moufang loop }structure on
it - this is a nonassociative analog of a group, but with associativity
replaced by weaker properties.

It is well-known that given a Riemannian metric $g$ on a $7$-manifold $M$
that admits $G_{2}$-structures, there is a family of compatible $G_{2}$%
-structures. Pointwise, such a family is parametrized by $SO\left( 7\right)
/G_{2}\cong \mathbb{R}P^{7}.$ In particular, given a $G_{2}$-structure $3$%
-form $\varphi ,$ any unit section $A$ of the octonion bundle can be used to
define a new $G_{2}$-structure $\sigma _{A}\left( \varphi \right) $ where $%
\sigma _{A}:\Omega ^{3}\left( M\right) \longrightarrow \Omega ^{3}\left(
M\right) $ is a map of $3$-forms that is quadratic in $A.$ The original
expression for $\sigma _{A}\left( \varphi \right) $ is due to Bryant \cite%
{bryant-2003}. All the $G_{2}$-structures that correspond to $g$ are then of
the form $\sigma _{A}\left( \varphi \right) $ for some unit octonion section 
$A$ and are called \emph{isometric }$G_{2}$-structures. In Section \ref%
{secIsomG2} we focus on the properties of isometric $G_{2}$-structures and
the map $\sigma _{A}$. In particular, we show that the action of $\sigma
_{A} $ on $\varphi $ corresponds to the action of the adjoint map $\func{Ad}%
_{A^{-\frac{1}{3}}}$ on the octonions (Theorem \ref{ThmSigmaAd}) and we use
that to show that $\sigma $ is compatible with octonion multiplication -
that is, $\sigma _{A}\left( \sigma _{B}\left( \varphi \right) \right)
=\sigma _{AB}\left( \varphi \right) $ where $AB$ is the octonion product of $%
A$ and $B$ with respect to $\varphi $ (Theorem \ref{ThmSigmaUV}). This gives
also a representation of the unit octonions on $3$-forms.

In Section \ref{secTorsion}, we give a brief overview of the $G_{2}$%
-structure torsion and in\ Section \ref{secOctoCovDiv}, we use the torsion $T
$ and octonion multiplication to introduce an octonionic covariant
derivative $D$ given 
\begin{equation}
D_{X}V=\nabla _{X}V-VT_{X}
\end{equation}%
where $T_{X}=X\lrcorner T$ is interpreted as an imaginary octonion section,
and $VT_{X}$ is the octonion product of $T_{X}$ and $V.$ This is then shown
to be partially compatible with octonion multiplication, that is given two
octonion sections $U\ $and $V,$ 
\begin{equation}
D_{X}\left( UV\right) =\left( \nabla _{X}U\right) V+U\left( D_{X}V\right) 
\end{equation}%
We then prove that it is moreover metric compatible. The $G_{2}$-structure
torsion then is interpreted as an $\func{Im}\mathbb{O}$-\emph{valued }$1$%
-form. In our case, the tangent bundle to the unit octonion Moufang loop is
precisely the space of imaginary octonions, so this is the exact analog of a
\textquotedblleft Lie algebra-valued $1$-form\textquotedblright\ that
represents a connection on a principal bundle. Note that while the idea of
constructing $G_{2}$-compatible connections using the torsion has been used
in the past, such as in \cite%
{AgricolaSrni,AgricolaSpinors,AgricolaFriedrich1}, the idea to interpret the 
$G_{2}$ torsion as a connection $1$-form on a nonassociative bundle is new
and is the key point in this paper. The curvature of this connection has a
standard part that comes from Levi-Civita connection and a part that comes
from the octonion structure. We prove that the octonion part of the
curvature of this octonionic connection is $\frac{1}{4}\pi _{7}\func{Riem}$,
which is the component of the Riemann curvature that lies in the $7$%
-dimensional representation of $G_{2}.$ It is well known \cite%
{BonanG2,karigiannis-2007} that the vanishing of $\pi _{7}\func{Riem}$ is a
necessary condition for a torsion-free $G_{2}$-structure, and now we have a
new interpretation of this quantity as an $\func{Im}\mathbb{O}$-valued $2$%
-form that is the octonionic exterior covariant derivative of the torsion $1$%
-form.

The octonion covariant derivative $D$ is defined with respect to a fixed $%
G_{2}$-structure $\varphi .$ However, we know that we have a choice of
isometric $G_{2}$-structures given by $\sigma _{V}\left( \varphi \right) $
for any unit octonion section $V.$ In Section \ref{secDeform} we consider
how $D$ is affected by a change of the $G_{2}$-structure within the metric
class. To do this, we first prove in Theorem \ref{ThmTorsV} how the torsion $%
1$-form $T^{\left( V\right) }$ for the $G_{2}$-structure $\sigma _{V}\left(
\varphi \right) $ is related to the original torsion $T.$ In turns out that 
\begin{equation}
T^{\left( V\right) }=VTV^{-1}+V\left( \nabla V^{-1}\right) =-\left(
DV\right) V^{-1}.  \label{torsdeform}
\end{equation}%
This relationship further reinforces the idea that the correct way of
thinking of $G_{2}$ torsion is to regard it as an octonionic connection $1$%
-form, since the expression (\ref{torsdeform}) is very similar to the
transformation of a principal bundle connection $1$-form under a change of
trivialization, i.e. a change of gauge. Moreover, this shows that the choice
of the particular $G_{2}$-structure $3$-form within the metric class
corresponds to picking a gauge. Finally, using (\ref{torsdeform}), we
conclude that $D$ is indeed \emph{covariant} with respect to a change of $%
G_{2}$-structure within the same metric class. In particular, if $\tilde{D}$
is the covariant derivative with respect to $\sigma _{V}\left( \varphi
\right) ,$ then 
\begin{equation}
\tilde{D}\left( AV^{-1}\right) =\left( DA\right) V^{-1}.
\end{equation}%
This is then used to note that the metric class contains a torsion-free $%
G_{2}$-structure if and only if there exists a nowhere-vanishing octonion
section that is parallel with respect to $D.$ This condition is independent
of the initial choice of the $G_{2}$-structure.

In turns out that much of the structure of the octonion bundle mirrors that
of the spinor bundle on a $7$-manifold. In Section \ref{secSpinor} we make
this relationship precise. It is well known that a $G_{2}$-structure may be
defined by a unit spinor on the manifold. Under the correspondence between
the spinor bundle and the octonion bundle, the fixed spinor is then mapped
to $1.$ A change of the unit spinor then corresponds to a transformation $%
\sigma _{V}$ of the $G_{2}$-structure for some appropriate $V.$ As it is
well known, the enveloping algebra of the octonions, i.e. the algebra of
left multiplication maps by an octonion under composition, is isomorphic to
the Clifford algebra on spinors \cite{HarveyBook}. However, the enveloping
algebra is by definition associative, so the correspondence of the octonion
bundle with the spinor bundle only captures part of the structure of the
octonion bundle. The full non-associative structure of the octonion bundle
cannot be seen in the spinor bundle, therefore it is expected that the
octonion carries more information than the spinor bundle, although the
difference is subtle. In particular, while there is no natural binary
operation on the spinor bundle, we can multiply octonions. In fact, Clifford
multiplication of a vector and a spinor translates to multiplication of two
octonions - therefore we are implicitly using the \emph{triality}
correspondence between vector and spinor representations.

Using the correspondence between spinors and octonions, in\ Section \ref%
{secDirac} we define an octonionic Dirac operator. In the torsion-free case,
an octonionic Dirac operator had been defined by Karigiannis in \cite%
{karigiannis-2006notes}, however this is the generalization for arbitrary $%
G_{2}$-structure torsion. Using octonion algebra we then give a direct proof
of the octonionic Lichnerowicz-Weitzenb\"{o}ck formula. It also follows that
the Dirac operator of an octonion section $V$ determines the $1$-dimensional
and $7$-dimensional components of the torsion of the corresponding $G_{2}$%
-structure $\sigma _{V}\left( \varphi \right) $ (this has also been recently
shown directly for spinors in \cite{AgricolaSpinors}). It is then a simple
observation to note that on compact manifold, a $G_{2}$-structure that
corresponds to a metric with vanishing total scalar curvature is
torsion-free if and only if the $1$-dimensional and $7$-dimensional
components of the torsion vanish.

The interpretation of the $G_{2}$-structure torsion as a connection $1$-form
and the choice of $G_{2}$-structure within a fixed metric class as a choice
of gauge suggests that there needs to be a way to select the
\textquotedblleft best\textquotedblright\ gauge. A natural way of doing this
is to consider critical points of a functional. On a compact manifold, a
reasonable functional to consider is the $L_{2}$-norm of the torsion tensor,
now regarded as a functional $\mathcal{E}\left( V\right) $ on the space of
unit octonion sections, so that the metric remains fixed:%
\begin{equation}
\mathcal{E}\left( V\right) =\frac{1}{2}\int_{M}\left\vert T^{\left( V\right)
}\right\vert ^{2}\func{vol}=\frac{1}{2}\int_{M}\left\vert DV\right\vert ^{2}%
\func{vol}.  \label{introEfunc}
\end{equation}%
We consider the basic properties of this functional in Section \ref%
{secEnergy}. Such a functional, but without the constraint on the metric,
has been been considered from different points of view by Weiss and Witt 
\cite{WeissWitt2,WeissWitt1} - as a functional on $3$-forms, and by Ammann,
Weiss and Witt \cite{AmmannWeissWitt1} - as a functional on spinors. In our
case, this reduces to an energy functional on unit octonion sections. The
equation for a critical point is then shown to be%
\begin{equation}
D^{\ast }DV-\left\vert DV\right\vert ^{2}V=0.
\end{equation}%
This is very similar to the equation for \emph{harmonic unit vector fields} 
\cite{GilMedranoVF1,WiegminkBending,WoodUnitVF}, which is one of the
equations that a vector field must satisfy in order to be a \emph{harmonic
map }from a manifold to the unit tangent bundle. In terms of $G_{2}$%
-structures, critical points of (\ref{introEfunc}) then are shown to
correspond to divergence-free torsion, i.e. a \emph{Coulomb gauge}. This is
not surprising, since in gauge theory, the Coulomb gauge corresponds
precisely to critical points of the $L_{2}$-norm of the connection \cite%
{DonaldsonGauge,UhlenbeckConnection}. This further reinforces the point of
view that $G_{2}$-structure torsion is a connection $1$-form for a
non-associative gauge theory.

\begin{description}
\item[Conventions] In this paper we will be using the following convention
for Ricci and Riemann curvature:%
\begin{equation}
\func{Ric}_{jl}=g^{ij}\func{Riem}_{ijkl}  \label{riemconvention}
\end{equation}%
Also, the convention for the orientation of a $G_{2}$-structure will same as
the one adopted by Bryant \cite{bryant-2003} and follows the author's
previous papers. In particular, this causes $\psi =\ast \varphi $ to have an
opposite sign compared to the works of Karigiannis, so many identities and
definitions cited from \cite%
{karigiannis-2005-57,karigiannis-2006notes,karigiannis-2007} may have
differing signs.

\item[Acknowledgments] The author would like to thank the anonymous referees
for the very detailed comments that really helped to improve this paper.
\end{description}

\section{$G_{2}$-structures}

\setcounter{equation}{0}\label{secg2struct}The 14-dimensional group $G_{2}$
is the smallest of the five exceptional Lie groups and is closely related to
the octonions, which is the noncommutative, nonassociative, $8$-dimensional
normed division algebra. In particular, $G_{2}$ can be defined as the
automorphism group of the octonion algebra. Given the octonion algebra $%
\mathbb{O}$, there exists a unique orthogonal decomposition into a real
part, that is isomorphic to $\mathbb{R},$ and an \emph{imaginary }(or \emph{%
pure}) part, that is isomorphic to $\mathbb{R}^{7}$:%
\begin{equation}
\mathbb{O}\cong \mathbb{R}\oplus \mathbb{R}^{7}
\end{equation}%
Correspondingly, given an octonion $a\in \mathbb{O}$, we can uniquely write 
\begin{equation*}
a=\func{Re}a+\func{Im}a
\end{equation*}%
where $\func{Re}a\in \mathbb{R}$, and $\func{Im}a\in \mathbb{R}^{7}$. We can
now use octonion multiplication to define a vector cross product $\times $
on $\mathbb{R}^{7}$. Given $u,v\in \mathbb{R}^{7}$, we regard them as
octonions in $\func{Im}\mathbb{O}$, multiply them together using octonion
multiplication, and then project the result to $\func{Im}\mathbb{O}$ to
obtain a new vector in $\mathbb{R}^{7}$:%
\begin{equation}
u\times v=\func{Im}\left( uv\right) .  \label{octovp}
\end{equation}
The subgroup of $GL\left( 7,\mathbb{R}\right) $ that preserves this vector
cross product is then precisely the group $G_{2}$. A detailed account of the
properties of the octonions and their relationship to exceptional Lie groups
is given by John Baez in \cite{BaezOcto}. The structure constants of the
vector cross product define a $3$-form on $\mathbb{R}^{7}$, hence $G_{2}$ is
alternatively defined as the subgroup of $GL\left( 7,\mathbb{R}\right) $
that preserves a particular $3$-form $\varphi _{0}$ \cite{Joycebook}.

\begin{definition}
Let $\left( e^{1},e^{2},...,e^{7}\right) $ be the standard basis for $\left( 
\mathbb{R}^{7}\right) ^{\ast }$, and denote $e^{i}\wedge e^{j}\wedge e^{k}$
by $e^{ijk} $. Then define $\varphi _{0}$ to be the $3$-form on $\mathbb{R}%
^{7}$ given by 
\begin{equation}
\varphi _{0}=e^{123}+e^{145}+e^{167}+e^{246}-e^{257}-e^{347}-e^{356}.
\label{phi0def}
\end{equation}%
Then $G_{2}$ is defined as the subgroup of $GL\left( 7,\mathbb{R}\right) $
that preserves $\varphi _{0}$.
\end{definition}

In general, given a $n$-dimensional manifold $M$, a $G$-structure on $M$ for
some Lie subgroup $G$ of $GL\left( n,\mathbb{R}\right) $ is a reduction of
the frame bundle $F$ over $M$ to a principal subbundle $P$ with fibre $G$. A 
$G_{2}$-structure is then a reduction of the frame bundle on a $7$%
-dimensional manifold $M$ to a $G_{2}$-principal subbundle. The obstructions
for the existence of a $G_{2}$-structure are purely topological.

\begin{theorem}[\protect\cite{FernandezGray, FriedrichNPG2, Gray-VCP}]
\label{ThmG2exist}Let $M$ be a $7$-dimensional smooth manifold. Then, $M$
admits a $G_{2}$-structure if and only if the Stiefel-Whitney classes $w_{1}$
and $w_{2}$ both vanish.
\end{theorem}

Thus, given a $7$-dimensional smooth manifold that is both orientable ($%
w_{1}=0$) and admits a spin structure ($w_{2}=0$), there always exists a $%
G_{2}$-structure on it.

It turns out that there is a $1$-$1$ correspondence between $G_{2}$%
-structures on a $7$-manifold and smooth $3$-forms $\varphi $ for which the $%
7$-form-valued bilinear form $B_{\varphi }$ as defined by (\ref{Bphi}) is
positive definite (for more details, see \cite{Bryant-1987} and the arXiv
version of \cite{Hitchin:2000jd}). 
\begin{equation}
B_{\varphi }\left( u,v\right) =\frac{1}{6}\left( u\lrcorner \varphi \right)
\wedge \left( v\lrcorner \varphi \right) \wedge \varphi  \label{Bphi}
\end{equation}%
Here the symbol $\lrcorner $ denotes contraction of a vector with the
differential form: 
\begin{equation}
\left( u\lrcorner \varphi \right) _{mn}=u^{a}\varphi _{amn}.
\label{contractdef}
\end{equation}%
Note that we will also use this symbol for contractions of differential
forms using the metric.

A smooth $3$-form $\varphi $ is said to be \emph{positive }if $B_{\varphi }$
is the tensor product of a positive-definite bilinear form and a
nowhere-vanishing $7$-form. In this case, it defines a unique Riemannian
metric $g_{\varphi }$ and volume form $\mathrm{vol}_{\varphi }$ such that
for vectors $u$ and $v$, the following holds 
\begin{equation}
g_{\varphi }\left( u,v\right) \mathrm{vol}_{\varphi }=\frac{1}{6}\left(
u\lrcorner \varphi \right) \wedge \left( v\lrcorner \varphi \right) \wedge
\varphi  \label{gphi}
\end{equation}%
An equivalent way of defining a positive $3$-form $\varphi $, is to say that
at every point, $\varphi $ is in the $GL\left( 7,\mathbb{R}\right) $-orbit
of $\varphi _{0}$. It can be easily checked that the metric (\ref{gphi}) for 
$\varphi =\varphi _{0}$ is in fact precisely the standard Euclidean metric $%
g_{0}$ on $\mathbb{R}^{7}$. Therefore, every $\varphi $ that is in the $%
GL\left( 7,\mathbb{R}\right) $-orbit of $\varphi _{0}$ has an \emph{%
associated} Riemannian metric $g$, that is in the $GL\left( 7,\mathbb{R}%
\right) $-orbit of $g_{0}.$ The only difference is that the stabilizer of $%
g_{0}$ (along with orientation) in this orbit is the group $SO\left(
7\right) $, whereas the stabilizer of $\varphi _{0}$ is $G_{2}\subset
SO\left( 7\right) $. This shows that positive $3$-forms forms that
correspond to the same metric, i.e., are \emph{isometric}, are parametrized
by $SO\left( 7\right) /G_{2}\cong \mathbb{RP}^{7}\cong S^{7}/\mathbb{Z}_{2}$%
. Therefore, on a Riemannian manifold, metric-compatible $G_{2}$-structures
are parametrized by sections of an $\mathbb{RP}^{7}$-bundle, or
alternatively, by sections of an $S^{7}$-bundle, with antipodal points
identified. In fact, the precise parametrization of isometric $G_{2}$%
-structures is well-known.

\begin{theorem}[\protect\cite{bryant-2003}]
\label{ThmSamegfam}Let $M$ be a $7$-dimensional smooth manifold. Suppose $%
\varphi $ is a positive $3$-form on $M$ with associated Riemannian metric $g$%
. Then, any positive $3$-form $\tilde{\varphi}$ for which $g$ is also the
associated metric, is given by the following expression:%
\begin{equation}
\tilde{\varphi}=\left( a^{2}-\left\vert \alpha \right\vert ^{2}\right)
\varphi -2a\alpha \lrcorner \left( \ast \varphi \right) +2\alpha \wedge
\left( \alpha \lrcorner \varphi \right)  \label{phisameg1}
\end{equation}%
where $a$ is a scalar function on $M$ and $\alpha $ is a vector field such
that 
\begin{equation}
a^{2}+\left\vert \alpha \right\vert ^{2}=1  \label{asqvsq}
\end{equation}
\end{theorem}

Note that the relation (\ref{asqvsq}) shows that indeed $\left( a,\alpha
\right) \in S^{7}$, and moreover, in the expression (\ref{phisameg1}),
simultaneously changing the sign of $a$ and $\alpha $ keeps $\tilde{\varphi}$
unchanged. The pair $\left( a,\alpha \right) $ can in fact be also
interpreted as a \emph{unit octonion} section, where $a$ is the real part,
and $\alpha $ is the imaginary part. It turns out that this is a natural
formalism in which to study isometric $G_{2}$-structures, and the main aim
of this paper is to develop this point of view.

An alternative way of studying $G_{2}$-structures is in terms of spinors. A
detailed account of the relationship between $Spin$ groups and $G_{2}$ can
be found in \cite{BaezOcto}, while explicit descriptions of different $G_{2}$%
-structures in terms of spinors can be found in \cite{AgricolaSpinors,
BaumFriedrichSpinors}, among others. This approach also makes a connection
with other types of $G$-structures on manifolds, with $SU\left( 3\right) $%
-structures on $6$-manifolds being of particular interest. Also, since in
this approach, a $G_{2}$-structure is defined by a single nowhere vanishing
spinor, it is also of relevance in theoretical physics, e.g. \cite%
{KasteMinasianFlux}. Now, given a $7$-dimensional manifold $M$ with $%
w_{1}=w_{2}=0$, that is, one that satisfies the conditions for the existence
of $G_{2}$-structures from Theorem \ref{ThmG2exist}, a Riemannian metric
will define a spinor bundle which will necessarily admit a nowhere vanishing
section. Any such spinor section will then define a positive $3$-form, and
hence a compatible $G_{2}$-structure on $M$. Moreover, spinor sections
within the same projective class define the same $G_{2}$-structure.
Therefore, metric-compatible $G_{2}$-structures are parametrized by sections
of the induced $\mathbb{R}P^{7}$-bundle. Thus, also using Theorem \ref%
{ThmSamegfam}, we have the following corollary of Theorem \ref{ThmG2exist}.

\begin{corollary}
Suppose $M$ is a smooth $7$-dimensional manifold that satisfies $w_{1}=0$
and $w_{2}=0$. Given any Riemannian metric $g$ on $M$, there exists a family
of $G_{2}$-structures for which $g$ is the associated metric.
\end{corollary}

\begin{definition}
If two $G_{2}$-structures $\varphi _{1}$ and $\varphi _{2}$ on $M$ have the
same associated metric $g$, we say that $\varphi _{1}$ and $\varphi _{2}$
are in the same \emph{metric class}.
\end{definition}

Using this observation we see that the set of all $G_{2}$-structures on $M$
is parametrized by the set of Riemannian metrics on $M$, and then within
each metric class, by projective classes of unit octonion sections $\left(
a,\alpha \right) $. In what follows, we will assume that $M$ is a smooth $7$%
-dimensional manifold with vanishing $w_{1}$ and $w_{2}$, which has a fixed
Riemannian metric $g$. We will then study the family of $G_{2}$-structures
for which $g$ is the associated metric.

A further property of $G_{2}$-structures is that the spaces of $p$-forms
decompose according to irreducible representations of $G_{2}$. Thus, $2$%
-forms split as $\Lambda ^{2}=\Lambda _{7}^{2}\oplus \Lambda _{14}^{2}$,
where 
\begin{equation*}
\Lambda _{7}^{2}=\left\{ \alpha \lrcorner \varphi \text{: for a vector field 
}\alpha \right\}
\end{equation*}%
and 
\begin{equation*}
\Lambda _{14}^{2}=\left\{ \omega \in \Lambda ^{2}\text{: }\left( \omega
_{ab}\right) \in \mathfrak{g}_{2}\right\} =\left\{ \omega \in \Lambda ^{2}%
\text{: }\omega \lrcorner \varphi =0\right\} .
\end{equation*}%
The $3$-forms split as $\Lambda ^{3}=\Lambda _{1}^{3}\oplus \Lambda
_{7}^{3}\oplus \Lambda _{27}^{3}$, where the one-dimensional component
consists of forms proportional to $\varphi $, forms in the $7$-dimensional
component are defined by a vector field $\Lambda _{7}^{3}=\left\{ \alpha
\lrcorner \psi \text{: for a vector field }\alpha \right\} $, and forms in
the $27$-dimensional component are defined by traceless, symmetric matrices: 
\begin{equation}
\Lambda _{27}^{3}=\left\{ \chi \in \Lambda ^{3}:\chi _{abc}=\mathrm{i}%
_{\varphi }\left( h\right) =h_{[a}^{d}\varphi _{bc]d}^{{}}\text{ for }h_{ab}~%
\text{traceless, symmetric}\right\} .  \label{lam327}
\end{equation}%
By Hodge duality, similar decompositions exist for $\Lambda ^{4}$ and $%
\Lambda ^{5}$. Further details of these decompositions can be found in \cite%
{bryant-2003,GrigorianG2Torsion1,GrigorianYau1,karigiannis-2005-57}.

\section{Octonion bundle}

\setcounter{equation}{0}\label{secOctobundle}Let $\left( M,g\right) $ be a
smooth $7$-dimensional Riemannian manifold, with $w_{1}=w_{2}=0.$ We know $M$
admits $G_{2}$-structures. In particular, let $\varphi $ be a $G_{2}$%
-structure for $M$ for which $g$ is the associated metric. We also use $g$
to define the Levi-Civita connection $\nabla $, and the Hodge star $\ast $.
In particular, $\ast \varphi $ is a $4$-form dual to $\varphi ,$ which we
will denote by $\psi $.

\begin{definition}
The \emph{octonion bundle }$\mathbb{O}M$ on $M$ is the rank $8$ real vector
bundle given by 
\begin{equation}
\mathbb{O}M\cong \Lambda ^{0}\oplus TM  \label{OMdef}
\end{equation}%
where $\Lambda ^{0}\cong M\times \mathbb{R}\ $is a trivial line bundle. At
each point $p\in M$, $\mathbb{O}_{p}M\cong \mathbb{R}\oplus T_{p}M.$
\end{definition}

The definition (\ref{OMdef}) simply mimics the decomposition of octonions
into real and imaginary parts. The bundle $\mathbb{O}M$ is defined as a real
bundle, but which will have additional structure as discussed below. Now let 
$A\in \Gamma \left( \mathbb{O}M\right) $ be a section of the octonion
bundle. We will call $A$ simply an \emph{octonion} \emph{on }$M.$ From (\ref%
{OMdef}), $A$ has a scalar component in $\Gamma \left( \Lambda ^{0}\right) $%
, i.e. just a function on $M$, as well as a vector component in $\Gamma
\left( TM\right) $, i.e. a vector field on $M$. We then have globally
defined projections%
\begin{eqnarray*}
\func{Re} &:&\Gamma \left( \mathbb{O}M\right) \longrightarrow \Gamma \left(
\Lambda ^{0}\right) \\
\func{Im} &:&\Gamma \left( \mathbb{O}M\right) \longrightarrow \Gamma \left(
TM\right) .
\end{eqnarray*}%
Therefore, we write $A=\func{Re}A+\func{Im}A.$ For convenience, we may also
write $A=\left( \func{Re}A,\func{Im}A\right) $ or $A=\left( 
\begin{array}{c}
\func{Re}A \\ 
\func{Im}A%
\end{array}%
\right) $. We also have a natural involution on $\mathbb{O}M$ - \emph{%
conjugation}. As for complex or quaternionic sections, define 
\begin{equation}
\bar{A}=\left( \func{Re}A,-\func{Im}A\right) .  \label{octoconj}
\end{equation}

Since $\mathbb{O}M$ is defined as a tensor bundle, the Riemannian metric $g$
on $M$ induces a metric on $\mathbb{O}M.$ Let $A=\left( a,\alpha \right) \in
\Gamma \left( \mathbb{O}M\right) .$ Then, 
\begin{eqnarray}
\left\vert A\right\vert ^{2} &=&\left\langle A,A\right\rangle =a^{2}+g\left(
\alpha ,\alpha \right)  \notag \\
&=&a^{2}+\left\vert \alpha \right\vert ^{2}  \label{Omet}
\end{eqnarray}%
We will be using the same notation for the norm, metric and inner product
for sections of $\mathbb{O}M$ as for standard tensors on $M$. It will be
clear from the context which is being used. If however we need to specify
that only the octonion inner product is used, we will use the notation $%
\left\langle \cdot ,\cdot \right\rangle _{\mathbb{O}}.$ The metric (\ref%
{Omet}) ensures that the real and imaginary parts are orthogonal to each
other.

\begin{definition}
Given the $G_{2}$-structure $\varphi $ on $M,$ we define a \emph{vector
cross product with respect to }$\varphi $ on $M.$ Let $\alpha $ and $\beta $
be two vector fields, then define%
\begin{equation}
\left\langle \alpha \times _{\varphi }\beta ,\gamma \right\rangle =\varphi
\left( \alpha ,\beta ,\gamma \right)  \label{vcrossdef}
\end{equation}%
for any vector field $\gamma $ \cite{Gray-VCP,karigiannis-2005-57}.
\end{definition}

In index notation, we can thus write 
\begin{equation}
\left( \alpha \times _{\varphi }\beta \right) ^{a}=\varphi _{\ bc}^{a}\alpha
^{b}\beta ^{c}  \label{vcrossdef2}
\end{equation}%
Note that $\alpha \times _{\varphi }\beta =-\beta \times _{\varphi }\alpha $%
. If there is no ambiguity as to which $G_{2}$-structure is being used to
define the cross product, we will simply denote it by $\times $, dropping
the subscript.

Using the contraction identity for $\varphi $ \cite%
{bryant-2003,GrigorianG2Review, karigiannis-2005-57}%
\begin{equation}
\varphi _{abc}\varphi _{mn}^{\ \ \ \ c}=g_{am}g_{bn}-g_{an}g_{bm}+\psi
_{abmn}  \label{phiphi1}
\end{equation}%
we obtain the following identity for the double cross product.

\begin{lemma}
Let $\alpha ,\beta ,\gamma $ be vector fields, then 
\begin{equation}
\alpha \times \left( \beta \times \gamma \right) =\left\langle \alpha
,\gamma \right\rangle \beta -\left\langle \alpha ,\beta \right\rangle \gamma
+\psi \left( \cdot ^{\sharp },\alpha ,\beta ,\gamma \right)
\label{doublecrossprod}
\end{equation}%
where $^{\sharp }$ means that we raise the index using the inverse metric $%
g^{-1}$.
\end{lemma}

Using the inner product and the cross product, we can now define the \emph{%
octonion product }on $\mathbb{O}M$.

\begin{definition}
Let $A,B\in \Gamma \left( \mathbb{O}M\right) .$ Suppose $A=\left( a,\alpha
\right) $ and $B=\left( b,\beta \right) $. Given the vector cross product (%
\ref{vcrossdef}) on $M,$ we define the \emph{octonion product }$A\circ
_{\varphi }B$\emph{\ with respect to }$\varphi $\emph{\ }as follows:%
\begin{equation}
A\circ _{\varphi }B=\left( 
\begin{array}{c}
ab-\left\langle \alpha ,\beta \right\rangle \\ 
a\beta +b\alpha +\alpha \times _{\varphi }\beta%
\end{array}%
\right)  \label{octoproddef}
\end{equation}
\end{definition}

If there is no ambiguity as to which $G_{2}$-structure is being used to
define the octonion product, for convenience, we will simply write$\ AB$ to
denote it. The octonion product behaves as expected with respect to
conjugation:

\begin{lemma}
Suppose $A$ and $B$ are sections of $\mathbb{O}M$, then 
\begin{subequations}%
\begin{eqnarray}
\overline{AB} &=&\overline{B}\overline{A} \\
A\overline{A} &=&\overline{A}A=\left\vert A\right\vert ^{2}
\end{eqnarray}%
\end{subequations}%
\end{lemma}

\begin{lemma}
Using the octonion product (\ref{octoproddef}), the inner product $%
\left\langle A,B\right\rangle $ of two octonions $A=\left( a,\alpha \right) $
and $B=\left( b,\beta \right) $ is given by%
\begin{equation}
\left\langle A,B\right\rangle =\frac{1}{2}\left( A\bar{B}+B\bar{A}\right) .
\label{octoinnerprod}
\end{equation}%
The \emph{commutator }$\left[ A,B\right] $ of $A$ and $B$ is given by 
\begin{eqnarray}
\left[ A,B\right] &=&AB-BA  \notag \\
&=&2\alpha \times \beta  \notag \\
&=&2\varphi \left( \cdot ^{\sharp },\alpha ,\beta \right)  \label{octocomm}
\end{eqnarray}%
The \emph{associator }$\left[ A,B,C\right] $ of three octonions $A,B$ and $%
C=\left( c,\gamma \right) $ is given by 
\begin{eqnarray}
\left[ A,B,C\right] &=&A\left( BC\right) -\left( AB\right) C  \notag \\
&=&2\psi \left( \cdot ^{\sharp },\alpha ,\beta ,\gamma \right)
\label{octoassoc}
\end{eqnarray}
\end{lemma}

\begin{proof}
The identities for the inner product and the commutator follow immediately
from (\ref{octoproddef}). The identity for (\ref{octoassoc}) follows from
the double cross product identity (\ref{doublecrossprod}).
\end{proof}

Crucially, (\ref{octoassoc}) shows that the associator is skew-symmetric.
This property of the octonion algebra is known as \emph{alternativity}. Note
that the associator of a non-associative algebra is usually defined with the
opposite sign to what we have, however due to the relation with $\psi $ it
is more convenient this way. Thus, given a $G_{2}$-structure on $M$, we can
fully transfer the octonion algebra structure to $\mathbb{O}M.$ The
expressions (\ref{octoproddef})-(\ref{octoassoc}) are exactly the same as
for the standard octonion algebra, as given, for example, in \cite{BaezOcto}%
. Therefore, $\mathbb{O}M$ is a bundle that carries a non-associative
division algebra structure on it.

We will need a few identities that the octonions satisfy. The proofs are
straightforward computations and are given in Appendix \ref{appProofs}.

\begin{lemma}
\label{lemOctoexp}Let $A=\left( 0,\alpha \right) \in \Gamma \left( \func{Im}%
\mathbb{O}M\right) .$ Then the exponential of $A,$ $e^{A}=\sum_{k=0}^{\infty
}\frac{1}{k!}A^{k}$, is given by%
\begin{equation}
e^{A}=\cos \left\vert \alpha \right\vert +\alpha \frac{\sin \left\vert
\alpha \right\vert }{\left\vert \alpha \right\vert }  \label{octoexp}
\end{equation}
\end{lemma}

\begin{corollary}
\label{corrOctopow}Suppose $B=\left( b,\beta \right) \in \Gamma \left( 
\mathbb{O}M\right) $, then for any $k\in \mathbb{Z}$ (assuming nowhere
vanishing $B$ if $k$ is negative) 
\begin{equation}
B^{k}=\left\vert B\right\vert ^{k}\left( \cos k\theta +\hat{\beta}\frac{\sin
k\theta }{\sin \theta }\right)  \label{octopower}
\end{equation}%
where $\hat{\beta}=\frac{\beta }{\left\vert B\right\vert }$ and $\theta \in 
\mathbb{R}$ is such that $\cos \theta =\frac{b}{\left\vert B\right\vert }$
and $\sin \theta =\left\vert \hat{\beta}\right\vert $.
\end{corollary}

The above Lemma \ref{lemOctoexp} and Corollary \ref{corrOctopow} are direct
analogs of similar well-known results for complex numbers. It is a useful
fact that the for any integer power $k,$ $\func{Im}B^{k}$ is a real multiple
of $\func{Im}B$. Using the alternative property of the associator, this also
shows that the octonions are \emph{power-associative, }and in fact, any
subalgebra generated by two elements and their conjugates is also
associative. In Lemma \ref{lemAssocIds} we collect some related identities.

\begin{lemma}
\label{lemAssocIds}For any $A,B,C\in \Gamma \left( \mathbb{O}M\right) $, and 
$k\in \mathbb{Z}^{+},$ the following identities hold

\begin{enumerate}
\item $\left[ A,B,C\right] =-\left[ \bar{A},B,C\right] $

\item $\left[ A^{k},A,C\right] =0$

\item $A\left[ A,B,C\right] =\left[ A,B,C\right] \bar{A}$

\item $\left[ A,A^{k}B,C\right] =\bar{A}^{k}\left[ A,B,C\right] $

\item $\left[ A,BA^{k},C\right] =\left[ A,B,C\right] \bar{A}^{k}$

\item $\left[ A^{k+1},B,C\right] =\left[ A^{k},B,C\right] \bar{A}+\left[
A,B,C\right] A^{k}$
\end{enumerate}
\end{lemma}

For $k=1$ and $k=2$, the last identity in Lemma \ref{lemAssocIds} gives us
important special cases.

\begin{corollary}
\label{CorrAssocId2}In particular,

\begin{enumerate}
\item $\left[ A^{2},B,C\right] =\left[ A,B,C\right] \left( A+\bar{A}\right) $

\item $\left[ A^{3},B,C\right] =\left[ A,B,C\right] \left( \bar{A}%
^{2}+\left\vert A\right\vert ^{2}+A^{2}\right) $
\end{enumerate}
\end{corollary}

For a given octonion $B\in \Gamma \left( \mathbb{O}M\right) $ we may define
the right translation map $R_{B}:\Gamma \left( \mathbb{O}M\right)
\longrightarrow \Gamma \left( \mathbb{O}M\right) $ and the left translation
map $L_{B}:\Gamma \left( \mathbb{O}M\right) \longrightarrow \Gamma \left( 
\mathbb{O}M\right) $ by%
\begin{eqnarray*}
R_{B}A &=&AB \\
L_{B}A &=&BA
\end{eqnarray*}%
Whenever $B\neq 0,$ these are invertible maps since $R_{B}R_{B^{-1}}A=\left(
AB^{-1}\right) B=A\left( B^{-1}B\right) =A$ and similarly for $L$. As
expected, the conjugates of $R_{B}$ and $L_{B}$ with respect to the octonion
metric are given by $R_{\bar{B}}$ and $L_{\bar{B}}$ respectively.

\begin{lemma}
\label{lemRLmult}For any octonions $A,B,C\in \Gamma \left( \mathbb{O}%
M\right) $, we have 
\begin{eqnarray}
\left\langle R_{B}A,C\right\rangle &=&\left\langle A,R_{\bar{B}%
}C\right\rangle \\
\left\langle L_{B}A,C\right\rangle &=&\left\langle A,L_{\bar{B}%
}C\right\rangle
\end{eqnarray}
\end{lemma}

From this we can see that 
\begin{equation*}
\left\langle R_{B}A,R_{B}C\right\rangle =\left\vert B\right\vert
^{2}\left\langle A,C\right\rangle
\end{equation*}%
and similarly for the left translation map. Therefore, whenever $B\neq 0$,
every $R_{B}$ and $L_{B}$ is an element of the conformal group of $\mathbb{R}%
^{8}$. Moreover, when $\left\vert B\right\vert =1$, $R_{B}$ and $L_{B}$
preserve the octonion metric, and are thus elements of the $O\left( 8\right) 
$ group. Due to the nonassociativity of octonion multiplication, neither the
left nor the right translation maps form subgroups of $O\left( 8\right) .$
In fact, the left and right maps in general do not commute, and the
associator can be thought of as the commutator of the right and left
translations:%
\begin{equation}
\left[ L_{A},R_{C}\right] B=A\left( BC\right) -\left( AB\right) C=\left[
A,B,C\right]  \label{LRcommassoc}
\end{equation}

Since multiplication by unit octonions preserves the norm, we can restrict
the octonion multiplication to unit octonions. Hence we define the subbundle
of \emph{unit octonions. }

\begin{definition}
Define the \emph{unit octonion bundle }$S\mathbb{O}M$\emph{\ on }$M$ as the
unit sphere subbundle of $\mathbb{O}M$ where at each point $p\in M,$ the
fiber is given by $S\mathbb{O}_{p}M=\left\{ A\in \mathbb{O}_{p}M:\left\vert
A\right\vert =1\right\} $.
\end{definition}

We can restrict octonion multiplication to unit octonions, so the fiber at
each point is the $7$-sphere $S^{7}$ with a non-associative binary operation
defined on it. In fact, the set of unit octonions form a \emph{Moufang loop}
- an algebraic structure with similar properties to a group, except that it
is non-associative. Instead of associativity, we have weaker properties, as
given by Lemma \ref{lemAssocIds}. The bundle $S\mathbb{O}M$ can then be
regarded as a \emph{principal Moufang loop bundle} - analogous to a
principal bundle.

\section{Isometric $G_{2}$-structures}

\setcounter{equation}{0}\label{secIsomG2}Since the octonions are
power-associative we can unambiguously define the \emph{adjoint map.}

\begin{definition}
For any nowhere-vanishing $V\in \Gamma \left( \mathbb{O}M\right) $, define
the map 
\begin{equation*}
\func{Ad}_{V}:\Gamma \left( \mathbb{O}M\right) \longrightarrow \Gamma \left( 
\mathbb{O}M\right)
\end{equation*}%
given by 
\begin{equation}
\func{Ad}_{V}A=VAV^{-1}  \label{defAdmap}
\end{equation}%
for any $A\in \Gamma \left( \mathbb{O}M\right) .$
\end{definition}

The adjoint map satisfies a number of properties. In particular, it is easy
to see that $\left( \func{Ad}_{V}\right) ^{-1}=\func{Ad}_{V^{-1}}$, so it is
invertible. Also, as we show below, $\func{Ad}_{V}$ preserves the octonion
metric. Let $A,B\in \Gamma \left( \mathbb{O}M\right) ,$ then%
\begin{eqnarray*}
\left\langle \func{Ad}_{V}A,\func{Ad}_{V}B\right\rangle &=&\left\langle
VAV^{-1},VBV^{-1}\right\rangle \\
&=&\frac{1}{\left\vert V\right\vert ^{4}}\left\langle VAV,VBV\right\rangle \\
&=&\left\langle A,B\right\rangle
\end{eqnarray*}%
Therefore, $\func{Ad}_{V}\in O\left( 8\right) $. However, $\func{Ad}_{V}$
preserves the multiplicative identity of $\mathbb{O}$, and therefore maps
imaginary octonions to imaginary octonions. It also follows trivially that
for $\alpha \in \func{Im}\Gamma \left( \mathbb{O}M\right) $, $\overline{%
\func{Ad}_{V}\alpha }=\overline{V\alpha V^{-1}}=-\func{Ad}_{V}\alpha $.
Hence it restricts to pure imaginary octonions, and this restriction $\left. 
\func{Ad}_{V}\right\vert _{\func{Im}\mathbb{O}}$ lies in $O\left( 7\right) $%
. Note that for brevity we will sometimes use $\func{Ad}_{V}$ to denote the
restriction $\left. \func{Ad}_{V}\right\vert _{\func{Im}\mathbb{O}}.$ It
will be clear from the context that this is regarded as a map of imaginary
octonions.

Note that 
\begin{equation*}
\func{Ad}_{kV}=\func{Ad}_{V}
\end{equation*}%
for any nowhere-vanishing scalar $k$, so in fact we can always assume that $%
V $ is a unit octonion.

Using the octonion multiplication rules in terms of the $G_{2}$-structure $%
\varphi $ we can write out $\left. \func{Ad}_{V}\right\vert _{\func{Im}%
\mathbb{O}}$ explicitly as an element in $O\left( 7\right) $. Suppose $\beta 
$ is pure imaginary, and let $V=\left( v_{0},v\right) $, then 
\begin{eqnarray}
\func{Ad}_{V}B &=&V\beta V^{-1}  \notag \\
&=&\frac{1}{\left\Vert V\right\Vert ^{2}}\left( v_{0}\mathbf{+}v\right)
\beta \left( v_{0}\mathbf{-}v\right)  \notag \\
&=&\frac{1}{\left\Vert V\right\Vert ^{2}}\left( v_{0}\mathbf{+}v\right)
\left( \left\langle v,\beta \right\rangle +v_{0}\beta +v\times \beta \right)
\notag \\
&=&\frac{1}{\left\Vert V\right\Vert ^{2}}\left( v_{0}^{2}\beta
+2v_{0}v\times \beta +v\left\langle v,\beta \right\rangle +v\times \left(
v\times \beta \right) \right)  \notag \\
&=&\frac{1}{\left\Vert V\right\Vert ^{2}}\left( \left( v_{0}^{2}-\left\vert
v\right\vert ^{2}\right) \beta +2v_{0}v\times \beta +2v\left\langle v,\beta
\right\rangle \right)  \label{AdVexpress}
\end{eqnarray}%
In index notation, this then gives the components of the matrix $\left. 
\func{Ad}_{V}\right\vert _{\func{Im}\mathbb{O}}$:%
\begin{equation}
\left( \left. \func{Ad}_{V}\right\vert _{\func{Im}\mathbb{O}}\right) _{\
b}^{a}=\frac{1}{\left\Vert V\right\Vert ^{2}}\left( \left(
v_{0}^{2}-\left\vert v\right\vert ^{2}\right) \delta _{\ b}^{a}-2v_{0}\left(
v\lrcorner \varphi \right) _{\ b}^{a}+2v^{a}v_{b}\right) .
\label{AdVexpressind}
\end{equation}%
Using this explicit description of $\left. \func{Ad}_{V}\right\vert _{\func{%
Im}\mathbb{O}}$, a computation of the determinant in \emph{Maple} shows that 
$\det \left( \left. \func{Ad}_{V}\right\vert _{\func{Im}\mathbb{O}}\right)
=+1$ for any non-zero $V$. The explicit calculation is somewhat messy.
Therefore, in fact, $\left. \func{Ad}_{V}\right\vert _{\func{Im}\mathbb{O}%
}\in SO\left( 7\right) $. Since $\left. \func{Ad}_{V}\right\vert _{\func{Re}%
\mathbb{O}}=+1,$ we find that $\func{Ad}_{V}\in SO\left( 8\right) $.

Consider also the following identities.

\begin{lemma}
\label{lemAssocIds2}Given a nowhere-vanishing octonion $V$, the following
identities hold for any $A,B\in \Gamma \left( \mathbb{O}M\right) $

\begin{enumerate}
\item $\left( VA\right) \left( BV^{-1}\right) =\func{Ad}_{V}\left( AB\right)
+\left[ A,B,V^{-1}\right] \left( V+\bar{V}\right) $

\item $\left( AV^{-1}\right) \left( VB\right) =AB+\left[ A,B,V^{-1}\right] V$
\end{enumerate}
\end{lemma}

The proof of Lemma \ref{lemAssocIds2} is given in the Appendix. Using these
identities we can now see what happens to the octonion product under the
action of $\func{Ad}$:

\begin{proposition}
\label{PropAdCprod}Given a nowhere-vanishing octonion $V$, the octonion
product is transformed as follows%
\begin{equation}
\left( \func{Ad}_{V}A\right) \left( \func{Ad}_{V}B\right) =\func{Ad}%
_{V}\left( AB\right) +\left[ A,B,V^{-1}\right] \left( V+\bar{V}+\frac{1}{%
\left\vert V\right\vert ^{2}}V^{3}\right)  \label{AdAdId}
\end{equation}%
and in particular, 
\begin{eqnarray}
\func{Ad}_{V^{-1}}\left[ \left( \func{Ad}_{V}A\right) \left( \func{Ad}%
_{V}B\right) \right] &=&AB+\left[ A,B,V^{-3}\right] V^{3}  \label{AdinAdAdid}
\\
&=&\left( AV^{-3}\right) \left( V^{3}B\right)
\end{eqnarray}
\end{proposition}

\begin{proof}
To work out $\left( \func{Ad}_{V}A\right) \left( \func{Ad}_{V}B\right) ,$ we
first use the identity 2 from Lemma \ref{lemAssocIds2} and the identities
from Lemma \ref{lemAssocIds}: 
\begin{eqnarray}
\left( \func{Ad}_{V}A\right) \left( \func{Ad}_{V}B\right) &=&\left(
VAV^{-1}\right) \left( VBV^{-1}\right)  \notag \\
&=&\left( VA\right) \left( BV^{-1}\right) +\left[ VA,BV^{-1},V^{-1}\right] V
\notag \\
&=&\left( VA\right) \left( BV^{-1}\right) +\bar{V}\left[ A,B,V^{-1}\right] 
\frac{V^{2}}{\left\vert V\right\vert ^{2}}  \notag \\
&=&\left( VA\right) \left( BV^{-1}\right) +\left[ A,B,V^{-1}\right] \frac{%
V^{3}}{\left\vert V\right\vert ^{2}}
\end{eqnarray}%
Now we apply identity 1 from Lemma \ref{lemAssocIds2}:%
\begin{eqnarray}
\left( \func{Ad}_{V}A\right) \left( \func{Ad}_{V}B\right) &=&\func{Ad}%
_{V}\left( AB\right) +\left[ A,B,V^{-1}\right] \left( \bar{V}+V\right) +%
\left[ A,B,V^{-1}\right] \frac{V^{3}}{\left\vert V\right\vert ^{2}}  \notag
\\
&=&\func{Ad}_{V}\left( AB\right) +\left[ A,B,V^{-1}\right] \left( \bar{V}+V+%
\frac{V^{3}}{\left\vert V\right\vert ^{2}}\right)
\end{eqnarray}%
For the second part, we just apply $\func{Ad}_{V^{-1}}$ to (\ref{AdAdId})
and then rewrite using the fact that the subalgebra generated by the two
elements $V$ and $\left[ A,B,V^{-1}\right] $ is associative. 
\begin{eqnarray}
\func{Ad}_{V^{-1}}\left[ \left( \func{Ad}_{V}A\right) \left( \func{Ad}%
_{V}B\right) \right] &=&AB+V^{-1}\left( \left[ A,B,V^{-1}\right] \left( \bar{%
V}+V+\frac{V^{3}}{\left\vert V\right\vert ^{2}}\right) \right) V  \notag \\
&=&AB+\left( V^{-1}\left[ A,B,V^{-1}\right] \right) \left( \left( \bar{V}+V+%
\frac{V^{3}}{\left\vert V\right\vert ^{2}}\right) V\right)
\end{eqnarray}%
Applying the identities from Lemma \ref{lemAssocIds}, we get 
\begin{eqnarray}
\func{Ad}_{V^{-1}}\left[ \left( \func{Ad}_{V}A\right) \left( \func{Ad}%
_{V}B\right) \right] &=&AB-\left( \left[ A,B,V\right] V\right) \left( \left( 
\bar{V}+V+\frac{V^{3}}{\left\vert V\right\vert ^{2}}\right) \frac{V}{%
\left\vert V\right\vert ^{4}}\right)  \notag \\
&=&AB-\left[ A,B,V\right] \left( \left( \bar{V}+V+\frac{V^{3}}{\left\vert
V\right\vert ^{2}}\right) \frac{V^{2}}{\left\vert V\right\vert ^{4}}\right) 
\notag \\
&=&AB-\left[ A,B,V\right] \left( \bar{V}^{2}+\left\vert V\right\vert
^{2}+V^{2}\right) \frac{V^{3}}{\left\vert V\right\vert ^{6}}
\end{eqnarray}%
Now we can use Corollary \ref{CorrAssocId2} to simplify the right-hand side
of (\ref{adinadadid2}) to obtain 
\begin{eqnarray}
\func{Ad}_{V^{-1}}\left[ \left( \func{Ad}_{V}A\right) \left( \func{Ad}%
_{V}B\right) \right] &=&AB-\left\vert V\right\vert ^{-6}\left[ A,B,V^{3}%
\right] V^{3}  \notag \\
&=&AB+\left[ A,B,V^{-3}\right] V^{3}  \label{advinvadad}
\end{eqnarray}%
Finally, using the identity 2 from Lemma \ref{lemAssocIds2}, we can rewrite
this as 
\begin{equation*}
\func{Ad}_{V^{-1}}\left[ \left( \func{Ad}_{V}A\right) \left( \func{Ad}%
_{V}B\right) \right] =\left( AV^{-3}\right) \left( V^{3}B\right) .
\end{equation*}
\end{proof}

Let us use the action of $\func{Ad}_{V}$ to define a new octonion product $%
A\circ _{V^{3}}B$ given by 
\begin{equation}
A\circ _{V^{3}}B=\func{Ad}_{V}\left[ \left( \func{Ad}_{V^{-1}}A\right)
\left( \func{Ad}_{V^{-1}}B\right) \right]
\end{equation}%
The corresponding $3$-form $\varphi _{V^{3}}$ that defines the product $%
A\circ _{V^{3}}B$ is then given by 
\begin{equation}
\varphi _{V^{3}}\left( A,B,C\right) =\varphi \left( \func{Ad}_{V^{-1}}A,%
\func{Ad}_{V^{-1}}B,\func{Ad}_{V^{-1}}C\right)  \label{phiadVdef}
\end{equation}%
We know that $\func{Ad}_{V}$ is an invertible map. Therefore, $\varphi _{V}$
is pointwise in the $GL\left( 7,\mathbb{R}\right) $-orbit of the original $3$%
-form $\varphi $, and is therefore another positive $3$-form, so it defines
a new $G_{2}$-structure. However, since $\func{Ad}_{V}\ $preserves the
metric, $\varphi _{V}$ has the same associated metric $g$ as $\varphi $. We
will now give an explicit description of (\ref{phiadVdef}) and will show
that the descriptions of isometric $G_{2}$-structures (\ref{phisameg1}) and (%
\ref{phiadVdef}) are equivalent. For convenience, let us define a map of $3$%
-forms that gives (\ref{phisameg1}):

\begin{definition}
Let $A=\left( a,\alpha \right) $ be a nowhere-vanishing octonion. Then,
define the map $\sigma _{A}:\Omega ^{3}\left( M\right) \longrightarrow
\Omega ^{3}\left( M\right) $ given by 
\begin{equation}
\sigma _{A}\left( \varphi \right) =\frac{1}{\left\vert A\right\vert ^{2}}%
\left( \left( a^{2}-\left\vert \alpha \right\vert ^{2}\right) \varphi
-2a\alpha \lrcorner \left( \ast \varphi \right) +2\alpha \wedge \left(
\alpha \lrcorner \varphi \right) \right)  \label{sigmaAdef}
\end{equation}
\end{definition}

In particular, from Theorem \ref{ThmSamegfam}, $\sigma _{A}$ is a map of
positive $3$-forms, and moreover, it preserves the metric class. Note that $%
\sigma _{A}=\sigma _{fA}$ for any nowhere vanishing function $f$. Therefore,
usually it is enough to take $A$ as a unit octonion.

\begin{theorem}
\label{ThmSigmaAd}Let $\left( \varphi ,g\right) $ be a $G_{2}$-structure on
a smooth $7$-dimensional manifold $M$. Then, for any nowhere-vanishing
octonion $V$, 
\begin{equation}
\sigma _{V^{3}}\left( \varphi \right) \left( \cdot ,\cdot ,\cdot \right)
=\varphi \left( \func{Ad}_{V^{-1}}\cdot ,\func{Ad}_{V^{-1}}\cdot ,\func{Ad}%
_{V^{-1}}\cdot \right)  \label{phiadadad}
\end{equation}
\end{theorem}

\begin{proof}
Note that for pure imaginary octonions $A,B,C,$ 
\begin{equation*}
\varphi \left( A,B,C\right) =\left\langle A\times B,C\right\rangle
=\left\langle AB,C\right\rangle
\end{equation*}%
Therefore, 
\begin{eqnarray*}
\varphi _{V^{3}}\left( A,B,C\right) &=&\varphi \left( \func{Ad}_{V^{-1}}A,%
\func{Ad}_{V^{-1}}B,\func{Ad}_{V^{-1}}C\right) \\
&=&\left\langle \func{Ad}_{V^{-1}}A\func{Ad}_{V^{-1}}B,\func{Ad}%
_{V^{-1}}C\right\rangle \\
&=&\left\langle \func{Ad}_{V}\left( \func{Ad}_{V^{-1}}A\func{Ad}%
_{V^{-1}}B\right) ,C\right\rangle
\end{eqnarray*}%
where we have used the fact that $\func{Ad}_{V^{-1}}^{\ast }=\func{Ad}_{V}.$
Now, using (\ref{AdinAdAdid}), we have 
\begin{eqnarray*}
\varphi _{V^{3}}\left( A,B,C\right) &=&\left\langle AB+\left[ A,B,V^{3}%
\right] V^{-3},C\right\rangle \\
&=&\varphi \left( A,B,C\right) +\left\langle \left[ A,B,V^{3}\right]
V^{-3},C\right\rangle
\end{eqnarray*}%
Suppose now $V^{3}=\left( u_{0},u\right) $ and for convenience, let $%
\left\vert V^{3}\right\vert ^{2}=u_{0}^{2}+\left\vert u\right\vert ^{2}=M,$
then,%
\begin{equation*}
\left[ A,B,V^{3}\right] V^{-3}=\frac{u_{0}}{M}\left[ A,B,u\right] -\frac{1}{M%
}\left[ A,B,u\right] \times u
\end{equation*}%
In index notation, and using (\ref{octoassoc}) to express the associator in
terms of $\psi ,$ we get 
\begin{equation}
\left( \varphi _{V}^{3}\right) _{abc}=\varphi _{abc}+\frac{2u_{0}}{M}\psi
_{cabd}u^{d}-\frac{2}{M}\varphi _{cmn}\psi _{\ \ abd}^{m}u^{d}u^{n}
\label{phiV2}
\end{equation}%
To simplify the last term we need to use a contraction identity between $%
\varphi $ and $\psi $ (see for example \cite%
{GrigorianG2Torsion1,karigiannis-2005-57}). 
\begin{equation}
\varphi _{abc}\psi _{mnp}^{\ \ \ \ \text{\ }c}=-3\left( g_{a[m}\varphi
_{np]b}-g_{b[m}\varphi _{np]a}\right)  \label{phipsiid}
\end{equation}%
where the square parentheses denote skew-symmetrization. Using this, we get 
\begin{equation*}
\varphi _{cmn}\psi _{\ \ abd}^{m}u^{d}u^{n}=\left\vert u\right\vert
^{2}\varphi _{abc}-3u_{[a}\varphi _{bc]m}u^{m}
\end{equation*}%
Therefore, (\ref{phiV2}), becomes%
\begin{equation*}
\left( \varphi _{V^{3}}\right) _{abc}=\left( 1-\frac{2}{M}\left\vert
u\right\vert ^{2}\right) \varphi _{abc}+\frac{2u_{0}}{M}\psi _{cabd}u^{d}+%
\frac{6}{M}u_{[a}\varphi _{bc]m}u^{m}
\end{equation*}%
This can now be rewritten in coordinate-free notation as%
\begin{equation*}
\varphi _{V^{3}}=\frac{1}{M}\left[ \left( u_{0}^{2}-\left\vert u\right\vert
^{2}\right) \varphi -2u_{0}u\lrcorner \psi +2u\wedge \left( u\lrcorner
\varphi \right) \right]
\end{equation*}%
Comparing with (\ref{sigmaAdef}), this is precisely equal to $\sigma
_{V^{3}}\left( \varphi \right) $.
\end{proof}

\begin{remark}
The expression (\ref{AdinAdAdid}) shows that $\circ _{V^{3}}=\circ $ if and
only if $V^{3}$ is real. So if without loss of generality we assume $V$ is a
unit octonion, then the octonion product is preserved by $\func{Ad}_{V}$ if
and only if $V^{6}=1$. This beautiful fact was originally discovered by
Manogue and Schray in \cite{ManogueOctos}.
\end{remark}

\begin{example}
Suppose $V=\left( \cos \frac{\pi }{3},\left( \sin \frac{\pi }{3}\right)
v\right) =\left( \frac{1}{2},\frac{\sqrt{3}}{2}v\right) $ for some unit
vector $v$. Then, $V^{3}=\left( -1,0\right) ,$ so $\sigma _{V^{3}}\left(
\varphi \right) =\varphi $. However, $\left. \func{Ad}_{V}\right\vert _{%
\func{Im}\mathbb{O}}\in SO\left( 7\right) $ is nontrivial. Using (\ref%
{AdVexpressind}) we have an explicit expression%
\begin{equation}
\left( \left. \func{Ad}_{V}\right\vert _{\func{Im}\mathbb{O}}\right) _{\
b}^{a}=-\frac{1}{2}\delta _{\ b}^{a}-\frac{\sqrt{3}}{2}\left( v\lrcorner
\varphi \right) _{\ b}^{a}+\frac{3}{2}v^{a}v_{b}  \label{AdAG2}
\end{equation}%
Since $\left. \func{Ad}_{V}\right\vert _{\func{Im}\mathbb{O}}$ preserves $%
\varphi $, pointwise, $\left. \func{Ad}_{V}\right\vert _{\func{Im}\mathbb{O}%
}\in G_{2}.$ In particular, for any nowhere-vanishing vector field $v$ on $M$%
, $\func{Ad}_{V}$ defines an octonion-product preserving map at every point
in $M$.
\end{example}

Due to the nonassociativity of octonions, in general $\func{Ad}_{U}\func{Ad}%
_{V}\neq \func{Ad}_{UV}$. However we do have such a composition property for
the map $\sigma _{V}$.

\begin{theorem}
\label{ThmSigmaUV}Let $\left( \varphi ,g\right) $ be a $G_{2}$-structure on
a smooth $7$-dimensional manifold $M$. Then, given nowhere-vanishing
octonions $U$ and $V$, 
\begin{equation}
\sigma _{U}\left( \sigma _{V}\varphi \right) =\sigma _{UV}\left( \varphi
\right)  \label{sigmaUV}
\end{equation}
\end{theorem}

Before we prove Theorem \ref{ThmSigmaUV} we need a few more properties of
octonion products.

\begin{remark}
\label{remSigmaUV}In Theorem \ref{ThmSigmaUV}, the octonion product is
defined using $\varphi $. However, we will show that 
\begin{equation*}
U\circ _{\varphi }V=U\circ _{V}V.
\end{equation*}%
Suppose that $V=\left( v_{0},v\right) .$ Then, 
\begin{eqnarray*}
v\lrcorner \sigma _{V}\left( \varphi \right) &=&\frac{1}{M}v\lrcorner \left[
\left( v_{0}^{2}-\left\vert v\right\vert ^{2}\right) \varphi
-2v_{0}v\lrcorner \psi +2v\wedge \left( v\lrcorner \varphi \right) \right] \\
&=&\frac{1}{M}\left( \left( v_{0}^{2}-\left\vert v\right\vert ^{2}\right)
v\lrcorner \varphi +2\left\vert v\right\vert ^{2}v\lrcorner \varphi \right)
\\
&=&v\lrcorner \varphi
\end{eqnarray*}%
Therefore, any product with $V$ using $\circ _{V}$, defined by $\sigma
_{V}\left( \varphi \right) $, will be equal to the product with using $\circ
,$ that is defined by $\varphi $. Hence, the product $UV$ is unambiguous
whether defined using $\varphi $ or using $\sigma _{V}\left( \varphi \right) 
$.
\end{remark}

From Theorem \ref{ThmSigmaAd} and Proposition \ref{PropAdCprod}, we know
that the octonion product defined by $\sigma _{V}\left( \varphi \right) $ is
given by 
\begin{equation}
A\circ _{V}B=AB+\left[ A,B,V\right] V^{-1}=\left( AV\right) \left(
V^{-1}B\right)  \label{OctoVAB}
\end{equation}%
Define $\left[ \cdot ,\cdot ,\cdot \right] _{V}$ to be the associator with
respect to the product $\circ _{V}$. Then using (\ref{OctoVAB}) as well as
the associator identities in Lemma \ref{lemAssocIds}, we obtain the
following expression for $\left[ \cdot ,\cdot ,\cdot \right] _{V}.$ The
proof is given in the Appendix.

\begin{lemma}
\label{lemAssocVId}Let $A,B,C\in \Gamma \left( \mathbb{O}M\right) $, and
define 
\begin{equation}
\left[ A,B,C\right] _{V}=A\circ _{V}\left( B\circ _{V}C\right) -\left(
A\circ _{V}B\right) \circ _{V}C  \label{assocV}
\end{equation}%
where the product $\circ _{V}$ is defined by (\ref{OctoVAB}) for a
nowhere-vanishing octonion $V$. Then, 
\begin{equation}
\left[ A,B,C\right] _{V}=\left[ A,B,CV\right] V^{-1}-\left[ A,B,V\right]
\left( V^{-1}C\right)  \label{AssocVid}
\end{equation}
\end{lemma}

Using Lemma \ref{lemAssocVId}, we can now prove Theorem \ref{ThmSigmaUV}.

\begin{proof}[Proof of Theorem \protect\ref{ThmSigmaUV}]
Let $A,B\in \Gamma \left( \mathbb{O}M\right) .$ Let $\tilde{\circ}$ be the
octonion product defined by $\sigma _{U}\left( \sigma _{V}\varphi \right) .$
Using (\ref{OctoVAB}), this is then given by 
\begin{equation}
A\tilde{\circ}B=A\circ _{V}B+\left[ A,B,U\right] _{V}\circ _{V}U^{-1}
\label{AcirctildB}
\end{equation}%
since we are now starting with $\circ _{V}$ that is defined by $\sigma
_{V}\left( \varphi \right) $, and are changing it to $\tilde{\circ}$ that is
defined by $\sigma _{U}\left( \sigma _{V}\varphi \right) $. Therefore,
expanding (\ref{AcirctildB}) using (\ref{OctoVAB}) and (\ref{AssocVid}), we
have 
\begin{eqnarray}
A\tilde{\circ}B &=&AB+\left[ A,B,V\right] V^{-1}+\left[ A,B,U\right]
_{V}U^{-1}+\left[ \left[ A,B,U\right] _{V},U^{-1},V\right] V^{-1}  \notag \\
&=&AB+\left[ A,B,V\right] V^{-1}+\left( \left[ A,B,UV\right] V^{-1}-\left[
A,B,V\right] \left( V^{-1}U\right) \right) U^{-1}  \notag \\
&&+\left[ \left[ A,B,UV\right] V^{-1}-\left[ A,B,V\right] \left(
V^{-1}U\right) ,U^{-1},V\right] V^{-1}  \notag \\
&=&AB+\left[ A,B,V\right] V^{-1}+\left( \left[ A,B,UV\right] V^{-1}\right)
U^{-1}+\left[ \left[ A,B,UV\right] V^{-1},U^{-1},V\right] V^{-1}
\label{UVthmexpr1} \\
&&-\left( \left[ A,B,V\right] \left( V^{-1}U\right) \right) U^{-1}-\left[ %
\left[ A,B,V\right] \left( V^{-1}U\right) ,U^{-1},V\right] V^{-1}  \notag
\end{eqnarray}%
Note that using Lemma \ref{lemAssocIds} and the definition of the
associator, we get 
\begin{eqnarray}
\left[ \left[ A,B,UV\right] V^{-1},U^{-1},V\right] V^{-1} &=&\left[ \left[
A,B,UV\right] ,U^{-1},V\right] \frac{\left\vert V\right\vert ^{2}}{%
\left\vert V\right\vert ^{4}}  \notag \\
&=&-\left[ \left[ A,B,UV\right] ,U^{-1},V^{-1}\right]  \notag \\
&=&\left[ \left[ A,B,UV\right] ,V^{-1},U^{-1}\right]  \notag \\
&=&\left[ A,B,UV\right] \left( V^{-1}U^{-1}\right) -\left( \left[ A,B,UV%
\right] V^{-1}\right) U^{-1}  \label{UVthmexpr2}
\end{eqnarray}%
and 
\begin{eqnarray}
\left[ \left[ A,B,V\right] \left( V^{-1}U\right) ,U^{-1},V\right] V^{-1}
&=&\left( \left( \left[ A,B,V\right] \left( V^{-1}U\right) \right) \left(
U^{-1}V\right) ^{{}}\right) V^{-1}-\left( \left[ A,B,V\right] \left(
V^{-1}U\right) \right) U^{-1}  \notag \\
&=&\left[ A,B,V\right] \left( \left( V^{-1}U\right) \left( U^{-1}V\right)
\right) V^{-1}-\left( \left[ A,B,V\right] \left( V^{-1}U\right) \right)
U^{-1}  \notag \\
&=&\left[ A,B,V\right] V^{-1}-\left( \left[ A,B,V\right] \left(
V^{-1}U\right) \right) U^{-1}  \label{UVthmexpr3}
\end{eqnarray}%
where we have used the fact that $V^{-1}U=\left( U^{-1}V\right) ^{-1}$ twice
- in the second line to conclude that $\left[ A,B,V\right] ,$ $V^{-1}U,$ and 
$U^{-1}V$ associate, and in the third line to simply. Now substituting (\ref%
{UVthmexpr2}) and (\ref{UVthmexpr3}) into (\ref{UVthmexpr1}), we are left
with 
\begin{eqnarray*}
A\tilde{\circ}B &=&AB+\left[ A,B,UV\right] \left( V^{-1}U^{-1}\right) \\
&=&AB+\left[ A,B,UV\right] \left( UV\right) ^{-1}
\end{eqnarray*}
Therefore, 
\begin{equation*}
A\tilde{\circ}B=A\circ _{UV}B
\end{equation*}%
Therefore, $\sigma _{U}\left( \sigma _{V}\varphi \right) =\sigma _{UV}\left(
\varphi \right) .$
\end{proof}

From Theorem \ref{ThmSigmaUV} we hence see that the action of $\sigma _{V}$
on positive $3$-forms corresponds to octonion multiplication on the left.
The map $\sigma _{V}$ then also gives a representation of the non-zero
octonion Moufang loop on $3$-forms. Given a fixed \textquotedblleft
reference\textquotedblright\ $G_{2}$-structure, this then allows us to
freely work with octonions rather than $3$-forms.

A few consequences of Theorem \ref{ThmSigmaUV} are the following. First of
all it is clear that $\sigma _{V^{-1}}\left( \sigma _{V}\varphi \right)
=\varphi $. Also, note that%
\begin{equation}
\sigma _{U^{3}}\left( \sigma _{V^{3}}\varphi \right) =\sigma
_{U^{3}V^{3}}\varphi  \label{sigmaU3V3}
\end{equation}%
However, for any octonions $A,B,C$ 
\begin{equation*}
\left( \sigma _{V^{3}}\varphi \right) \left( A,B,C\right) =\varphi \left( 
\func{Ad}_{V^{-1}}A,\func{Ad}_{V^{-1}}B,\func{Ad}_{V^{-1}}C\right)
\end{equation*}%
Hence, 
\begin{eqnarray}
\sigma _{U^{3}}\left( \sigma _{V^{3}}\varphi \right) \left( A,B,C\right)
&=&\left( \sigma _{V^{3}}\varphi \right) \left( \func{Ad}_{U^{-1}}^{\left(
V^{3}\right) }A,\func{Ad}_{U^{-1}}^{\left( V^{3}\right) }B,\func{Ad}%
_{U^{-1}}^{\left( V^{3}\right) }C\right)  \notag \\
&=&\varphi \left( \func{Ad}_{V^{-1}}\left( \func{Ad}_{U^{-1}}^{\left(
V^{3}\right) }A\right) ,\func{Ad}_{V^{-1}}\left( \func{Ad}_{U^{-1}}^{\left(
V^{3}\right) }B\right) ,\func{Ad}_{V^{-1}}\left( \func{Ad}_{U^{-1}}^{\left(
V^{3}\right) }C\right) \right)  \label{sigmaU3V32}
\end{eqnarray}%
where $\func{Ad}_{U^{-1}}^{\left( V^{3}\right) }$ means that we are applying
the $\func{Ad}$ operator with respect to the $G_{2}$-structure $\sigma
_{V^{3}}\varphi $. Finally, 
\begin{equation*}
\left( \sigma _{U^{3}V^{3}}\varphi \right) \left( A,B,C\right) =\varphi
\left( \func{Ad}_{\left( U^{3}V^{3}\right) ^{-\frac{1}{3}}}A,\func{Ad}%
_{\left( U^{3}V^{3}\right) ^{-\frac{1}{3}}}B,\func{Ad}_{\left(
U^{3}V^{3}\right) ^{-\frac{1}{3}}}C\right)
\end{equation*}%
Therefore, we may conclude that 
\begin{equation}
\func{Ad}_{\left( U^{3}V^{3}\right) ^{\frac{1}{3}}}\func{Ad}_{V^{-1}}\func{Ad%
}_{U^{-1}}^{\left( V^{3}\right) }\in G_{2}  \label{AdU3V3}
\end{equation}%
In particular, if $U$ and $V$ are both $6$th roots of unity, then $\sigma
_{U^{3}V^{3}}\varphi =\sigma _{\pm 1}\varphi =\varphi $. Also, $\func{Ad}%
^{\left( V^{3}\right) }=\func{Ad}$. Therefore, we conclude that in this case 
\begin{equation*}
\func{Ad}_{V^{-1}}\func{Ad}_{U^{-1}}\in G_{2}.
\end{equation*}%
This of course is to be expected, since $\func{Ad}_{U^{-1}}\in G_{2}$ and $%
\func{Ad}_{V^{-1}}\in G_{2},$ so their composition is also in $G_{2}$ by the
group property.

\section{Torsion of a $G_{2}$-structure}

\setcounter{equation}{0}\label{secTorsion}So far we have only looked at the
algebraic properties of $G_{2}$-structures and octonions. However, given a $%
G_{2}$-structure $\varphi $ with an associated metric $g$, we may use the
metric to define the Levi-Civita connection $\nabla $. The \emph{intrinsic
torsion }of a $G_{2}$-structure is then defined by $\nabla \varphi $.
Following \cite{GrigorianG2Torsion1,karigiannis-2007}, we can write 
\begin{equation}
\nabla _{a}\varphi _{bcd}=2T_{a}^{\ e}\psi _{ebcd}^{{}}  \label{codiffphi}
\end{equation}%
where $T_{ab}$ is the \emph{full torsion tensor}. Similarly, we can also
write 
\begin{equation}
\nabla _{a}\psi _{bcde}=-8T_{a[b}\varphi _{cde]}  \label{psitorsion}
\end{equation}%
We can also invert (\ref{codiffphi}) to get an explicit expression for $T$ 
\begin{equation}
T_{a}^{\ m}=\frac{1}{48}\left( \nabla _{a}\varphi _{bcd}\right) \psi ^{mbcd}.
\end{equation}%
This $2$-tensor fully defines $\nabla \varphi $ \cite{GrigorianG2Torsion1}.

\begin{remark}
The torsion tensor $T$ as defined here is actually corresponds to $\frac{1}{2%
}T$ in \cite{GrigorianG2Torsion1} and $-\frac{1}{2}T$ in \cite%
{karigiannis-2007}. Even though this requires extra care when translating
various results, it will turn out to be more convenient, because otherwise
we would have a factor of $\frac{1}{2}$ everywhere.
\end{remark}

In general we can obtain an orthogonal decomposition of $T_{ab}$ according
to representations of $G_{2}$ into \emph{torsion components}: 
\begin{equation}
T=\tau _{1}g+\tau _{7}\lrcorner \varphi +\tau _{14}+\tau _{27}
\end{equation}%
where $\tau _{1}$ is a function, and gives the $\mathbf{1}$ component of $T$%
. We also have $\tau _{7}$, which is a $1$-form and hence gives the $\mathbf{%
7}$ component, and, $\tau _{14}\in \Lambda _{14}^{2}$ gives the $\mathbf{14}$
component and $\tau _{27}$ is traceless symmetric, giving the $\mathbf{27}$
component. As it was originally shown by Fern\'{a}ndez and Gray \cite%
{FernandezGray}, there are in fact a total of 16 torsion classes of $G_{2}$%
-structures that arise depending on which of the components are non-zero.
Moreover, as shown in \cite{karigiannis-2007}, the torsion components $\tau
_{i}$ relate directly to the expression for $d\varphi $ and $d\psi $. It can
also be shown \cite{CleytonIvanovCurv,GrigorianG2Torsion1,karigiannis-2007},
that $T$ satisfies a \textquotedblleft Bianchi identity\textquotedblright :

\begin{proposition}
\label{propTorsBianchi}Let $\varphi $ be a $G_{2}$-structure with an
associated metric $g$, and Levi-Civita connection $\nabla $. The torsion
tensor then satisfies%
\begin{equation}
\nabla _{a}T_{bc}-\nabla _{b}T_{ac}+2T_{am}^{{}}T_{bn}^{{}}\varphi _{\ \ \
c}^{mn}-\frac{1}{4}\func{Riem}_{abmn}^{{}}\varphi _{\ \ \ c}^{mn}=0
\label{torsionbianchi}
\end{equation}%
where $\func{Riem}$ is the Riemann curvature tensor of the metric $g$.
\end{proposition}

In particular, the term in (\ref{torsionbianchi}) involving $\func{Riem}$,
is precisely the component of the curvature $2$-form that lies in $\Lambda
_{7}^{2}$ , namely $\pi _{7}\func{Riem}$. Proposition \ref{propTorsBianchi}
then tells us that $\pi _{7}\func{Riem}$ is fully determined by the torsion,
and in particular, if the torsion vanishes, then $\pi _{7}\func{Riem}=0$.
Moreover, if we contract (\ref{torsionbianchi}) with $\varphi _{\ \ d}^{bc},$
we obtain an expression for the Ricci curvature $\func{Ric}$ in terms of $T$ 
\cite{CleytonIvanovCurv,GrigorianG2Torsion1,karigiannis-2007}: 
\begin{eqnarray}
\func{Ric}_{ab} &=&2\left( \nabla _{a}^{{}}T_{nm}^{{}}-\nabla
_{n}^{{}}T_{am}^{{}}\right) \varphi _{\ \ \ b}^{nm}-4T_{an}^{{}}T_{\ b}^{n}+4%
\func{Tr}\left( T\right) T_{ab}  \label{torsricci} \\
&&+4T_{ac}^{{}}T_{nm}^{{}}\psi _{\ \ \ \ b}^{nmc\ \ \ }  \notag
\end{eqnarray}%
This then shows that if the torsion vanishes, then so does $\func{Ric}$. Of
course, Ricci curvature is a function of the metric, so it is invariant over
the metric class of $G_{2}$-structures. In particular, the scalar curvature
is given by 
\begin{equation}
\frac{1}{4}R=42\tau _{1}^{2}+30\left\vert \tau _{7}\right\vert
^{2}-\left\vert \tau _{14}\right\vert ^{2}-\left\vert \tau _{27}\right\vert
^{2}+6\func{div}\tau _{7}  \label{torsscalcurv}
\end{equation}

When the torsion vanishes, that is $T=0$, the $G_{2}$-structure is said to
be torsion-free. This is equivalent to $\nabla \varphi =0$ and also
equivalent, by Fern\'{a}ndez and Gray, to $d\varphi =d\psi =0$. Moreover, a $%
G_{2}$-structure is torsion-free if and only if the holonomy group $\func{Hol%
}\left( g\right) $ of the corresponding metric $g$ is contained in $G_{2}$ 
\cite{Joycebook}. The holonomy group is then precisely equal to $G_{2}$ if
and only if the fundamental group $\pi _{1}$ is finite \cite{Joycebook}. As
we have seen, for any metric there is a family of compatible $G_{2}$%
-structures. Holonomy however is a property of the metric, so we can
reformulate the correspondence between $G_{2}$-structures and holonomy as
follows.

\begin{theorem}[\protect\cite{FernandezGray,Joycebook}]
\label{ThmHolonomy}Let $\left( M,g\right) $ be a smooth $7$-dimensional
Riemannian manifold with $w_{1}=w_{2}=0$. Then, $\func{Hol}\left( g\right)
\subseteq G_{2}$ if and only if there exists a torsion-free $G_{2}$%
-structure $\varphi $ that is compatible with $g$.
\end{theorem}

Therefore, in order to understand the holonomy of a given metric we have to
understand if the corresponding metric class of $G_{2}$-structures contains
a torsion-free $G_{2}$-structure. A\ necessary condition for a metric class
to admit a torsion-free $G_{2}$-structure is $\func{Ric}=0$. The converse
would be true if any metric on a $7$-manifold with $\func{Ric}=0$ has
reduced holonomy.

It is possible to explicitly work out the expression for the torsion of a $%
G_{2}$-structure $\sigma _{V}\left( \varphi \right) $ in terms of the
torsion of $\varphi $, however the expression in index notation is not too
illuminating. To get the expression in terms of octonions, we will first
define a covariant derivative on sections of the octonion bundle.

\section{Octonion covariant derivative}

\setcounter{equation}{0}\label{secOctoCovDiv}

Consider the octonion bundle $\mathbb{O}M$ with the octonion algebra defined
by the $G_{2}$-structure $\varphi $ with torsion tensor $T$. Then, we can
extend the Levi-Civita connection $\nabla $ to sections of $\mathbb{O}M$.
Let $A=\left( a,\alpha \right) \in \Gamma \left( \mathbb{O}M\right) ,$ then
define the covariant derivative on $\mathbb{O}M$ as 
\begin{equation}
\nabla _{X}A=\left( \nabla _{X}a,\nabla _{X}\alpha \right)  \label{delXA}
\end{equation}%
for any $X\in \Gamma \left( TM\right) $. Now the question is how does $%
\nabla $ interact with octonion multiplication.

\begin{proposition}
\label{propOctoLC}Suppose $A,B\in \Gamma \left( \mathbb{O}M\right) $. Then,
for $X\in \Gamma \left( TM\right) $ 
\begin{equation}
\nabla _{X}\left( AB\right) =\left( \nabla _{X}A\right) B+A\left( \nabla
_{X}B\right) -\left[ T_{X},A,B\right]  \label{nablaXAB}
\end{equation}%
where $T_{X}=\left( 0,X\lrcorner T\right) $.\newline
\end{proposition}

\begin{proof}
Suppose $A=\left( a,\alpha \right) $ and $B=\left( b,\beta \right) $, then,
using the definition of octonion multiplication (\ref{octoproddef}) we write 
\begin{equation*}
\nabla _{X}\left( AB\right) =\nabla _{X}\left( 
\begin{array}{c}
ab-\left\langle \alpha ,\beta \right\rangle \\ 
a\beta +b\alpha +\varphi \left( \alpha ,\beta ,\cdot ^{\sharp }\right)%
\end{array}%
\right)
\end{equation*}%
Using the Leibniz property and metric compatibility of $\nabla _{X},$ we
then get 
\begin{equation*}
\nabla _{X}\left( AB\right) =\left( \nabla _{X}A\right) B+B\left( \nabla
_{X}A\right) +\left( 
\begin{array}{c}
0 \\ 
\left( \nabla _{X}\varphi \right) \left( \alpha ,\beta ,\cdot ^{\sharp
}\right)%
\end{array}%
\right)
\end{equation*}%
However, from (\ref{codiffphi}), $\nabla _{X}\varphi =2T_{X}\lrcorner \psi $%
, and using the relationship between $\psi $ and the associator (\ref%
{octoassoc}), we get 
\begin{eqnarray*}
\left( \nabla _{X}\varphi \right) \left( \alpha ,\beta ,\cdot ^{\sharp
}\right) &=&2\psi \left( T_{X},\alpha ,\beta ,\cdot ^{\sharp }\right) \\
&=&-\left[ T_{X},A,B\right]
\end{eqnarray*}%
Therefore, indeed we obtain (\ref{nablaXAB}).
\end{proof}

Note that if either of $A$ or $B$ in (\ref{nablaXAB}) is real, then the
associator vanishes, and we recover the standard Leibniz rule for $\nabla $.

\begin{remark}
Proposition \ref{propOctoLC} has two important implications. Firstly, note
that $\left[ T_{X},A,B\right] $ vanishes for all $X,A,B$ if and only if $%
T=0. $ Therefore, the Levi-Civita connection is compatible with octonion
multiplication if and only if the $G_{2}$-structure is torsion-free.
Therefore, the torsion $T$ is an obstruction to $\nabla $ being compatible
with octonion product. This is of course expected, since the product is
defined by $\varphi $, and $T$ is the precisely given by $\nabla \varphi $.
Secondly, here we are treating the torsion tensor $T$ as a \textquotedblleft
pure octonion-valued $1$-form\textquotedblright\ on $M$. That is, 
\begin{equation}
T\in \Gamma \left( T^{\ast }M\otimes \func{Im}\mathbb{O}M\right) =\Omega
^{1}\left( \func{Im}\mathbb{O}M\right) ,  \label{T1form}
\end{equation}%
so that in particular for any vector $X$ on $M,$ 
\begin{equation*}
T_{X}\in \Gamma \left( \func{Im}\mathbb{O}M\right) .
\end{equation*}%
This presents an important shift in perception of what torsion of a $G_{2}$%
-structure actually is. Recall that a principal bundle connection can be
thought of as a Lie algebra-valued $1$-form. A Lie algebra is the tangent
space to the identity of a Lie group. In our case, the pure imaginary
octonions precisely form the tangent space to $1$ in the Moufang loop of
unit octonions. Therefore, $T$ is the octonionic analog of a
\textquotedblleft Lie algebra-valued $1$-form\textquotedblright .
Alternatively, it can be thought as some kind of a \textquotedblleft gauge
connection\textquotedblright\ for a non-associative gauge theory.
\end{remark}

For a generic $G_{2}$-structure, $\nabla $ does not satisfy the derivation
property with respect to the octonion product. Let us however define an
adapted covariant derivative, using $T$ as a \textquotedblleft
connection\textquotedblright\ $1$-form.

\begin{definition}
Define the octonion covariant derivative $D$ such for any $X\in \Gamma
\left( TM\right) ,$ 
\begin{equation*}
D_{X}:\Gamma \left( \mathbb{O}M\right) \longrightarrow \Gamma \left( \mathbb{%
O}M\right)
\end{equation*}%
given by 
\begin{equation}
D_{X}A=\nabla _{X}A-AT_{X}  \label{octocov}
\end{equation}%
for any $A\in \Gamma \left( \mathbb{O}M\right) .$
\end{definition}

Using the octonion covariant derivative, we in particular have 
\begin{equation}
D_{X}1=-T_{X}  \label{DX1}
\end{equation}%
The idea to use the $G_{2}$-structure torsion to define a new connection is
certainly not new (see for example \cite%
{AgricolaSrni,AgricolaSpinors,AgricolaFriedrich1}, and references therein).
However using it to define an \emph{octonion} covariant derivative is a new
concept. This covariant derivative satisfies a \emph{partial} derivation
property with respect to the octonion product.

\begin{proposition}
\label{propDXprod}Suppose $A,B\in \Gamma \left( \mathbb{O}M\right) $ and $%
X\in \Gamma \left( TM\right) $, then%
\begin{equation}
D_{X}\left( AB\right) =\left( \nabla _{X}A\right) B+A\left( D_{X}B\right)
\label{octocovprod}
\end{equation}
\end{proposition}

\begin{proof}
We use the definition of $D$ (\ref{octocov}) to write out $D_{X}\left(
AB\right) $:%
\begin{equation*}
D_{X}\left( AB\right) =\nabla _{X}\left( AB\right) -\left( AB\right) T_{X}
\end{equation*}%
Then expanding $\nabla _{X}\left( AB\right) $ using Proposition \ref%
{propOctoLC}, and applying properties of the associator we get 
\begin{eqnarray*}
D_{X}\left( AB\right) &=&\left( \nabla _{X}A\right) B+A\left( \nabla
_{X}B\right) -\left[ T_{X},A,B\right] -\left( AB\right) T_{X} \\
&=&\left( \nabla _{X}A\right) B+A\left( \nabla _{X}B\right) -\left[ A,B,T_{X}%
\right] -\left( AB\right) T_{X} \\
&=&\left( \nabla _{X}A\right) B+A\left( \nabla _{X}B\right) -A\left(
BT_{X}\right) +\left( AB\right) T_{X}-\left( AB\right) T_{X} \\
&=&\left( \nabla _{X}A\right) B+A\left( \nabla _{X}B-BT_{X}\right) \\
&=&\left( \nabla _{X}A\right) B+A\left( D_{X}B\right)
\end{eqnarray*}
\end{proof}

Therefore, $D_{X}$ satisfies a \textquotedblleft
one-sided\textquotedblright\ derivation identity - the derivative on the
first term of the right hand side of (\ref{octocovprod}) is a standard $%
\nabla ,$ however on the second term we have a $D.$ However, if $A$ is real
in (\ref{octocovprod}), then we see that it does give us what we would
expect. Moreover, we now show that $D$ is metric-compatible. Recall from\
Section \ref{secOctobundle} that we extend the metric $g$ on $M$ to $\mathbb{%
O}M$ by setting: 
\begin{equation}
g\left( A,B\right) =\left( \func{Re}A\right) \left( \func{Re}B\right)
+g\left( \func{Im}A,\func{Im}B\right)  \label{gABdef}
\end{equation}%
where the imaginary parts are now regarded as vectors on $M$.

\begin{proposition}
\label{propDXmetric}Suppose $A,B\in \Gamma \left( \mathbb{O}M\right) $ and $%
X\in \Gamma \left( TM\right) $, then%
\begin{equation}
\nabla _{X}\left( g\left( A,B\right) \right) =g\left( D_{X}A,B\right)
+g\left( A,D_{X}B\right) \,  \label{DXmet}
\end{equation}
\end{proposition}

\begin{proof}
We can rewrite $g\left( D_{X}A,B\right) $ as%
\begin{eqnarray}
g\left( \nabla _{X}A-AT_{X},B\right) &=&g\left( \nabla _{X}A,B\right)
-g\left( AT_{X},B\right)  \notag \\
&=&g\left( \nabla _{X}A,B\right) -g\left( T_{X},\bar{A}B\right)
\label{gDXAB}
\end{eqnarray}%
where we have used Lemma \ref{lemRLmult}: $L_{A}^{\ast }=L_{\bar{A}}.$
Similarly, 
\begin{eqnarray}
g\left( A,D_{X}B\right) &=&g\left( A,\nabla _{X}B\right) -g\left(
A,BT_{X}\right)  \notag \\
&=&g\left( A,\nabla _{X}B\right) -g\left( T_{X},\bar{B}A\right)
\label{gBDXA}
\end{eqnarray}%
Combining (\ref{gDXAB}) and (\ref{gBDXA}), we obtain 
\begin{equation*}
g\left( D_{X}A,B\right) +g\left( A,D_{X}B\right) =g\left( \nabla
_{X}A,B\right) +g\left( A,\nabla _{X}B\right) -g\left( T_{X},\bar{A}B+\bar{B}%
A\right)
\end{equation*}%
However, $T_{X}$ is pure imaginary, while $\bar{A}B+\bar{B}A$ is real.
Therefore, their inner product vanishes. Hence, 
\begin{equation*}
g\left( D_{X}A,B\right) +g\left( A,D_{X}B\right) =g\left( \nabla
_{X}A,B\right) +g\left( A,\nabla _{X}B\right) =\nabla _{X}\left( g\left(
A,B\right) \right) .
\end{equation*}
\end{proof}

Further, given $\mathbb{O}M$-valued differential forms, we can extend $D$ to
an exterior covariant derivative $d_{D}$. Define $\Omega ^{p}\left( \mathbb{O%
}M\right) $ to be sections of the bundle $\Lambda ^{p}\left( T^{\ast
}M\right) \otimes \mathbb{O}M$, that is, octonion-valued $p$-forms. Then, 
\begin{equation}
d_{D}:\Omega ^{p}\left( \mathbb{O}M\right) \longrightarrow \Omega
^{p+1}\left( \mathbb{O}M\right)  \label{Dextdiff}
\end{equation}%
is such that 
\begin{equation}
d_{D}Q=d_{\nabla }Q-\left( -1\right) ^{p}Q\overset{\circ }{\wedge }T
\label{Dextdiff2}
\end{equation}%
where $d_{\nabla }$ is the skew-symmetrized $\nabla $ and $\overset{\circ }{%
\wedge }$ is a combination of exterior product and octonion product. More
concretely, in index notation, if we suppose $Q$ is given by 
\begin{equation*}
Q=\frac{1}{p!}Q_{a_{1}..a_{p}}dx^{a_{1}}\wedge ...\wedge dx^{a_{p}}
\end{equation*}%
where $Q_{a_{1}..a_{p}}$ is an octonion section for any fixed $\left\{
a_{1},...,a_{p}\right\} $. Then%
\begin{equation}
\left( d_{D}Q\right) _{b_{1}...b_{p+1}}=\left( p+1\right) \left( \nabla
_{\lbrack b_{1}}Q_{b_{2}...b_{p+1}]}-\left( -1\right)
^{p}Q_{[b_{1}...b_{p}}\circ T_{b_{p+1}]}\right)  \label{Dextdiff3}
\end{equation}%
where each $T_{b_{k}}$ is an imaginary octonion for any $b_{k}$. Suppose $A$
is an $\mathbb{O}M$-valued $0$-form, $B$ is an $\mathbb{O}M$-value $1$-form,
and $C$ is an $\mathbb{O}M$-valued $2$-form, we have the following explicit
formulas for $d_{D}$: 
\begin{subequations}%
\begin{eqnarray}
\left( d_{D}A\right) _{k} &=&\nabla _{k}A-AT_{k} \\
\left( d_{D}B\right) _{kl} &=&2\nabla _{\lbrack
k}B_{l]}+B_{k}T_{l}-B_{l}T_{k} \\
\left( d_{D}C\right) _{klm} &=&3\nabla _{\lbrack
k}C_{lm]}-C_{kl}T_{m}-C_{lm}T_{k}-C_{mk}T_{l}
\end{eqnarray}%
\end{subequations}%

Note that even if $T=0$, $d_{D}^{2}=d_{\nabla }^{2}$ is a function of $\func{%
Riem}$ and is in general not equal to zero. Lemma \ref{lemd2nap} below gives
the precise statement. This is a standard result for vector-valued
differential forms, so the proof is given in Appendix \ref{appProofs}.

\begin{lemma}
\label{lemd2nap}The operator $d_{\nabla }^{2}:\Omega ^{p}\left( \mathbb{O}%
M\right) \longrightarrow \Omega ^{p+2}\left( \mathbb{O}M\right) \ $is given
by 
\begin{equation}
d_{\nabla }^{2}P=\func{Riem}\wedge \left( \func{Im}P\right)  \label{d2nab0}
\end{equation}%
here $\func{Riem}$ is regarded as a section of $\Omega ^{2}\left( M\right)
\otimes \func{End}\left( \func{Im}\mathbb{O}M\right) \cong \Omega ^{2}\left(
M\right) \otimes \func{End}\left( TM\right) $ so that in (\ref{d2nab0}), $%
\func{Riem}\wedge \func{Im}P$ is a wedge product in $\Omega ^{\ast }\left(
M\right) $ and moreover $\func{Riem}$ acts as an endomorphism on $\func{Im}P$%
.
\end{lemma}

Let us now assume that $M$ is compact. We can then define the $L_{2}$-inner
product of octonion-valued forms. Suppose $P$ and $Q$ are octonion-valued $p$%
-forms, then let%
\begin{equation}
\left\langle P,Q\right\rangle _{L^{2}}=\int_{M}\left\langle P,Q\right\rangle 
\func{vol}  \label{OctoL2ip}
\end{equation}%
where $\func{vol}$ is the standard volume form on $M$ defined by the metric
and the orientation and $\left\langle \cdot ,\cdot \right\rangle $ is the
canonical extension of $g$ from $\Lambda ^{p}\left( T^{\ast }M\right) $ to $%
\Lambda ^{p}\left( T^{\ast }M\right) \otimes \mathbb{O}M$. Using (\ref%
{OctoL2ip}), we then define the adjoint operator to $d_{D}$ - the
codifferential $d_{D}^{\ast }$. Let $P$ be an $\mathbb{O}M$-valued $p$-form,
and $Q$ an $\mathbb{O}M$-valued $\left( p-1\right) $-form, then 
\begin{equation}
\left\langle d_{D}^{\ast }P,Q\right\rangle _{L^{2}}=\left\langle
P,d_{D}Q\right\rangle _{L^{2}}  \label{Dcodiff1}
\end{equation}%
The codifferential $d_{D}^{\ast }$ is then a map 
\begin{equation}
d_{D}^{\ast }:\Omega ^{p}\left( \mathbb{O}M\right) \longrightarrow \Omega
^{p-1}\left( \mathbb{O}M\right)
\end{equation}

A direct computation using Stokes' Theorem and Lemma \ref{lemRLmult}, gives 
\begin{eqnarray*}
\left\langle P,d_{D}Q\right\rangle _{L^{2}} &=&\int_{M}\frac{p}{p!}%
\left\langle P^{b_{1}...b_{p}},\nabla _{b_{1}}Q_{b_{2}...b_{p}}-\left(
-1\right) ^{p-1}Q_{b_{1}...b_{p-1}}T_{b_{p}}\right\rangle _{\mathbb{O}}\func{%
vol} \\
&=&-\frac{1}{\left( p-1\right) !}\int \left\langle \nabla
_{b_{1}}P^{b_{1}b_{2}...b_{p}}+\left( -1\right) ^{p-1}P^{b_{2}...b_{p}b_{1}}%
\bar{T}_{b_{1}},Q_{b_{2}...b_{p}}\right\rangle _{\mathbb{O}}\func{vol} \\
&=&-\frac{1}{\left( p-1\right) !}\int \left\langle \nabla
_{b_{1}}P^{b_{1}b_{2}...b_{p}}-P^{b_{1}b_{2}...b_{p}}T_{b_{1}},Q_{b_{2}...b_{p}}\right\rangle _{%
\mathbb{O}}\func{vol} \\
&=&\left\langle d_{D}^{\ast }P,Q\right\rangle _{L^{2}}
\end{eqnarray*}%
Therefore, 
\begin{equation*}
\left( d_{D}^{\ast }P\right) _{b_{2}...b_{p}}=-\left( \nabla _{b_{1}}P_{\
b_{2}...b_{p}}^{b_{1}}-P_{\ b_{2}...b_{p}}^{b_{1}}T_{b_{1}}\right)
\end{equation*}%
In particular, 
\begin{equation}
\left( d_{D}^{\ast }P\right) _{b_{2}...b_{p}}=-D_{b_{1}}P_{\
b_{2}..b_{p}}^{b_{1}}.  \label{ddsformula}
\end{equation}%
We will thus define the \emph{divergence }of a $p$-form \ $P$ with respect
to $D$ as the $\left( p-1\right) $-form $\func{Div}P,$ given by 
\begin{equation}
\left( \func{Div}P\right) _{b_{2}...b_{p}}=-D_{b_{1}}P_{\
b_{2}..b_{p}}^{b_{1}}  \label{DivPdef}
\end{equation}%
and thus, 
\begin{equation*}
d_{D}^{\ast }P=-\func{Div}P.
\end{equation*}%
This is the complete analog of the standard codifferential $d^{\ast }$ being
equal to $-\func{div},$ where divergence is now with respect to $\nabla .$
Let us now consider the derivatives of the torsion $T$, which we know is an $%
\func{Im}\mathbb{O}M$-valued $1$-form. Applying the definition (\ref{octocov}%
) of $D,$ we obtain%
\begin{equation*}
D_{i}T_{j}=\nabla _{i}T_{j}-T_{j}T_{i}
\end{equation*}%
Expanding the octonion product $T_{j}T_{i},$ we then have 
\begin{eqnarray}
D_{i}T_{j} &=&\nabla _{i}T_{j}-T_{j}\times T_{i}+\left\langle
T_{j},T_{i}\right\rangle _{\mathbb{O}}  \notag \\
&=&\nabla _{i}T_{j}+T_{i}\times T_{j}+\left\langle T_{i},T_{j}\right\rangle
_{\mathbb{O}}  \label{DTors}
\end{eqnarray}%
We then use (\ref{DTors}) to find $d_{D}T$ and $d_{D}^{\ast }T$.

\begin{proposition}
\label{PropTorsDeriv} Suppose the octonion product on $\mathbb{O}M$ is
defined by the $G_{2}$-structure $\varphi $ with torsion $T$. Then, 
\begin{eqnarray}
d_{D}T &=&\frac{1}{4}\left( \pi _{7}\func{Riem}\right)  \label{extDT} \\
d_{D}^{\ast }T &=&-\func{div}T-\left\vert T\right\vert ^{2}  \label{dsT}
\end{eqnarray}%
where $\pi _{7}\func{Riem}\in \Omega ^{2}\left( \func{Im}\mathbb{O}M\right) $
$\cong \Omega ^{2}\left( TM\right) $- a vector-valued $2$-form$.$ Similarly, 
$\func{div}T\in \Omega ^{0}\left( \func{Im}\mathbb{O}M\right) $ and $%
\left\vert T\right\vert ^{2}\in \Omega ^{0}\left( \func{Re}\mathbb{O}%
M\right) $.
\end{proposition}

\begin{proof}
To obtain the exterior derivative $d_{D}T,$ we skew-symmetrize (\ref{DTors}):%
\begin{equation}
\left( d_{D}T\right) _{ij}=2\left( \nabla _{\lbrack i}T_{j]}+T_{i}\times
T_{j}\right)  \label{extDT1}
\end{equation}%
So far we have considered $T$ as an $\func{Im}\mathbb{O}M$-valued $1$-form,
and have suppressed the octonion index on $T$. Writing out (\ref{extDT1}) in
full, we have%
\begin{equation}
\left( d_{D}T\right) _{ij}^{\ \alpha }=2\left( \nabla _{\lbrack i}T_{j]}^{\
\ \alpha }+T_{i}^{\ \beta }T_{j}^{\ \gamma }\varphi _{\ \beta \gamma
}^{\alpha }\right)  \label{extDt2}
\end{equation}%
However using the Bianchi identity for $G_{2}$ torsion (Proposition \ref%
{propTorsBianchi}), we see that the right hand side of (\ref{extDt2}) is
precisely $\frac{1}{4}\left( \pi _{7}\func{Riem}\right) :$%
\begin{equation*}
2\left( \nabla _{\lbrack i}T_{j]}^{\ \ \alpha }+T_{i}^{\ \beta }T_{j}^{\
\gamma }\varphi _{\ \beta \gamma }^{\alpha }\right) =\frac{1}{4}\func{Riem}%
_{ij}^{\ \ \beta \gamma }\varphi _{\ \ \beta \gamma }^{\alpha }
\end{equation*}%
Therefore, 
\begin{equation*}
\left( d_{D}T\right) _{ij}^{\ \alpha }=\frac{1}{4}\func{Riem}_{ij}^{\ \
\beta \gamma }\varphi _{\ \ \beta \gamma }^{\alpha }=\frac{1}{4}\left( \pi
_{7}\func{Riem}\right) _{ij}^{\ \ \alpha }.
\end{equation*}%
To find $d_{D}^{\ast }T$, we write 
\begin{equation*}
d_{D}^{\ast }T=-D^{i}T_{i}=-\func{div}T-\left\vert T\right\vert ^{2}.
\end{equation*}
\end{proof}

In particular, using Proposition \ref{PropTorsDeriv} and Lemma \ref{lemd2nap}%
, we can now work out the action of $d_{D}^{2}$ on octonion-valued forms.

\begin{proposition}
\label{propdD2}Suppose $P\in \Omega ^{p}\left( \mathbb{O}M\right) $. Then, 
\begin{eqnarray}
d_{D}^{2}P &=&\func{Riem}\wedge \left( \func{Im}P\right) -P\overset{\circ }{%
\wedge }d_{D}T  \notag \\
&=&\func{Riem}\wedge \left( \func{Im}P\right) -\frac{1}{4}P\overset{\circ }{%
\wedge }\left( \pi _{7}\func{Riem}\right)  \label{d2D}
\end{eqnarray}
\end{proposition}

\begin{proof}
From the definition of $d_{D}$ (\ref{Dextdiff2}), we have 
\begin{equation*}
d_{D}P=d_{\nabla }P-\left( -1\right) ^{p}P\overset{\circ }{\wedge }T
\end{equation*}%
Hence, 
\begin{eqnarray*}
d_{D}^{2}P &=&d_{D}\left( d_{\nabla }P\right) -\left( -1\right)
^{p}d_{D}\left( P\overset{\circ }{\wedge }T\right) \\
&=&d_{\nabla }^{2}P-\left( -1\right) ^{p+1}d_{\nabla }P\overset{\circ }{%
\wedge }T-\left( -1\right) ^{p}\left( d_{\nabla }P\right) \overset{\circ }{%
\wedge }T-\left( -1\right) ^{2p}P\overset{\circ }{\wedge }d_{D}T \\
&=&d_{\nabla }^{2}P-P\overset{\circ }{\wedge }d_{D}T
\end{eqnarray*}%
The expression (\ref{d2D}) then follows when we use Lemma \ref{lemd2nap} to
rewrite $d_{\nabla }^{2}$ and Proposition \ref{PropTorsDeriv} to rewrite $%
d_{D}T$.
\end{proof}

In particular, if $P=\left( p_{0},p\right) \in \Omega ^{0}\left( \mathbb{O}%
M\right) $, i.e. a section of $\Gamma \left( \mathbb{O}M\right) $ then (\ref%
{d2D}) gives us: 
\begin{equation}
\left( d_{D}^{2}P\right) =\func{Riem}\left( p\right) -\frac{1}{4}P\circ
\left( \pi _{7}\func{Riem}\right)  \label{d2D0}
\end{equation}

\begin{remark}
In the expression (\ref{d2D}) we see that $d_{D}^{2}P$ has two components - $%
\func{Riem}\wedge \left( \func{Im}P\right) ,$ which comes from $d_{\nabla
}^{2}$ and does not depend on the octonion product, and an \textquotedblleft
octonionic\textquotedblright\ part $-P\overset{\circ }{\wedge }d_{D}T$ which
is fully determined by the torsion and involves octonion multiplication.
This gives $\pi _{7}\func{Riem}$ a new interpretation as an octonionic
curvature. In particular, $\pi _{7}\func{Riem}$ completely determines the
real part of $d_{D}^{2}$:%
\begin{equation}
\func{Re}\left( d_{D}^{2}P\right) =\frac{1}{4}\left( -p_{0}\wedge \pi _{7}%
\func{Riem}+\left\langle \rho \wedge \pi _{7}\func{Riem}\right\rangle _{%
\mathbb{O}}\right)  \label{Red2D}
\end{equation}%
where in the second term of (\ref{Red2D}), the octonion inner product is
combined with the wedge product of differential forms.
\end{remark}

\section{Change of reference $G_{2}$-structure}

\setcounter{equation}{0}\label{secDeform}In the previous section we have
considered the octonion covariant derivative with respect to a fixed $G_{2}$%
-structure $\varphi $. However from Section \ref{secIsomG2}, we know that
any nowhere-vanishing octonion section $V$ defines a new $G_{2}$-structure $%
\sigma _{V}\left( \varphi \right) $ in the same metric class. Recall from (%
\ref{OctoVAB}) that the octonion product defined by $\sigma _{V}\left(
\varphi \right) $ is given by 
\begin{equation}
A\circ _{V}B=AB+\left[ A,B,V\right] V^{-1}=\left( AV\right) \left(
V^{-1}B\right)
\end{equation}%
We can work out $\nabla _{X}\left( A\circ _{V}B\right) $ directly.

\begin{lemma}
\label{lemNabAVB}Let $V$ be a nowhere-vanishing octonion section, and
suppose $\circ _{V}$ is the octonion product defined by the $G_{2}$%
-structure $\sigma _{V}\left( \varphi \right) $ and $\left[ \cdot ,\cdot
,\cdot \right] _{V}$ is the corresponding associator. Then, for any $A,B\in
\Gamma \left( \mathbb{O}M\right) ,$ and any vector field $X,$%
\begin{equation}
\nabla _{X}\left( A\circ _{V}B\right) =\left( \nabla _{X}A\right) \circ
_{V}B+A\circ _{V}\nabla _{X}B-\left[ \func{Ad}_{V}T_{X}+V\left( \nabla
_{X}V^{-1}\right) ,A,B\right] _{V}  \label{NabAVB}
\end{equation}
\end{lemma}

The proof of Lemma \ref{lemNabAVB} is straightforward, but technical, so we
give it in Appendix \ref{appProofs}. However, from (\ref{nablaXAB}), the
torsion $T^{\left( V\right) }$ of $\sigma _{V}\left( \varphi \right) $ is
given by 
\begin{equation}
\nabla _{X}\left( A\circ _{V}B\right) =\left( \nabla _{X}A\right) \circ
_{V}B+A\circ _{V}\left( \nabla _{X}B\right) -\left[ T_{X}^{V},A,B\right] _{V}
\label{TorsV1}
\end{equation}%
for any octonion sections $A$, $B$, and vector field $X$. Comparing (\ref%
{NabAVB}) and (\ref{TorsV1}), we therefore have the following result.

\begin{theorem}
\label{ThmTorsV}Let $M$ be a smooth $7$-dimensional manifold with a $G_{2}$%
-structure $\left( \varphi ,g\right) $ with torsion $T\in \Omega ^{1}\left( 
\func{Im}\mathbb{O}M\right) $. For a nowhere-vanishing $V\in \Gamma \left( 
\mathbb{O}M\right) ,$ consider the $G_{2}$-structure $\sigma _{V}\left(
\varphi \right) .$ Then, the torsion $T^{\left( V\right) }$ of $\sigma
_{V}\left( \varphi \right) $ is given by%
\begin{equation}
T^{\left( V\right) }=\func{Im}\left( \func{Ad}_{V}T+V\left( \nabla
V^{-1}\right) \right)  \label{TorsV2}
\end{equation}%
In particular, if $V$ has constant norm, $T^{\left( V\right) }$ is given by 
\begin{equation}
T^{\left( V\right) }=-\left( DV\right) V^{-1}  \label{TorsV3}
\end{equation}
\end{theorem}

\begin{proof}
This is a direct consequence of Lemma \ref{lemNabAVB}. Since (\ref{NabAVB})
is defined for arbitrary $A,B$, by comparing it with (\ref{TorsV1}), we find
that the imaginary parts of $T^{\left( V\right) }$ and $\func{Ad}%
_{V}T+V\left( \nabla V^{-1}\right) $ must agree. However, by definition, $%
T^{\left( V\right) }\in \Omega ^{1}\left( \func{Im}\mathbb{O}M\right) $ is
pure imaginary, so (\ref{TorsV2}) holds. Note that in general $\func{Ad}%
_{V}T+V\left( \nabla V^{-1}\right) $ has a real part:%
\begin{eqnarray}
\func{Re}\left( \func{Ad}_{V}T+V\left( \nabla V^{-1}\right) \right)
&=&\left\langle \func{Ad}_{V}T+V\left( \nabla V^{-1}\right) ,1\right\rangle 
\notag \\
&=&\left\langle V\left( \nabla V^{-1}\right) ,1\right\rangle  \notag \\
&=&-\left\langle \left( \nabla V\right) V^{-1},1\right\rangle  \notag \\
&=&-\frac{1}{\left\vert V\right\vert ^{2}}\left\langle \nabla
V,V\right\rangle  \notag \\
&=&-\frac{1}{2}\frac{1}{\left\vert V\right\vert ^{2}}\nabla \left\vert
V\right\vert ^{2}=-\nabla \ln \left\vert V\right\vert  \label{ReTorsV}
\end{eqnarray}%
In particular, if $\left\vert V\right\vert $ is constant, then the real part
vanishes, and hence 
\begin{eqnarray*}
T^{\left( V\right) } &=&\func{Ad}_{V}T+V\left( \nabla V^{-1}\right) \\
&=&VTV^{-1}-\left( \nabla V\right) V^{-1} \\
&=&-\left( \nabla V-VT\right) V^{-1} \\
&=&-\left( DV\right) V^{-1}.
\end{eqnarray*}
\end{proof}

\begin{remark}
Theorem \ref{ThmTorsV} shows that under a transformation of $G_{2}$%
-structures given by $\varphi \longrightarrow \sigma _{V}\left( \varphi
\right) $, the torsion $1$-form transforms in a way similar to the
transformation of a principal bundle connection $1$-form under a change of
trivialization. Since in our case $T$ is not the full connection $1$-form -
it is only part of the connection that also includes the Levi-Civita
connection, the transformation rule (\ref{TorsV2}) involves $\nabla $ rather
than standard partial derivatives as one has on a principal bundle. As
Corollary \ref{corrTorsV0} shows below, the existence of torsion-free $G_{2}$%
-structures in a given metric class reduces to solving the equation $DV=0$
for some nowhere-zero $V\in \Gamma \left( \mathbb{O}M\right) .$
\end{remark}

\begin{corollary}
\label{corrTorsV0}Let $M$ be a smooth $7$-dimensional manifold with a $G_{2}$%
-structure $\left( \varphi ,g\right) $ with torsion $T\in \Omega ^{1}\left( 
\func{Im}\mathbb{O}M\right) $. There exists a torsion-free $G_{2}$-structure 
$\tilde{\varphi}$ in the same metric class as $\varphi $ if and only if
there exists a nowhere-zero octonion section $V$ such that 
\begin{equation}
DV=0  \label{DV0}
\end{equation}%
where the covariant derivative $D$ is defined by (\ref{octocov}) using the
torsion $1$-form $T$. Moreover, $\tilde{\varphi}=\sigma _{V}\left( \varphi
\right) $.
\end{corollary}

\begin{proof}
If $DV=0$, then using the metric compatibility of $D$ from Proposition \ref%
{propDXmetric}, we get $\nabla \left\vert V\right\vert =0.$ Hence $%
\left\vert V\right\vert $ is a non-zero constant. Therefore, from Theorem %
\ref{ThmTorsV}, the torsion of $\sigma _{V}\left( \varphi \right) $ is given
by 
\begin{equation*}
T^{\left( V\right) }=-\left( DV\right) V^{-1}=0.
\end{equation*}%
Conversely, suppose there exists a torsion-free $G_{2}$-structure $\tilde{%
\varphi}$ in the same metric class as $\varphi $. Then, by Theorem \ref%
{ThmSamegfam}, $\ \tilde{\varphi}=\sigma _{U}\left( \varphi \right) $ for
some nowhere-vanishing octonion section $U$. By Theorem \ref{ThmTorsV}, 
\begin{equation*}
T^{\left( U\right) }=\func{Im}\left( \func{Ad}_{U}T+U\nabla U^{-1}\right) =0
\end{equation*}%
However, let $V=\frac{U}{\left\vert U\right\vert },$ so that $V$ is a unit
octonion. Then, 
\begin{eqnarray*}
\func{Ad}_{U}T &=&\func{Ad}_{V}T \\
U\nabla U^{-1} &=&-\left( \nabla U\right) U^{-1} \\
&=&-\nabla \left( V\left\vert U\right\vert \right) \left\vert U\right\vert
^{-1}V^{-1}=-\frac{1}{\left\vert U\right\vert }\nabla \left\vert
U\right\vert -\left( \nabla V\right) V^{-1}
\end{eqnarray*}%
But, $\frac{1}{\left\vert U\right\vert }\nabla \left\vert U\right\vert $ is
real, so 
\begin{equation*}
0=\func{Im}\left( \func{Ad}_{U}T+U\nabla U^{-1}\right) =\func{Im}\left( 
\func{Ad}_{V}T-\left( \nabla V\right) V^{-1}\right) =-\left( DV\right) V^{-1}
\end{equation*}%
Thus $DV=0$.
\end{proof}

An interesting special case of Corollary \ref{corrTorsV0} is when $\varphi $
is already torsion-free. In this case, $D=\nabla $, hence the condition (\ref%
{DV0}) becomes simply $\nabla V=0.$ Moreover, now the real and pure
imaginary parts of $V$ are differentiated separately, so we just require $%
\nabla v=0$ for some vector field $v$ on $M$. Given a unit parallel vector
field $v,$ any unit octonion $V=\left( a,bv\right) ,$ for constants $a$ and $%
b$ such that $a^{2}+b^{2}=1$, will define a torsion-free $G_{2}$-structure.
This shows that all the other torsion-free $G_{2}$-structures in the same
metric class as $\varphi $ are parametrized by parallel vector fields on $M$
together with a choice of a phase factor.

\begin{definition}
Let $\mathcal{F}_{g}$ be the space of torsion-free $G_{2}$-structures that
are compatible with the metric $g.$
\end{definition}

\begin{theorem}
\label{ThmFgBetti}Suppose $\left( M,g\right) $ is a $7$-dimensional
Riemannian manifold with $\func{Hol}\left( g\right) \subseteq G_{2}$. If $M$
admits $m$ linearly independent parallel vectors, then 
\begin{equation}
\mathcal{F}_{g}\cong \mathbb{R}P^{m}  \label{FgRP}
\end{equation}%
If moreover, $M$ is compact, then $\mathcal{F}_{g}\cong \mathbb{R}P^{b^{1}}$%
, where $b^{1}$ is the first Betti number. In particular, if $b^{1}=0$, the
torsion-free $G_{2}$-structure in this metric class is unique
\end{theorem}

\begin{proof}
Suppose $\varphi $ is a torsion-free $G_{2}$-structure on $M$ that is
compatible with $g$. From Corollary \ref{corrTorsV0}, we know that any
torsion-free $G_{2}$-structure $\tilde{\varphi}$ in the same metric class is
given by $\tilde{\varphi}=\sigma _{V}\left( \varphi \right) $ for an
octonion $V$ with $\nabla V=0$. This is equivalent to $\nabla v=0$ where $v=%
\func{Im}V$ is a vector field on $M$. If $M$ does admit parallel vector
fields, then any parallel vector field $v$ defines a torsion-free $G_{2}$%
-structure. Suppose now $M$ admits $m$\thinspace $>0$ linearly independent
parallel vectors $v_{1},...,v_{m}$. We can then define a global $\left(
m+1\right) $-subframe on $\mathbb{O}M$ spanned by $V_{0}=\left( 1,0\right) $
and $V_{i}=\left( 0,v_{i}\right) $ for $1\leq i\leq m$.

Therefore, any octonion in the space spanned by $\left\langle
V_{k}\right\rangle $ for $0\leq k\leq m$ defines a torsion-free $G_{2}$%
-structure. However, as we see from the definition of the map $\sigma _{V}$ (%
\ref{sigmaAdef}), any constant multiple of $V_{k}$ defines the same $G_{2}$%
-structure. Therefore, the torsion-free $G_{2}$-structures are in a 1-1
correspondence with projective lines in $\left\langle V_{k}\right\rangle .$
Hence $\mathcal{F}_{g}\cong \mathbb{R}P^{m}$.

If $M$ is compact, then it is a standard fact that, since $\func{Ric}=0$,
parallel vector fields are in a 1-1 correspondence with harmonic forms, and
thus $m=b^{1}$. If $b^{1}=0$, then there exist no parallel vectors on $M$,
and therefore $\varphi $ is the unique torsion-free $G_{2}$-structure that
is compatible with the metric $g$.
\end{proof}

\begin{remark}
If a compact Riemannian 7-manifold $\left( M,g\right) $ has $\func{Hol}%
\left( g\right) \subseteq G_{2}$, then the only possible values for $b^{1}$
are $0,1,3,7$ \cite{Joycebook} and the number of linearly independent
parallel vector fields on $M$ is equal to $b^{1}$. However even if $M$ is
non-compact, the number of parallel vectors $m$ can also be only $0,1,3,7$,
even if $m\neq b^{1}$\ (in fact $m\leq b^{1}$)$.$ This is easy to see -
suppose we have two orthogonal parallel vector fields, $v_{1}$ and $v_{2}$.
Then, using a torsion-free $G_{2}$-structure $\varphi $, we can define $%
v_{3}=v_{1}\times _{\varphi }v_{2}$. This will also be parallel (since $%
\varphi $ is parallel), and it will be orthogonal to both $v_{1}$ and $v_{2}$%
. Thus once we have at least 2 parallel vector fields, we must actually have 
$3$. Similarly, if we have at least $4$ orthogonal parallel vector fields,
by considering cross products, we find that we actually must have $7$.
Another way of looking at this is that if we have have no parallel vector
fields, then the octonion bundle globally splits only as $\mathbb{O}\cong 
\mathbb{R}\oplus \func{Im}\mathbb{O}$ where $\func{Im}\mathbb{O}$
corresponds to $TM$. If we have one parallel vector field, then we in fact
have a globally-defined \emph{complex plane} inside the octonion bundle, so
now we get a splitting as $\mathbb{O\cong C\oplus }V^{6}$ where $V^{6}$ now
corresponds to the tangent bundle of Calabi-Yau $3$-fold. Further, if we
have three parallel vector fields, together with the octonion $\left(
1,0\right) $ these form a globally defined \emph{quaternion} subspace, so
that $\mathbb{O\cong H}\oplus V^{4},$ where $V^{4}$ now corresponds to the
tangent bundle of a 4 real dimensional Hyperk\"{a}hler manifold. Finally, if
we have a global frame of parallel vector fields, then the octonion bundle
just becomes a direct product of $\mathbb{R}^{7}$ and the standard octonion
algebra.
\end{remark}

The condition (\ref{DV0}) for existence of a torsion-free $G_{2}$-structure
depends on the initial choice of the $G_{2}$-structure $\varphi $. This is
the reference $G_{2}$-structure. However if we chose a different reference $%
G_{2}$-structure, the covariant derivative $D$ would be defined differently.
As we have seen from the deformation of the torsion, we can interpret the
choice of the reference $G_{2}$-structure as a choice of trivialization.
Therefore we need to understand whether the condition (\ref{DV0}) is
invariant under a change of trivialization.

\begin{proposition}
\label{propDV}Suppose $\left( \varphi ,g\right) $ is a $G_{2}$-structure on
a $7$-manifold $M$, with torsion $T$ and corresponding octonion covariant
derivative $D$. Suppose $V$ is a unit octonion section, and $\tilde{\varphi}%
=\sigma _{V}\left( \varphi \right) $ is the corresponding $G_{2}$-structure,
that has torsion $\tilde{T}$, given by (\ref{TorsV3}), and an octonion
covariant derivative $\tilde{D}$. Then, for any octonion section $A$, we
have 
\begin{equation}
\tilde{D}\left( AV^{-1}\right) =\left( DA\right) V^{-1}  \label{DtildeAV}
\end{equation}%
and equivalently, 
\begin{equation}
\tilde{D}A=\left( D\left( AV\right) \right) V^{-1}  \label{DtildeAV2}
\end{equation}
\end{proposition}

\begin{proof}
Since $\left\vert V\right\vert =1,$ recall from Theorem \ref{ThmTorsV}, that 
$\tilde{T}$ is given by 
\begin{equation}
\tilde{T}=VTV^{-1}+V\left( \nabla V^{-1}\right)  \label{propDVTtilde}
\end{equation}%
Hence, 
\begin{eqnarray}
\tilde{D}\left( AV^{-1}\right) &=&\nabla \left( AV^{-1}\right) -\left(
AV^{-1}\right) \circ _{V}\tilde{T}  \notag \\
&=&\nabla \left( AV^{-1}\right) -\left( AV^{-1}V\right) \left( V^{-1}\tilde{T%
}\right)  \label{propDVdavin}
\end{eqnarray}%
where we have used the expression (\ref{OctoVAB}) for $\circ _{V}.$
Substituting (\ref{propDVTtilde}) into (\ref{propDVdavin}) and using
Proposition \ref{propOctoLC} to expand the first term in (\ref{propDVdavin}%
), we obtain 
\begin{eqnarray}
\tilde{D}\left( AV^{-1}\right) &=&\nabla \left( AV^{-1}\right) -A\left(
TV^{-1}+\nabla V^{-1}\right)  \notag \\
&=&\left( \nabla A\right) V^{-1}+A\left( \nabla V^{-1}\right) -\left[
T,A,V^{-1}\right] -A\left( TV^{-1}\right) -A\left( \nabla V^{-1}\right) 
\notag \\
&=&\left( \nabla A\right) V^{-1}+\left[ A,T,V^{-1}\right] -A\left(
TV^{-1}\right)  \notag \\
&=&\left( \nabla A-AT\right) V^{-1}=\left( DA\right) V^{-1}
\end{eqnarray}%
The equivalent expression (\ref{DtildeAV2}) follows immediately.
\end{proof}

\begin{remark}
From Proposition \ref{propDV} we also obtain that the octonion covariant
exterior derivative (\ref{Dextdiff}) also transforms in a similar way, and
hence, for $P\in \Omega ^{0}\left( \mathbb{O}M\right) $, we see that 
\begin{eqnarray*}
\tilde{d}_{D}^{2}P &=&\tilde{d}_{D}\left( d_{D}\left( PV\right) V^{-1}\right)
\\
&=&\left( d_{D}^{2}\left( PV\right) \right) V^{-1}
\end{eqnarray*}%
Therefore, if $F$ and $\tilde{F}$ are the curvature operators for $G_{2}$%
-structures $\varphi $ and $\sigma _{V}\left( \varphi \right) $,
respectively, then we see that 
\begin{equation}
\tilde{F}=R_{V^{-1}}FR_{V}  \label{curvtransform}
\end{equation}%
where $R_{V}$ is the right multiplication map by $V$ and the right hand side
of (\ref{curvtransform}) is a composition of maps. Further study of the
algebraic properties of these octonion structures is needed to understand
how to interpret (\ref{curvtransform}) and whether there are any invariants
of the octonion curvature with respect to the transformation (\ref%
{curvtransform}).
\end{remark}

From (\ref{DtildeAV}), we see that $DA=0$ with respect to the $G_{2}$%
-structure $\varphi $ if and only if $\tilde{D}\left( AV^{-1}\right) =0$
with respect to the $G_{2}$-structure $\sigma _{V}\left( \varphi \right) $.
Therefore, the solutions of the equation $DA=0$ are in a 1-1 correspondence
with solutions of the equation $\tilde{D}A=0$.

Suppose $\tilde{\varphi}=\sigma _{V}\left( \varphi \right) $, so that $%
\varphi =\sigma _{V^{-1}}\left( \tilde{\varphi}\right) $. Then, using
Theorem \ref{ThmSigmaUV}, 
\begin{eqnarray}
\sigma _{A}\varphi &=&\sigma _{A}\sigma _{V^{-1}}\tilde{\varphi}  \notag \\
&=&\sigma _{AV^{-1}}\tilde{\varphi}
\end{eqnarray}%
The torsion $T^{\left( A\right) }$ of $\sigma _{A}\varphi $ is then given by 
$T^{\left( A\right) }=-\left( DA\right) A^{-1}$ with respect to $G_{2}$%
-structure $\varphi ,$ but with respect to the $G_{2}$-structure $\tilde{%
\varphi}$, the expression is 
\begin{eqnarray*}
T^{\left( A\right) } &=&-\tilde{D}\left( AV^{-1}\right) \circ _{V}\left(
AV^{-1}\right) ^{-1} \\
&=&-\tilde{D}\left( AV^{-1}\right) \circ _{V}\left( VA^{-1}\right) \\
&=&-\left( \left( DA\right) V^{-1}\right) \circ _{V}\left( VA^{-1}\right) \\
&=&-\left( DA\right) A^{-1}
\end{eqnarray*}%
where we have used (\ref{DtildeAV}) and the definition of $\circ _{V}$ (\ref%
{OctoVAB}). Therefore, the two descriptions give the same result and hence
the torsion is well-defined.

\begin{remark}
This shows that the \textquotedblleft covariant\textquotedblright\
derivative $D$ is indeed \emph{covariant }under change of trivialization. In
particular, if we have different choices of the reference $G_{2}$-structure,
say $\varphi $ and $\tilde{\varphi}$, the octonion solutions of $DA=0$ and $%
\tilde{D}\tilde{A}=0$ will be different, however they still define the same $%
G_{2}$-structures under the map $\sigma .$ That is, $\sigma _{A}\left(
\varphi \right) =\sigma _{\tilde{A}}\left( \tilde{\varphi}\right) $.
\end{remark}

The equation $DA=0$ is linear, but it is an overdetermined PDE, since
pointwise, we have 56 equations for 8 variables of $A$. If we however
restrict to compact manifolds, we can show that this is equivalent to an
elliptic equation $D^{2}A=0$ where we define $D^{2}=D_{i}D^{i}$. From the
definition of $D$ this is clearly elliptic, since the leading term is just
the ordinary rough Laplacian $\nabla ^{2}=\nabla _{i}\nabla ^{i}$.

\begin{proposition}
\label{propDAD2A}Suppose $\left( \varphi ,g\right) $ is a $G_{2}$-structure
on a compact $7$-dimensional manifold $M$. Then, if $A\in C^{2}\left( 
\mathbb{O}M\right) $ is a twice differentiable octonion section, $DA=0$ if
and only if $D^{2}A=0.$
\end{proposition}

\begin{proof}
Clearly, if $DA=0,$ then $D^{2}A$ also vanishes. Conversely, suppose $A$ is
an octonion section that satisfies $D^{2}A=0$. Using the compatibility of $D$
with the metric (\ref{DXmet}), we have the following identity:%
\begin{eqnarray}
\nabla ^{2}\left\vert A\right\vert ^{2} &=&2\nabla ^{i}\left( \left\langle
D_{i}A,A\right\rangle _{\mathbb{O}}\right)  \notag \\
&=&2\left( \left\langle D^{2}A,A\right\rangle _{\mathbb{O}}+\left\vert
DA\right\vert ^{2}\right)  \label{nabsqnormA}
\end{eqnarray}%
However, if $D^{2}A=0$, we get 
\begin{equation}
\nabla ^{2}\left\vert A\right\vert ^{2}=2\left\vert DA\right\vert ^{2}
\label{nabsqnormA2}
\end{equation}%
Therefore, $\nabla ^{2}\left\vert A\right\vert ^{2}\geq 0$ on $M,$ and by
the Weak Maximum Principle, $\left\vert A\right\vert ^{2}$ must be constant,
and hence $\left\vert DA\right\vert ^{2}=0$ everywhere on $M$.
\end{proof}

Combining Corollary \ref{corrTorsV0} and Proposition \ref{propDAD2A}, we
thus have the following important result.

\begin{theorem}
\label{ThmD2V0}Let $M$ be a smooth $7$-dimensional compact manifold with a $%
G_{2}$-structure $\left( \varphi ,g\right) $ with torsion $T\in \Omega
^{1}\left( \func{Im}\mathbb{O}M\right) $. There exists a torsion-free $G_{2}$%
-structure $\tilde{\varphi}$ in the same metric class as $\varphi $ if and
only if there exists a non-zero $V\in C^{2}\left( \mathbb{O}M\right) $ that
satisfies the linear elliptic PDE 
\begin{equation}
D^{2}V=0  \label{D2Veq}
\end{equation}%
where the covariant derivative $D$ is defined by (\ref{octocov}) using the
torsion $1$-form $T$. Moreover, $\tilde{\varphi}=\sigma _{V}\left( \varphi
\right) $.
\end{theorem}

\begin{remark}
Given a fixed metric $g$ on a compact manifold $M$, it is well-known that a $%
G_{2}$-structure $\varphi $ in the metric class of $g$ is torsion-free if
and only if $\Delta \varphi =0$. Although this is a linear elliptic PDE for
the $3$-form $\varphi ,$ there is an extra non-linear condition to make sure
that $\varphi $ is actually compatible with $g$. Equation (\ref{D2Veq}) is
also linear and elliptic, but it does not require any further conditions.
Since any non-zero solution of (\ref{D2Veq}) will actually be a parallel
octonion (with respect to $D$), it will have constant norm, and will be
nowhere-vanishing. The trade-off here is that the equation (\ref{D2Veq})
depends on the choice of trivialization, but in a covariant way. This is to
be expected, because the existence of solutions to (\ref{D2Veq}) is a
property of the metric - the dimension of the kernel of $D^{2}$ tells us the
holonomy group of the metric.
\end{remark}

\section{Relationship to the spinor bundle}

\setcounter{equation}{0}\label{secSpinor}It is well-known (e.g. \cite%
{BaezOcto,HarveyBook}) that quaternions and octonions have a very close
relationship with spinors in $3,4$ and $7,8$ dimensions, respectively. In
particular, multiplication by imaginary octonions is equivalent to Clifford
multiplication on spinors in $7$ dimensions. More precisely, the \emph{%
enveloping algebra} of the octonion algebra is isomorphic to the Clifford
algebra in $7$ dimensions. The (left) enveloping algebra of $\mathbb{O}$
consists of left multiplication maps $L_{A}:V\longrightarrow AV$ for $A,V\in 
\mathbb{O},$ under composition \cite{SchaferNonassoc}. Similarly, a right
enveloping algebra may also be defined. Since the binary operation in the
enveloping algebra is defined to be composition, it is associative. For
octonions $A,B,V$ we thus have 
\begin{eqnarray*}
L_{A}L_{B}\left( V\right) +L_{B}L_{A}\left( V\right) &=&A\left( BV\right)
+B\left( AV\right) \\
&=&\left( AB\right) V+\left[ A,B,V\right] +\left( BA\right) V+\left[ B,A,V%
\right] \\
&=&\left( AB+BA\right) V
\end{eqnarray*}%
Hence, if $A,B$ are pure imaginary, then indeed, 
\begin{equation}
L_{A}L_{B}+L_{B}L_{A}=-2\left\langle A,B\right\rangle \func{Id}
\label{octoenvelopCliff}
\end{equation}%
which is the defining identity for a Clifford algebra. We see that while the
octonion algebra does give rise to the Clifford algebra, in the process we
lose the nonassociative structure, and hence the octonion algebra has more
structure than the corresponding Clifford algebra. Note that also due to
non-associativity of the octonions, in general that $L_{A}L_{B}\neq L_{AB}.$
In fact, we have the following relationship.

\begin{lemma}
\label{lemABC}Let $A,B,C$ be octonion sections, with multiplication defined
by the $G_{2}$-structure $\varphi $, then 
\begin{equation*}
A\left( BC\right) =\left( A\circ _{C}B\right) C
\end{equation*}%
where $\circ _{C}$ denotes octonion multiplication with respect to the $%
G_{2} $-structure $\sigma _{C}\left( \varphi \right) .$ In particular, 
\begin{equation*}
L_{A}L_{B}C=L_{A\circ _{C}B}C
\end{equation*}
\end{lemma}

\begin{proof}
We can write 
\begin{eqnarray*}
A\left( BC\right) &=&\left( AB\right) C+\left[ A,B,C\right] \\
&=&\left[ AB+\left[ A,B,C\right] C^{-1}\right] C \\
&=&\left( A\circ _{C}B\right) C
\end{eqnarray*}%
where we have used (\ref{OctoVAB}).
\end{proof}

Let $\mathcal{S}$ be the spinor bundle on the $7$-manifold $M.$ It is then
well-known that a nowhere-vanishing spinor on $M$ defines a $G_{2}$%
-structure via a bilinear expression involving Clifford multiplication. In
fact, given a unit norm spinor $\xi \in \Gamma \left( \mathcal{S}\right) ,$
we may define 
\begin{equation}
\varphi _{\xi }\left( \alpha ,\beta ,\gamma \right) =-\left\langle \xi
,\alpha \cdot \left( \beta \cdot \left( \gamma \cdot \xi \right) \right)
\right\rangle _{S}  \label{phixispin}
\end{equation}%
where $\cdot $ denotes Clifford multiplication, $\alpha ,\beta ,\gamma $ are
arbitrary vector fields and $\left\langle \cdot ,\cdot \right\rangle _{S}$
is the inner product on the spinor bundle. The next lemma shows that we get
exactly the same expression if we use the octonion representation of the
Clifford algebra.

\begin{lemma}
Let $\alpha ,\beta ,\gamma \in \func{Im}\Gamma \left( \mathbb{O}M\right) ,$
and suppose $V\in \Gamma \left( \mathbb{O}M\right) $ is a unit octonion
section. Then, 
\begin{equation}
\left( \sigma _{V}\varphi \right) \left( \alpha ,\beta ,\gamma \right)
=-\left\langle V,\alpha \left( \beta \left( \gamma V\right) \right)
\right\rangle _{\mathbb{O}}  \label{phiVocto}
\end{equation}
\end{lemma}

\begin{proof}
Using Lemma \ref{lemABC}, we have 
\begin{eqnarray*}
\alpha \left( \beta \left( \gamma V\right) \right) &=&\alpha \left( \left(
\beta \circ _{V}\gamma \right) V\right) \\
&=&\left( \alpha \circ _{V}\left( \beta \circ _{V}\gamma \right) \right) V
\end{eqnarray*}%
Hence, using the fact that $\left\vert V\right\vert =1$, 
\begin{eqnarray*}
\left\langle V,\alpha \left( \beta \left( \gamma V\right) \right)
\right\rangle _{\mathbb{O}} &=&\left\langle V,\left( \alpha \circ _{V}\left(
\beta \circ _{V}\gamma \right) \right) V\right\rangle _{\mathbb{O}} \\
&=&\left\langle 1,\alpha \circ _{V}\left( \beta \circ _{V}\gamma \right)
\right\rangle _{\mathbb{O}} \\
&=&-\left\langle \alpha ,\beta \circ _{V}\gamma \right\rangle _{\mathbb{O}}
\end{eqnarray*}%
Therefore, 
\begin{eqnarray*}
\left\langle V,\alpha \left( \beta \left( \gamma V\right) \right)
\right\rangle _{\mathbb{O}} &=&-\left\langle \alpha ,\beta \circ _{V}\gamma
\right\rangle _{\mathbb{O}} \\
&=&-\left( \sigma _{V}\varphi \right) \left( \alpha ,\beta ,\gamma \right)
\end{eqnarray*}
\end{proof}

The main difference between (\ref{phixispin}) and (\ref{phiVocto}) is that
the right hand side of (\ref{phixispin}) only depends on the Clifford
algebra (and hence only on the metric), while the right hand side of (\ref%
{phiVocto}) already assumes a choice of a reference $G_{2}$-structure.
Suppose the reference $G_{2}$-structure $\varphi _{\xi }$ is defined by a
unit norm spinor $\xi $ using (\ref{phixispin}). This choice of a reference $%
G_{2}$-structure then induces a correspondence between spinors and
octonions. Define the linear map $j_{\xi }:\Gamma \left( \mathcal{S}\right)
\longrightarrow \Gamma \left( \mathbb{O}M\right) $ by 
\begin{subequations}%
\label{jxiprops} 
\begin{eqnarray}
j_{\xi }\left( \xi \right) &=&1 \\
j_{\xi }\left( V\cdot \eta \right) &=&V\circ _{\varphi _{\xi }}j_{\xi
}\left( \eta \right)  \label{jxiprops2}
\end{eqnarray}%
\end{subequations}%
for any octonion $V$ and spinor $\eta $ and where $\circ _{\varphi _{\xi }}$%
denotes octonion multiplication with respect to the $G_{2}$-structure $%
\varphi _{\xi },$ and for $V=\left( v_{0},v\right) ,$ the Clifford product
is given by $V\cdot \eta =v_{0}\eta +v\cdot \eta .$ Now if $\eta =A\cdot \xi 
$ for some octonion section $A,$ then in (\ref{jxiprops2}) 
\begin{equation}
j_{\xi }\left( \eta \right) =A
\end{equation}%
Note that in $7$ dimensions, if we fix a nowhere-zero spinor $\xi ,$ then we
get a pointwise decomposition of $\mathcal{S\ }$as $\mathbb{R}\cdot \xi
\oplus \left\{ X\cdot \xi :X\in \mathbb{R}^{7}\right\} $ \cite%
{AgricolaSpinors,FriedrichNPG2}, so given any spinor $\eta $, we can write
it as $\eta =A\cdot \xi $ for some octonion section $A$. Therefore, the map $%
j_{\xi }$ is in fact pointwise an isomorphism of real vector spaces from
spinors to octonions.

\begin{lemma}
\label{lemjxi}The map $j_{\xi }$ respects the inner product. That is, for
spinors $\eta _{1}$ and $\eta _{2}$, 
\begin{equation}
\left\langle \eta _{1},\eta _{2}\right\rangle _{S}=\left\langle j_{\xi
}\left( \eta _{1}\right) ,j_{\xi }\left( \eta _{2}\right) \right\rangle _{%
\mathbb{O}}
\end{equation}
\end{lemma}

\begin{proof}
Suppose $\eta _{1}=V_{1}\cdot \xi $ and $\eta _{2}=V_{2}\cdot \xi $ where $%
V_{1}=\left( a_{1},v_{1}\right) $ and $V_{2}=\left( a_{2},v_{2}\right) $%
.Then, 
\begin{eqnarray*}
\left\langle \eta _{1},\eta _{2}\right\rangle _{S} &=&\left\langle
V_{1}\cdot \xi ,V_{2}\cdot \xi \right\rangle _{S} \\
&=&a_{1}a_{2}\left\vert \xi \right\vert ^{2}+\left\langle v_{1}\cdot \xi
,v_{2}\cdot \xi \right\rangle _{S} \\
&=&a_{1}a_{2}\left\vert \xi \right\vert ^{2}+\left\langle
v_{1},v_{2}\right\rangle \left\vert \xi \right\vert ^{2} \\
&=&\left\langle V_{1},V_{2}\right\rangle _{\mathbb{O}} \\
&=&\left\langle j_{\xi }\left( \eta _{1}\right) ,j_{\xi }\left( \eta
_{2}\right) \right\rangle _{\mathbb{O}}
\end{eqnarray*}%
where we have the property that the Clifford product is skew-adjoint with
respect to the spinor inner product \cite{HarveyBook}.
\end{proof}

Under the map $j_{\xi }$ we then have 
\begin{eqnarray*}
\varphi _{\xi }\left( \alpha ,\beta ,\gamma \right) &=&-\left\langle \xi
,\alpha \cdot \left( \beta \cdot \left( \gamma \cdot \xi \right) \right)
\right\rangle _{S} \\
&=&-\left\langle j_{\xi }\left( \xi \right) ,\alpha \left( \beta \left(
\gamma \left( j_{\xi }\left( \xi \right) \right) \right) \right)
\right\rangle _{\mathbb{O}} \\
&=&-\left\langle 1,\alpha \left( \beta \gamma \right) \right\rangle \\
&=&\left\langle \alpha ,\beta \gamma \right\rangle
\end{eqnarray*}%
as expected. Then, for $\eta =A\cdot \xi $, using Lemma \ref{lemjxi}, we get 
\begin{eqnarray*}
\varphi _{\eta }\left( \alpha ,\beta ,\gamma \right) &=&-\left\langle \eta
,\alpha \cdot \left( \beta \cdot \left( \gamma \cdot \eta \right) \right)
\right\rangle _{S} \\
&=&-\left\langle j_{\xi }\left( \eta \right) ,\alpha \left( \beta \left(
\gamma \left( j_{\xi }\left( \eta \right) \right) \right) \right)
\right\rangle _{\mathbb{O}} \\
&=&-\left\langle A,\alpha \left( \beta \left( \gamma A\right) \right)
\right\rangle
\end{eqnarray*}%
where the octonion multiplication is with respect to $\varphi _{\xi }.$
Using (\ref{phiVocto}) this then shows that 
\begin{equation}
\varphi _{A\cdot \xi }=\sigma _{A}\left( \varphi _{\xi }\right)
\label{phiAxisigmaA}
\end{equation}%
This shows that our Theorem \ref{ThmSigmaUV} on composition of $\sigma _{U}$
and $\sigma _{V}$ can be restated in terms of spinors.

\begin{corollary}
Let $\xi $ be a spinor of unit norm on a $7$-dimensional manifold and let $%
\varphi _{\xi }$ be the $G_{2}$-structure defined by (\ref{phixispin}).
Then, for any unit octonions $U,V$%
\begin{equation}
\varphi _{U\cdot \left( V\cdot \xi \right) }=\varphi _{\left( UV\right)
\cdot \xi }  \label{phiUVxi}
\end{equation}%
where the octonion product $UV$ on the right hand side is defined
unambiguously using $\varphi _{\xi }$ or $\varphi _{V\cdot \xi }$ (see
Remark \ref{remSigmaUV}).
\end{corollary}

\begin{proof}
Theorem \ref{ThmSigmaUV} tells us that 
\begin{equation*}
\sigma _{U}\left( \sigma _{V}\varphi _{\xi }\right) =\sigma _{UV}\varphi
_{\xi }
\end{equation*}%
However, from (\ref{phiAxisigmaA}), we get 
\begin{eqnarray*}
\sigma _{U}\left( \sigma _{V}\varphi _{\xi }\right) &=&\sigma _{U}\left(
\varphi _{V\cdot \xi }\right) =\varphi _{U\cdot \left( V\cdot \xi \right) }
\\
\sigma _{UV}\varphi _{\xi } &=&\varphi _{\left( UV\right) \cdot \xi },
\end{eqnarray*}%
which gives us (\ref{phiUVxi}).
\end{proof}

It is also well-known (e.g. \cite[Definition 4.2 and Lemma 4.3]%
{AgricolaSpinors}) that given the spinorial covariant derivative $\nabla
^{S} $ on $\mathcal{S},$ which is obtained by lifting the Levi-Civita
connection to $\mathcal{S},$ we get 
\begin{equation}
\nabla _{X}^{S}\xi =-T_{X}^{\left( \xi \right) }\cdot \xi  \label{LCspinxi}
\end{equation}%
where $T^{\left( \xi \right) }$ is the torsion tensor of $\varphi _{\xi }.$
Note that in \cite{AgricolaSpinors}, the torsion endomorphism is denoted by $%
S,$ and compared to our conventions $S=-T.$ The negative sign is due to a
different sign in the definition (\ref{phixispin}) of $\varphi _{\xi }$ in
terms of the spinor $\xi $. Equation (\ref{LCspinxi}) gives us an important
relationship between the spinorial covariant derivative and the octonion
covariant derivative.

\begin{theorem}
\label{thmLCspinOctCov}Let $\xi \in \Gamma \left( \mathcal{S}\right) $ be a
unit spinor on a $7$-manifold $M$ and let $\varphi _{\xi }$ be the $G_{2}$%
-structure defined by $\xi $ via (\ref{phixispin}). Then, for any $\eta \in
\Gamma \left( \mathcal{S}\right) $ 
\begin{equation}
j_{\xi }\left( \nabla _{X}^{S}\eta \right) =D_{X}^{\left( \xi \right)
}\left( j_{\xi }\left( \eta \right) \right)  \label{jxispinoctder}
\end{equation}%
where $D^{\left( \xi \right) }$ is the octonion covariant derivative (\ref%
{octocov}) with respect to the $G_{2}$-structure $\varphi _{\xi }.$
\end{theorem}

\begin{proof}
Using (\ref{LCspinxi}), we have 
\begin{eqnarray}
j_{\xi }\left( \nabla _{X}^{S}\xi \right) &=&-T_{X}^{\left( \xi \right) }
\label{jxiLCspin} \\
&=&D_{X}^{\left( \xi \right) }1=D_{X}^{\left( \xi \right) }j_{\xi }\left(
\xi \right)
\end{eqnarray}%
Then, for an arbitrary spinor $\eta =A\cdot \xi ,$ 
\begin{equation*}
\nabla _{X}^{S}\eta =\left( \nabla _{X}A\right) \cdot \xi +A\cdot \nabla
_{X}^{S}\xi
\end{equation*}%
and using the properties of $j_{\xi }$ (\ref{jxiprops}), we conclude 
\begin{eqnarray}
j_{\xi }\left( \nabla _{X}^{S}\eta \right) &=&\left( \nabla _{X}A\right)
\cdot j_{\xi }\left( \xi \right) +A\cdot j_{\xi }\left( \nabla _{X}^{S}\xi
\right)  \notag \\
&=&\nabla _{X}A-AT_{X}^{\left( \xi \right) }  \notag \\
&=&D_{X}^{\left( \xi \right) }A=D_{X}^{\left( \xi \right) }j_{\xi }\left(
\eta \right) .  \label{jxiLCspineta}
\end{eqnarray}
\end{proof}

\begin{remark}
Theorem \ref{thmLCspinOctCov} thus shows that with respect to the real
vector bundle isomorphism $S\cong \mathbb{O}M$ given by $j_{\xi },$ the spin
bundle connection $\nabla ^{S}$ is mapped precisely to the octonion
connection $D$. From the above properties it may seem that given a fixed
nowhere vanishing spinor, the octonion bundle is isomorphic to the spinor
bundle. The isomorphism however is only at the level of real vector bundles
and connections. However, similar to the relationship between the Clifford
algebra and the enveloping algebra of the octonions, the two structures are
not fully isomorphic, precisely due to the fact that the octonion bundle has
a nonassociative product which is not present in the spinor bundle.
Therefore, the octonion bundle retains all of the information from the
spinor bundle, but has some additional structure. This is also reflected in
the fact that there is no natural binary operation on spinors. When
transitioning to octonions, we are implicitly applying the \emph{triality }%
relationship between spinors and vectors to define the octonion
multiplication \cite{BaezOcto}.
\end{remark}

\begin{remark}
Theorem \ref{thmLCspinOctCov} also shows that our condition $DV=0$ (\ref{DV0}%
) for the existence of a torsion-free $G_{2}$-structure is equivalent to the
well-known condition $\nabla ^{S}\eta =0$ for some nowhere-vanishing spinor $%
\eta .$
\end{remark}

\section{Dirac operator}

\setcounter{equation}{0}\label{secDirac}We may define a distinguished $\func{%
Im}\mathbb{O}$-valued $1$-form $\delta \in \Omega ^{1}\left( \func{Im}%
\mathbb{O}M\right) $ such that for any vector $X$ on $M,$ $\delta \left(
X\right) \mathcal{\in }\Gamma \left( \func{Im}\mathbb{O}M\right) ,$ with
components given by 
\begin{equation}
\delta \left( X\right) =\left( 0,X\right) .  \label{deltadef}
\end{equation}%
Therefore in particular, $\delta $ is the isomorphism that takes vectors to
imaginary octonions. In components, its imaginary part is simply represented
by the Kronecker delta:%
\begin{equation}
\delta _{i}=\left( 0,\delta _{i}^{\ \alpha }\right) .  \label{deltadef2}
\end{equation}%
Note that left multiplication by $\delta $ gives a representation of the
Clifford algebra, so these are precise analogs of the gamma-matrices used in
physics. Below are some properties of $\delta $

\begin{lemma}
\label{lemDelProps} Suppose $\delta \in \Omega ^{1}\left( \func{Im}\mathbb{O}%
M\right) $ is defined by (\ref{deltadef2}) on a $7$-manifold $M$ with $G_{2}$%
-structure $\varphi $ and metric $g$. It then satisfies the following
properties, where octonion multiplication is with respect to $\varphi $

\begin{enumerate}
\item $\nabla \delta =0$

\item $\delta _{i}\delta _{j}=\left( -g_{ij},\varphi _{ij}^{\ \ \alpha
}\right) $

\item $\delta _{i}\left( \delta _{j}\delta _{k}\right) =\left( -\varphi
_{ijk},\psi _{\ ijk}^{\alpha }-\delta _{i}^{\alpha }g_{jk}+\delta _{\
j}^{\alpha }g_{ik}-\delta _{\ k}^{\alpha }g_{ij}\right) $

\item For any $A=\left( a_{0},\alpha \right) \in \Gamma \left( \mathbb{O}%
M\right) ,$ 
\begin{equation}
\delta _{i}A=\left( 
\begin{array}{c}
-\alpha _{i} \\ 
a_{0}\delta _{i}-\left( \alpha \lrcorner \varphi \right) _{i}%
\end{array}%
\right)
\end{equation}
\end{enumerate}

\begin{proof}
It is obvious that $\delta $ is parallel with respect to the Levi-Civita
connection. Consider now the octonion product $\delta _{i}\delta _{j}.$
Writing octonion real and imaginary parts in column notation for clarity, we
have 
\begin{equation*}
\delta _{i}\delta _{j}=\left( 
\begin{array}{c}
0 \\ 
\delta _{i}%
\end{array}%
\right) \left( 
\begin{array}{c}
0 \\ 
\delta _{j}%
\end{array}%
\right) =\left( 
\begin{array}{c}
-\left\langle \delta _{i},\delta _{j}\right\rangle _{\func{Im}\mathbb{O}} \\ 
\delta _{i}\times \delta _{j}%
\end{array}%
\right) =\left( 
\begin{array}{c}
-g_{ij} \\ 
\varphi _{\ ij}^{\alpha }%
\end{array}%
\right)
\end{equation*}%
and similarly, 
\begin{eqnarray*}
\delta _{i}\left( \delta _{j}\delta _{k}\right) &=&\left( 
\begin{array}{c}
0 \\ 
\delta _{i}%
\end{array}%
\right) \left( 
\begin{array}{c}
-g_{jk} \\ 
\varphi _{\ jk}^{\alpha }%
\end{array}%
\right) \\
&=&\left( 
\begin{array}{c}
-\varphi _{ijk} \\ 
-\delta _{i}^{\alpha }g_{jk}+\varphi _{\ \ i}^{\alpha \ \gamma }\varphi
_{\gamma jk}%
\end{array}%
\right) \\
&=&\left( 
\begin{array}{c}
-\varphi _{ijk} \\ 
\psi _{\ ijk}^{\alpha }-\delta _{i}^{\alpha }g_{jk}+\delta _{\ j}^{\alpha
}g_{ik}-\delta _{\ k}^{\alpha }g_{ij}%
\end{array}%
\right)
\end{eqnarray*}%
where we have used the contraction identity (\ref{phiphi1}) for $\varphi .$

Finally, consider $\delta _{i}A.$ We can write 
\begin{eqnarray*}
\delta _{i}A &=&a_{0}\delta _{i}+\delta _{i}\left( 0,\alpha \right) \\
&=&a_{0}\delta _{i}+\alpha ^{j}\delta _{i}\delta _{j} \\
&=&a_{0}\delta _{i}+\alpha ^{j}\left( 
\begin{array}{c}
-g_{ij} \\ 
\varphi _{\ ij}^{\alpha }%
\end{array}%
\right) \\
&=&\left( 
\begin{array}{c}
-\alpha _{i} \\ 
a_{0}\delta _{i}-\left( \alpha \lrcorner \varphi \right) _{i}%
\end{array}%
\right)
\end{eqnarray*}
\end{proof}
\end{lemma}

We can now define the octonion Dirac operator $\NEG{D}$ using $\delta $ and
the octonion covariant derivative $D$ (\ref{octocov}). Let $A\in \Gamma
\left( \mathbb{O}M\right) ,$ then define $\NEG{D}A$ as 
\begin{equation}
\NEG{D}A=\delta \overset{\circ }{\lrcorner }\left( DA\right)
\label{defDiracOp}
\end{equation}%
where $\overset{\circ }{\lrcorner }$ is a combination of contraction and
octonion multiplication. In coordinates, (\ref{defDiracOp}) is given by%
\begin{equation}
\NEG{D}A=\delta ^{i}\circ \left( D_{i}A\right) .
\end{equation}%
This operator is precisely what we obtain by applying the map $j_{\xi }$ to
the standard Dirac operator on the spinor bundle. We can use this definition
to work some properties of the operator. First of all, let us prove that $%
\NEG{D}$ is covariant under a change of the reference $G_{2}$-structure:

\begin{proposition}
\label{propDiracV}Suppose $\left( \varphi ,g\right) $ is a $G_{2}$-structure
on a $7$-manifold $M$, with torsion $T$ and corresponding octonion covariant
derivative $D$ and Dirac operator $\NEG{D}.$ Suppose $V$ is a unit octonion
section, and $\tilde{\varphi}=\sigma _{V}\left( \varphi \right) $ is the
corresponding $G_{2}$-structure, that has torsion $\tilde{T}$, given by (\ref%
{TorsV3}), an octonion covariant derivative $\tilde{D}$ and Dirac operator $%
\widetilde{\NEG{D}}.$ Then, for any octonion section $A$, we have 
\begin{equation}
\widetilde{\NEG{D}}\left( AV^{-1}\right) =\left( \NEG{D}A\right) V^{-1}
\end{equation}
\end{proposition}

\begin{proof}
This follows from the covariant nature of $D.$ Suppose $\tilde{\circ}$
denotes octonion product with respect to $\tilde{\varphi}.$ By definition of 
$\widetilde{\NEG{D}}$, we then have 
\begin{equation*}
\widetilde{\NEG{D}}\left( AV^{-1}\right) =\delta \overset{\tilde{\circ}}{%
\lrcorner }\left( \tilde{D}\left( AV^{-1}\right) \right)
\end{equation*}%
Using Proposition \ref{propDV} and (\ref{OctoVAB}), we rewrite this as 
\begin{eqnarray*}
\widetilde{\NEG{D}}\left( AV^{-1}\right) &=&\delta \overset{\tilde{\circ}}{%
\lrcorner }\left( \left( DA\right) V^{-1}\right) \\
&=&\delta \overset{\circ }{\lrcorner }\left( \left( DA\right) V^{-1}\right) +%
\left[ \delta ,\left( DA\right) V^{-1},V\right] V^{-1} \\
&=&\left( \delta \overset{\circ }{\lrcorner }\left( DA\right) \right) V^{-1}+%
\left[ \delta ,DA,V^{-1}\right] +\left[ \delta ,DA,V\right] \\
&=&\left( DA\right) V^{-1}
\end{eqnarray*}%
where we have also used Lemma \ref{lemAssocIds}.
\end{proof}

\begin{theorem}
\label{thmDirac}Suppose $V$ is a unit octonion section, and suppose $\tilde{%
\varphi}=\sigma _{V}\left( \varphi \right) $ has torsion tensor $\tilde{T}$.
Then, 
\begin{equation}
\NEG{D}V=\left( 
\begin{array}{c}
7\tilde{\tau}_{1} \\ 
-6\tilde{\tau}_{7}%
\end{array}%
\right) V  \label{DirV}
\end{equation}%
where $\tilde{\tau}_{1}=\frac{1}{7}\func{Tr}\tilde{T}$ and $\tilde{\tau}_{7}=%
\frac{1}{6}\tilde{T}\lrcorner \tilde{\varphi}$ are the $1$-dimensional and $%
7 $-dimensional components of $\tilde{T},$ respectively. In particular, 
\begin{equation}
\NEG{D}1=\left( 
\begin{array}{c}
7\tau _{1} \\ 
-6\tau _{7}%
\end{array}%
\right)  \label{Dir1}
\end{equation}%
where $\tau _{1}$ and $\tau _{7}$ are the corresponding components of $T$ -
the torsion tensor of the $G_{2}$-structure $\varphi .$
\end{theorem}

\begin{proof}
Let us first verify (\ref{Dir1}). Indeed, since from (\ref{DX1}), $D1=-T$ we
have 
\begin{eqnarray*}
\NEG{D}1 &=&\delta ^{i}D_{i}1=-\delta ^{i}T_{i} \\
&=&-T_{i}^{\ j}\left( \delta ^{i}\delta _{j}\right) \\
&=&-T_{i}^{\ j}\left( 
\begin{array}{c}
-g_{\ j}^{i} \\ 
\varphi _{\ \ j}^{\alpha i}%
\end{array}%
\right)
\end{eqnarray*}%
where have also used Lemma \ref{lemDelProps}. Thus, 
\begin{equation}
\NEG{D}1=\left( 
\begin{array}{c}
\func{Tr}T \\ 
-T\lrcorner \varphi%
\end{array}%
\right)  \label{Dir1T}
\end{equation}%
and (\ref{Dir1}) follows.

Now let $\widetilde{\NEG{D}}$ be the Dirac operator with respect to the $%
G_{2}$-structure $\tilde{\varphi}=\sigma _{V}\left( \varphi \right) .$ To
get (\ref{DirV}), we note that 
\begin{equation*}
\widetilde{\NEG{D}}1=\left( 
\begin{array}{c}
7\tilde{\tau}_{1} \\ 
-6\tilde{\tau}_{7}%
\end{array}%
\right) ,
\end{equation*}%
however from Proposition \ref{propDiracV}, 
\begin{equation*}
\widetilde{\NEG{D}}1=\widetilde{\NEG{D}}\left( VV^{-1}\right) =\left( \NEG%
{D}V\right) V^{-1}
\end{equation*}%
Hence, 
\begin{equation*}
\left( \NEG{D}V\right) =\left( \widetilde{\NEG{D}}1\right) V
\end{equation*}%
and thus we get (\ref{DirV}).
\end{proof}

Equivalently we can translate the results of Theorem \ref{thmDirac} into the
language of spinors.

\begin{corollary}
\label{corrDirSpin}Let $\xi $ be a unit spinor and suppose $\tau _{1}$ and $%
\tau _{7}$ are the $1$- and $7$-dimensional components of the torsion of the 
$G_{2}$-structure $\varphi _{\xi }.$ Then, $\xi $ satisfies 
\begin{equation}
\slashed{\nabla}\xi =7\tau _{1}\xi -6\tau _{7}\cdot \xi
\end{equation}%
where $\slashed{\nabla}$ is the Dirac operator on the spinor bundle.
\end{corollary}

\begin{corollary}
\label{Corr1427}Let $\xi $ be a unit spinor. Then the corresponding $G_{2}$%
-structure $\varphi _{\xi }$ has only $14$- and $27$-dimensional torsion
components if and only if $\slashed{\nabla}\xi =0,$ that is, $\xi $ is a
harmonic spinor.
\end{corollary}

The result in Corollary \ref{Corr1427} has also been proved recently using a
different method in \cite{AgricolaSpinors}.

The octonionic Dirac operator also satisfies the Lichnerowicz-Weitzenb\"{o}%
ck formula. Of course given the spinorial Lichnerowicz-Weitzenb\"{o}ck
formula, we immediately obtain the octonionic analog using the map $j_{\xi
}, $ however we can also prove it using octonionic techniques.

\begin{theorem}[Lichnerowicz-Weitzenb\"{o}ck]
For any smooth octonion section $V,$ we have 
\begin{equation}
\NEG{D}^{2}V=d_{D}^{\ast }d_{D}V+\frac{1}{4}RV  \label{BWformula}
\end{equation}%
where $R$ is the scalar curvature.
\end{theorem}

\begin{proof}
If $V$ is identically zero at a point $p$ and a neighborhood around it, then
(\ref{BWformula}) is trivially true, since all the operators are local.
Suppose $V\neq 0$ at $p$, then at least locally we can change the reference $%
G_{2}$-structure to $\sigma _{V}\left( \varphi \right) .$ Then, if $%
\widetilde{\NEG{D}}$ is the Dirac operator corresponding to $\sigma
_{V}\left( \varphi \right) ,$ Proposition \ref{propDiracV} shows that 
\begin{eqnarray}
\left( \NEG{D}V\right) &=&\left( \widetilde{\NEG{D}}1\right) V  \notag \\
\NEG{D}^{2}V &=&\NEG{D}\left( \left( \widetilde{\NEG{D}}1\right) V\right) 
\notag \\
&=&\left( \widetilde{\NEG{D}}^{2}1\right) V  \label{BWproof0}
\end{eqnarray}%
So in this case, it is enough to verify (\ref{BWformula}) for $V=1.$ We then
have%
\begin{eqnarray*}
\NEG{D}^{2}1 &=&\delta ^{i}D_{i}\left( \delta ^{j}D_{j}1\right) \\
&=&\delta ^{i}\left( \delta ^{j}\left( D_{i}D_{j}1\right) \right)
\end{eqnarray*}%
since $\delta $ is parallel with respect to $D.$ Applying the associator, we
get 
\begin{equation}
\NEG{D}^{2}1=\left( \delta ^{i}\delta ^{j}\right) \left( D_{i}D_{j}1\right) +%
\left[ \delta ^{i},\delta ^{j},D_{i}D_{j}1\right]  \label{BWproof1}
\end{equation}%
From Lemma \ref{lemDelProps}, $\delta ^{i}\delta ^{j}=\left( 
\begin{array}{c}
-g^{ij} \\ 
\varphi ^{\alpha ij}%
\end{array}%
\right) ,$ hence the first term in (\ref{BWproof1}) becomes 
\begin{eqnarray}
\left( \delta ^{i}\delta ^{j}\right) \left( D_{i}D_{j}1\right) &=&\left( 
\begin{array}{c}
-g^{ij} \\ 
\varphi ^{\alpha ij}%
\end{array}%
\right) \left( 
\begin{array}{c}
\func{Re}D_{i}D_{j}1 \\ 
\func{Im}D_{i}D_{j}1%
\end{array}%
\right)  \notag \\
&=&d_{D}^{\ast }d_{D}1+\left( 
\begin{array}{c}
0 \\ 
\varphi ^{\alpha ij}%
\end{array}%
\right) \left( 
\begin{array}{c}
\func{Re}D_{i}D_{j}1 \\ 
\func{Im}D_{i}D_{j}1%
\end{array}%
\right)  \label{BWproof2}
\end{eqnarray}%
where we have used the formula (\ref{ddsformula}) for $d_{D}^{\ast }$ in
coordinates. Note that now in the second term of (\ref{BWproof1}) as well as
(\ref{BWproof2}), the indices $i$ and $j$ are skew-symmetrized, but from
Proposition \ref{propdD2},%
\begin{equation}
D_{[i}D_{j]}1=\frac{1}{2}d_{D}^{2}1=-\frac{1}{8}\pi _{7}\func{Riem}
\label{BWproof4}
\end{equation}%
Hence, $\func{Re}D_{[i}D_{j]}1=0.$ Using (\ref{BWproof2}) and (\ref%
{octoassoc}), equation (\ref{BWproof1}) becomes%
\begin{equation*}
\NEG{D}^{2}1=d_{D}^{\ast }d_{D}1-\frac{1}{8}\varphi ^{kij}\left( \pi _{7}%
\func{Riem}\right) _{ij}^{\ \ l}\left( \delta _{k}\delta _{l}\right) -\frac{1%
}{4}\psi _{\ \ \ }^{kijl}\left( \pi _{7}\func{Riem}\right) _{ijl}\delta _{k}
\end{equation*}%
Using Lemma \ref{lemDelProps}, $\delta _{k}\delta _{l}=-g_{kl}+\varphi
_{kl}^{\ \ m}\delta _{m},$ thus 
\begin{equation}
\NEG{D}^{2}1=d_{D}^{\ast }d_{D}1+\frac{1}{8}\varphi ^{kij}\left( \pi _{7}%
\func{Riem}\right) _{ijk}^{\ \ }-\frac{1}{8}\varphi ^{kij}\left( \pi _{7}%
\func{Riem}\right) _{ij}^{\ \ l}\varphi _{kl}^{\ \ m}\delta _{m}-\frac{1}{4}%
\psi _{\ \ \ }^{kijl}\left( \pi _{7}\func{Riem}\right) _{ijl}\delta _{k}
\label{BWproof5}
\end{equation}%
However, using the identity (\ref{phiphi1}), as well as the Riemannian
Bianchi identity, 
\begin{eqnarray*}
\varphi ^{kij}\left( \pi _{7}\func{Riem}\right) _{ijk}^{\ \ } &=&\varphi
^{kij}\func{Riem}_{ijmn}\varphi _{\ \ \ \ k}^{mn} \\
&=&\func{Riem}_{ijmn}\left( \psi ^{ijmn}+g^{im}g^{jn}-g^{in}g^{jm}\right) \\
&=&2R
\end{eqnarray*}%
where we have used (\ref{riemconvention}). Similarly, using the identity (%
\ref{phiphi1}) and the Bianchi identity, 
\begin{eqnarray*}
\varphi ^{kij}\left( \pi _{7}\func{Riem}\right) _{ij}^{\ \ l}\varphi
_{kl}^{\ \ m} &=&\left( \pi _{7}\func{Riem}\right) _{ijl}^{\ \ }\left( \psi
_{\ }^{ijlm}+g^{il}g^{jm}-g^{im}g^{jl}\right) \\
&=&\left( \pi _{7}\func{Riem}\right) _{ijl}^{\ \ }\psi _{\ }^{ijlm}
\end{eqnarray*}%
Now we are left with 
\begin{equation}
\NEG{D}^{2}1=d_{D}^{\ast }d_{D}1+\frac{1}{4}R+\frac{1}{8}\left( \pi _{7}%
\func{Riem}\right) _{ijl}^{\ \ }\psi _{\ }^{ijlm}\delta _{m}.
\label{BWproof6}
\end{equation}%
Consider 
\begin{equation*}
\left( \pi _{7}\func{Riem}\right) _{ijl}^{\ \ }\psi _{\ }^{ijlm}=-\func{Riem}%
_{ijpq}\varphi _{\ \ l}^{pq}\psi _{\ }^{ijml}
\end{equation*}%
Using the identity (\ref{phipsiid}),%
\begin{eqnarray*}
\func{Riem}_{ijpq}\varphi _{\ \ l}^{pq}\psi _{\ }^{ijml} &=&-3\func{Riem}%
_{ijpq}\left( g^{p[i}\varphi ^{jm]q}-g^{q[i}\varphi ^{jm]p}\right) \\
&=&-6\func{Riem}_{ijpq}g^{p[i}\varphi ^{jm]q} \\
&=&6\func{Riem}_{ijq}^{\ \ [i}\varphi ^{jm]q} \\
&=&-2\left( \func{Ric}\right) _{jq}\varphi ^{jmq}+2\left( \func{Ric}\right)
_{iq}\varphi ^{miq}+2\func{Riem}_{ijq}^{\ \ \ m}\varphi ^{ijq} \\
&=&0
\end{eqnarray*}%
Thus, we are left with 
\begin{equation*}
\NEG{D}^{2}1=d_{D}^{\ast }d_{D}1+\frac{1}{4}R.
\end{equation*}%
For an arbitrary octonion section $V$ that is nonzero at a point $p,$ using (%
\ref{BWproof0}) we obtain (\ref{BWformula}).

Now suppose $V\left( p\right) =0$ and is not identically zero in any
neighborhood of $p$. Let $L=$ $\NEG{D}^{2}V-d_{D}^{\ast }d_{D}V-\frac{1}{4}%
RV.$ Since $V$ is fixed, this is in particular a continuous map from a
neighborhood of $p$ to the octonions. We can then find a sequence of points $%
\left\{ p_{n}\right\} $ in a neighborhood of $p$ such that $%
p_{n}\longrightarrow p$ and such that for each $n$, either $V\left(
p_{n}\right) \neq 0$ or $V\left( p_{n}\right) =0$ and $V\equiv 0$ in a
neighborhood of $p_{n}.$ Using the previous cases in the proof, we find that 
$L\left( p_{n}\right) =0$ for all $n$. Therefore, by continuity, $L\left(
p\right) =0,$ and thus the identity is satisfied. We thus conclude that (\ref%
{BWformula}) is true for any (smooth) octonion section $V$.
\end{proof}

\begin{remark}
On a compact manifold $M$, for a unit octonion $V,$ we then find that 
\begin{equation}
\int_{M}\left\vert \NEG{D}V\right\vert ^{2}\func{vol}=\int_{M}\left\vert
DV\right\vert ^{2}\func{vol}+\frac{1}{4}\int_{M}R\func{vol}
\label{dvintegral}
\end{equation}%
This shows that if $\int_{M}R\func{vol}=0$, i.e. the total scalar curvature
is zero, then $\NEG{D}V=0$ if and only if $DV=0.$ This then implies that on
a compact manifold, a $G_{2}$-structure that is compatible with a metric
that has vanishing total scalar curvature, is torsion-free if and only if
its $\tau _{1}$ and $\tau _{7}$ torsion components both vanish. This can
also be obtained directly by integrating the expression (\ref{torsscalcurv})
for the scalar curvature in terms of torsion components. Note that if the
scalar curvature vanishes \emph{pointwise}, then also from (\ref%
{torsscalcurv}), we see that even without requiring compactness we find that
a compatible $G_{2}$-structure is torsion-free if and only if $\tau _{1}$
and $\tau _{7}$ both vanish.
\end{remark}

We can also write out the octonionic Dirac operator explicitly in terms of
the $G_{2}$-structure torsion.

\begin{theorem}
\label{thmDiracexp}Suppose $A=\left( a_{0},\alpha \right) \in \Gamma \left( 
\mathbb{O}M\right) .$ Then, 
\begin{equation}
\NEG{D}A=\left( 
\begin{array}{c}
-\func{div}_{T}A \\ 
\func{grad}_{T}A+\func{curl}_{T}A%
\end{array}%
\right)  \label{octodiracexp}
\end{equation}%
where $\func{div}_{T},$ $\func{grad}_{T}$ and $\func{curl}_{T}$ are given by 
\begin{subequations}%
\label{Tdivgradcurl} 
\begin{eqnarray}
\func{div}_{T}A &=&\func{div}\alpha -a_{0}\func{Tr}T+\left\langle \alpha
,T\lrcorner \varphi \right\rangle =\func{div}\alpha -7a_{0}\tau
_{1}-6\left\langle \alpha ,\tau _{7}\right\rangle \\
\func{grad}_{T}A &=&\func{grad}a_{0}+T\left( v\right) \\
\func{curl}_{T}A &=&\func{curl}\alpha +a_{0}\left( T\lrcorner \varphi
\right) -\alpha \func{Tr}T+T_{\alpha }-\alpha \lrcorner \left( T\lrcorner
\psi \right) \\
&=&\func{curl}\alpha +6a_{0}\tau _{7}-7\alpha \tau _{1}+T_{\alpha }+2\alpha
\lrcorner \tau _{14}+4\alpha \times \tau _{7}  \notag
\end{eqnarray}%
\end{subequations}
where $\left( T_{\alpha }\right) ^{k}=\alpha ^{i}T_{i}^{\ k}$ and 
\begin{equation}
\left( \func{curl}\alpha \right) ^{i}=\varphi ^{iab}\nabla _{a}\alpha _{b}.
\label{curl1}
\end{equation}
\end{theorem}

\begin{remark}
In \cite{karigiannis-2006notes}, Karigiannis gave an expression for an
octonionic Dirac operator in the torsion-free case. In the case $T=0$, (\ref%
{octodiracexp}) reduces to%
\begin{equation}
\NEG{D}A=\left( -\func{div}\alpha ,\func{grad}a_{0}+\func{curl}\alpha \right)
\label{diracT0}
\end{equation}%
which is precisely the expression given by Karigiannis.
\end{remark}

\begin{proof}[Proof of Theorem \protect\ref{thmDiracexp}]
By definition (\ref{defDiracOp}), we have 
\begin{equation}
\NEG{D}A=\delta ^{i}\left( \nabla _{i}A-AT_{i}\right)  \label{dirAexp1}
\end{equation}%
The first term is then 
\begin{eqnarray*}
\delta ^{i}\nabla _{i}A &=&\left( 
\begin{array}{c}
0 \\ 
g^{ij}%
\end{array}%
\right) \left( 
\begin{array}{c}
\nabla _{i}a_{0} \\ 
\nabla _{i}\alpha ^{k}%
\end{array}%
\right) \\
&=&\left( 
\begin{array}{c}
-g^{ij}\nabla _{i}\alpha _{j} \\ 
\nabla ^{l}a_{0}+\varphi _{\ jk}^{l}g^{ij}\nabla _{i}\alpha ^{k}%
\end{array}%
\right) \\
&=&\left( -\func{div}\alpha ,\func{grad}a_{0}+\func{curl}\alpha \right)
\end{eqnarray*}%
which is precisely (\ref{diracT0}) - the term which is independent of the
torsion. Now let us look at the second term of (\ref{dirAexp1}). We have%
\begin{equation}
\delta ^{i}\left( AT_{i}\right) =a_{0}\delta ^{i}T_{i}+\alpha ^{j}\delta
^{i}\left( \delta _{j}T_{i}\right) .  \label{dirAexp2}
\end{equation}%
Now again look at each term separately. The first term in (\ref{dirAexp2})
is then just 
\begin{eqnarray*}
a_{0}\delta ^{i}T_{i} &=&-a_{0}\NEG{D}1 \\
&=&-a_{0}\left( 
\begin{array}{c}
\func{Tr}T \\ 
-T\lrcorner \varphi%
\end{array}%
\right)
\end{eqnarray*}%
where we have used (\ref{Dir1T}) and the computation before (\ref{Dir1T}).
Since the components $T_{i}^{\ k}$ are real, the second term in (\ref%
{dirAexp2}) can be re-written in the following way%
\begin{eqnarray*}
\alpha ^{j}\delta ^{i}\left( \delta _{j}T_{i}\right) &=&\alpha ^{j}T_{i}^{\
k}\delta ^{i}\left( \delta _{j}\delta _{k}\right) \\
&=&\alpha ^{j}T_{i}^{\ k}\left( 
\begin{array}{c}
-\varphi _{\ jk}^{i} \\ 
\psi _{\ \ jk}^{mi}-g^{mi}g_{jk}+\delta _{\ j}^{m}\delta _{k}^{i}-\delta _{\
k}^{m}\delta _{\ j}^{i}%
\end{array}%
\right) \\
&=&\left( 
\begin{array}{c}
\,\left\langle \alpha ,T\lrcorner \varphi \right\rangle \\ 
\alpha \lrcorner \left( T\lrcorner \psi \right) -T\left( \alpha \right)
+\alpha \func{Tr}T-T_{\alpha }%
\end{array}%
\right)
\end{eqnarray*}%
where we have also used Lemma \ref{lemDelProps}. Overall, 
\begin{equation*}
\NEG{D}A=\left( 
\begin{array}{c}
-\func{div}\alpha +a_{0}\func{Tr}T-\left\langle \alpha ,T\lrcorner \varphi
\right\rangle \\ 
\func{grad}a_{0}+a_{0}T\lrcorner \varphi +\func{curl}\alpha -\alpha
\lrcorner \left( T\lrcorner \psi \right) +T\left( \alpha \right) -\alpha 
\func{Tr}T+T_{\alpha }%
\end{array}%
\right)
\end{equation*}%
Applying the definitions (\ref{Tdivgradcurl}) of $\func{div}_{T},$ $\func{%
grad}_{T}$ and $\func{curl}_{T},$ we obtain (\ref{octodiracexp}). The
alternative expressions using components of $T$ are then derived using the
identities%
\begin{eqnarray*}
\func{Tr}T &=&7\tau _{1} \\
T\lrcorner \varphi &=&6\tau _{7} \\
T\lrcorner \psi &=&4\tau _{7}\lrcorner \varphi -2\tau _{14}
\end{eqnarray*}
\end{proof}

\begin{remark}
The motivation for the definitions (\ref{Tdivgradcurl}) is the following.
The standard $\func{div},$ $\func{grad}$ and $\func{curl}$ are obtained by
computing $\delta ^{i}\nabla _{i}A.$ In particular, 
\begin{subequations}%
\label{divgradcurl} 
\begin{eqnarray}
\func{div}\alpha &=&\left\langle \delta ^{i},\func{Im}\nabla
_{i}A\right\rangle \\
\func{grad}a_{0} &=&\delta ^{i}\left( \func{Re}\nabla _{i}A\right) \\
\func{curl}\alpha &=&\delta ^{i}\times \left( \func{Im}\nabla _{i}A\right) .
\end{eqnarray}%
\end{subequations}%
The expressions (\ref{Tdivgradcurl}) are then similarly obtained by
replacing $\nabla A$ in (\ref{divgradcurl}) by $DA:$ 
\begin{subequations}%
\label{divTgradTcurlT}%
\begin{eqnarray}
\func{div}_{T}A &=&\left\langle \delta ^{i},\func{Im}D_{i}A\right\rangle \\
\func{grad}_{T}A &=&\delta ^{i}\left( \func{Re}D_{i}A\right) \\
\func{curl}_{T}A &=&\delta ^{i}\times \left( \func{Im}D_{i}A\right) .
\end{eqnarray}%
\end{subequations}%
Note that since now there is some mixing of real and imaginary parts of $A,$
it only makes sense to apply these operators to $A$ as a whole. It is then
clear that 
\begin{equation*}
\NEG{D}A=\delta ^{i}D_{i}A=\left( -\func{div}_{T}A,\func{grad}_{T}A+\func{%
curl}_{T}A\right)
\end{equation*}%
It then takes a routine calculation to actually obtain the expressions (\ref%
{Tdivgradcurl}) from (\ref{divTgradTcurlT}).
\end{remark}

\section{Energy functional}

\setcounter{equation}{0}\label{secEnergy}Given a $7$-dimensional Riemannian
manifold that admit $G_{2}$-structures, we have a choice of $G_{2}$%
-structures that correspond to the given Riemannian metric $g$. After fixing
an arbitrary $G_{2}$-structure $\varphi $ in this metric class, all the
other $G_{2}$-structures that are compatible with $g$ are parametrized by
unit octonion sections, up to a sign. Given a unit octonion section $V,$ the
corresponding $G_{2}$-structure $\sigma _{V}\left( \varphi \right) $ will
have torsion $T^{\left( V\right) }$ given by $T^{\left( V\right) }=-\left(
DV\right) V^{-1},$ where $D$ is the octonion covariant derivative with
respect to $\varphi $. The question is how to pick the \textquotedblleft
best\textquotedblright\ representative of this metric class. The choice of a
particular $G_{2}$-structure in a fixed metric class is akin to choosing a
gauge in gauge theory. Obviously, if the metric has holonomy contained in $%
G_{2},$ then the \textquotedblleft best\textquotedblright\ representative
should be a torsion-free $G_{2}$-structure that corresponds to that metric.
In general however, one approach, at least on compact manifolds, would be to
pick a gauge that minimizes some functional. The obvious choice is the $%
L_{2} $-norm of the torsion. Suppose $M$ is now compact, and define a
functional $\mathcal{E}:\Gamma \left( S\mathbb{O}M\right) \longrightarrow 
\mathbb{R}$ , where $S\mathbb{O}M$ is the unit sphere subbundle, by 
\begin{eqnarray}
\mathcal{E}\left( V\right) &=&\frac{1}{2}\int_{M}\left\vert T^{\left(
V\right) }\right\vert ^{2}\func{vol} \\
&=&\frac{1}{2}\int_{M}\left\vert \left( DV\right) V^{-1}\right\vert ^{2}%
\func{vol} \\
&=&\frac{1}{2}\int_{M}\left\vert DV\right\vert ^{2}\func{vol}  \label{eV3}
\end{eqnarray}%
Thus, this is simply the energy functional for unit octonion sections. Note
that a similar energy functional for spinors has been recently studied by
Ammann, Weiss and Witt \cite{AmmannWeissWitt1}, however in their case, the
metric was unconstrained, and so the functional was both on spinors and
metrics. Using the properties of $D,$ we obtain the critical points using
standard methods:

\begin{proposition}
The critical points of $\mathcal{E}$ satisfy 
\begin{equation}
D^{\ast }DV-\left\vert DV\right\vert ^{2}V=0.  \label{unitvecteq}
\end{equation}
\end{proposition}

\begin{proof}
To make the restriction to unit octonions explicit, let us introduce a
Lagrange multiplier function $\lambda ,$ so that now the functional is given
by 
\begin{equation*}
\mathcal{E}\left( V\right) =\frac{1}{2}\int_{M}\left( \left\vert
DV\right\vert ^{2}-\lambda \left( \left\vert V\right\vert ^{2}-1\right)
\right) \func{vol}
\end{equation*}%
Variations of $\lambda $ give the pointwise constraint 
\begin{equation}
\left\vert V\right\vert ^{2}=1  \label{critpt1}
\end{equation}
Now, suppose $V\left( t\right) $ is a $1$-parameter family of unit octonion
sections, then 
\begin{eqnarray*}
\frac{d}{dt}\mathcal{E}\left( V\left( t\right) \right) &=&\frac{1}{2}%
\int_{M}\left( \frac{d}{dt}\left\vert DV\left( t\right) \right\vert
^{2}-\lambda \frac{d}{dt}\left\vert V\right\vert ^{2}\right) \func{vol} \\
&=&\int \left\langle D\frac{d}{dt}V\left( t\right) ,DV\left( t\right)
\right\rangle _{\mathbb{O}}-\lambda \left\langle V\left( t\right) ,\frac{d}{%
dt}V\left( t\right) \right\rangle _{\mathbb{O}}\func{vol} \\
&=&\int \left\langle \frac{d}{dt}V\left( t\right) ,D^{\ast }DV\left(
t\right) -\lambda V\left( t\right) \right\rangle _{\mathbb{O}}\func{vol}
\end{eqnarray*}%
Therefore, at a critical point, we must also have 
\begin{equation}
D^{\ast }DV-\lambda V=0  \label{critpoint2}
\end{equation}%
From (\ref{critpt1}), using the metric-compatible property of $D$ (\ref%
{DXmet}), we find that 
\begin{eqnarray}
\left\langle DV,V\right\rangle &=&0  \label{dvv1} \\
\left\langle D^{\ast }DV,V\right\rangle -\left\vert DV\right\vert ^{2} &=&0
\label{d2vv1}
\end{eqnarray}%
Therefore, by taking the inner product of (\ref{critpoint2}) with $V,$ we
conclude that $\lambda =\left\vert DV\right\vert ^{2}$, and hence obtain (%
\ref{unitvecteq}).
\end{proof}

\begin{remark}
In the case when the torsion vanishes, and when restricted to imaginary
octonions, since $D=\nabla $, the functional (\ref{eV3}) reduces to the
well-known energy functional of a unit vector field. In that case, the
critical points of this energy functional are known as \emph{harmonic unit
vector fields}. The energy functional of a unit vector fields (sometimes
also known as the \emph{total bending}) have been studied independently by
Gil-Medrano, Wiegmink, and Wood, among others (\cite%
{GilMedranoVF1,WiegminkBending,WoodUnitVF}). In particular, the equation
satisfied by a harmonic unit vector field $v$ is very similar to the
equation (\ref{unitvecteq}):%
\begin{equation}
\nabla ^{\ast }\nabla v-\left\vert \nabla v\right\vert ^{2}v=0
\label{harmunitvecteq}
\end{equation}%
Such unit vectors are called harmonic, because the equation (\ref%
{harmunitvecteq}) is one of the equations obtained when considering \emph{%
harmonic maps} from a manifold into the unit tangent bundle. A unit vector
field can be regarded as a harmonic map if and only if it satisfies (\ref%
{harmunitvecteq}) as well as an equation that depends on the curvature. More
details on harmonic unit vector fields can be in particular be found in \cite%
{BoeckxVanhecke1,HanYim,Ishihara1}.
\end{remark}

Given a unit octonion that is a critical point of the functional $\mathcal{E}%
,$ we can interpret the equation (\ref{unitvecteq}) in terms of the torsion
of the corresponding $G_{2}$-structure $\sigma _{V}\left( \varphi \right) .$

\begin{corollary}
\label{corrDivT}A unit octonion $V$ is a critical point of the functional $%
\mathcal{E}$ if and only if the torsion $T^{\left( V\right) }$ of the $G_{2}$%
-structure $\sigma _{V}\left( \varphi \right) $ satisfies%
\begin{equation}
\func{div}T^{\left( V\right) }=0.  \label{divT0}
\end{equation}
\end{corollary}

\begin{proof}
For convenience, let $\tilde{T}=T^{\left( V\right) }.$ Note that from (\ref%
{TorsV3}), 
\begin{equation*}
\tilde{T}=-\left( DV\right) V^{-1}
\end{equation*}%
Hence, the equation (\ref{unitvecteq}) can be rewritten as 
\begin{equation*}
-D^{\ast }\left( \tilde{T}V\right) -\left\vert \tilde{T}\right\vert ^{2}V=0
\end{equation*}%
Using the expression (\ref{ddsformula}) for $D^{\ast },$ and the property of 
$D$ (\ref{octocovprod}) we can write%
\begin{equation*}
D^{a}\left( \tilde{T}_{a}V\right) -\left\vert \tilde{T}\right\vert
^{2}V=\left( \nabla ^{a}\tilde{T}_{a}\right) V+\tilde{T}_{a}D^{a}V-\left%
\vert \tilde{T}\right\vert ^{2}V
\end{equation*}%
However, $DV=-\tilde{T}V,$ so 
\begin{eqnarray*}
\tilde{T}_{a}D^{a}V &=&-\tilde{T}_{a}\left( \tilde{T}^{a}V\right) \\
&=&-\left( \tilde{T}_{a}\tilde{T}^{a}\right) V-\left[ \tilde{T}_{a},\tilde{T}%
^{a},V\right]
\end{eqnarray*}%
Now, 
\begin{equation*}
\tilde{T}_{a}\tilde{T}^{a}=\left( 
\begin{array}{c}
0 \\ 
\tilde{T}_{a}^{\ m}%
\end{array}%
\right) \left( 
\begin{array}{c}
0 \\ 
\tilde{T}_{{}}^{\ an}%
\end{array}%
\right) =\left( 
\begin{array}{c}
-\left\vert \tilde{T}\right\vert ^{2} \\ 
\tilde{T}_{a}^{\ m}\tilde{T}_{{}}^{\ an}\varphi _{mn}^{\ \ \ \ p}%
\end{array}%
\right)
\end{equation*}%
however, $\tilde{T}_{a}^{\ m}\tilde{T}_{{}}^{\ an}$ is symmetric in $m$ and $%
n$, so $\tilde{T}_{a}\tilde{T}^{a}=-\left\vert \tilde{T}\right\vert ^{2}.$
For the same reason, the associator $\left[ \tilde{T}_{a},\tilde{T}^{a},V%
\right] =0.$ Hence, 
\begin{equation*}
\tilde{T}_{a}D^{a}V=\left\vert \tilde{T}\right\vert ^{2}V
\end{equation*}%
Therefore, 
\begin{equation*}
D^{\ast }DV-\left\vert DV\right\vert ^{2}V=\left( \func{div}\tilde{T}\right)
V
\end{equation*}%
Since $V$ is nowhere-vanishing, $D^{\ast }DV-\left\vert DV\right\vert
^{2}V=0 $ if and only if $\func{div}\tilde{T}=0.$
\end{proof}

\begin{remark}
From Corollary \ref{corrDivT} we obtain a different interpretation of the
critical points of $\mathcal{E}$ - the critical points correspond to $G_{2}$%
-structures that have divergence-free torsion. This description fits very
well with the interpretation of the $G_{2}$-structure torsion as a
connection for a non-associative gauge theory. The condition $\func{div}T=0$
is then simply the analog of the Coulomb gauge. It is well-known (e.g. \cite%
{DonaldsonGauge,TaoBook}) that in gauge theory, the Coulomb gauge $d^{\ast
}A=0,$ for the gauge connection $A,$ corresponds to critical points of the $%
L_{2}$-norm of $A$. In our situation, we have an exactly similar thing
happening. This gives an interesting link between the harmonic map point of
view and the Coulomb gauge point of view. Given Uhlenbeck's existence result
for the Coulomb gauge \cite{UhlenbeckConnection}, there may be a possibility
of an existence result for divergence-free torsion.
\end{remark}

The characterization of divergence-free torsion as corresponding to critical
points of the functional $\mathcal{E}$ shows that $G_{2}$-structures with
such torsion (whenever they exist) are in some sense special. However the
significance of divergence-free torsion still needs to be investigated. In
particular, we can see that unit norm eigensections of $\NEG{D}$ are in fact
critical points of $\mathcal{E}.$

\begin{proposition}
If $V$ is a unit eigensection of the Dirac operator $\NEG{D},$ then $V$ is a
critical point of the functional $\mathcal{E}.$
\end{proposition}

\begin{proof}
Suppose $V$ is a unit eigensection of $\NEG{D}$ with eigenvalue $\lambda ,$
then%
\begin{eqnarray*}
\NEG{D}V &=&\lambda V \\
\NEG{D}^{2}V &=&\lambda ^{2}V
\end{eqnarray*}%
However, from the Lichnerowicz-Weitzenb\"{o}ck formula (\ref{BWformula}) we
have 
\begin{eqnarray*}
D^{\ast }DV &=&\NEG{D}^{2}V-\frac{1}{4}RV \\
&=&\left( \lambda ^{2}-\frac{1}{4}R\right) V
\end{eqnarray*}%
Since $\left\vert V\right\vert ^{2}=1,$ from (\ref{d2vv1}), we have 
\begin{eqnarray*}
\left\vert DV\right\vert ^{2} &=&\left\langle D^{\ast }DV,V\right\rangle \\
&=&\lambda ^{2}-\frac{1}{4}R
\end{eqnarray*}%
Therefore, indeed, 
\begin{equation*}
D^{\ast }DV-\left\vert DV\right\vert ^{2}V=0
\end{equation*}%
and $V$ is hence a critical point of $\mathcal{E}.$
\end{proof}

\begin{remark}
From (\ref{DirV}) we see that the condition that $\NEG{D}V=\lambda V$ for a
constant $\lambda ,$ simply means that the torsion $\tilde{T}$ of $\sigma
_{V}\left( \varphi \right) $ has a constant $1$-dimensional component and a
vanishing $7$-dimensional component. Using the $G_{2}$-structure Bianchi
identity (\ref{torsionbianchi}), and in particular, the expressions from 
\cite[Proposition 3.3]{GrigorianG2Torsion1}, which are derived from it, we
then obtain that this implies both $14$- and $27$-dimensional components of $%
\tilde{T}$ are divergence-free, and hence $\func{div}\tilde{T}=0.$ This is
an alternative way to prove the above result. In fact, if $\tilde{T}$ has a
vanishing $7$-dimensional component, it is true that $\func{div}\tilde{T}=0$
if and only if $V$ is a constant norm eigenvalue of $\NEG{D}$.
\end{remark}

In general, however, we don't know if the functional $\mathcal{E}$ has any
critical points for a given metric, or if a critical point does exist,
whether it corresponds to a minimum. However, another approach, that has
been successful in the study of harmonic maps (for example, \cite{StruweChen}%
) as well as for other functionals of $G_{2}$-structures \cite%
{bryant-2003,BryantXu,GrigorianCoflow,karigiannis-2007,WeissWitt1,XuYe}
would be to consider a gradient flow of $\mathcal{E}$. Since we can always
redefine the reference $G_{2}$-structure to correspond to the initial value
of the flow, this would give the following initial value problem

\begin{equation}
\left\{ 
\begin{array}{c}
\frac{\partial V}{\partial t}=-D^{\ast }DV+\left\vert DV\right\vert ^{2}V \\ 
V\left( 0\right) =1%
\end{array}%
\right.  \label{heatflow}
\end{equation}%
The properties of this flow will be the subject of further study.

\section{Concluding remarks}

\setcounter{equation}{0}The octonion bundle formalism for $G_{2}$-structures
that has been developed in this paper raises multiple directions for further
research. The interpretation of the $G_{2}$-structure torsion as a
connection on a non-associative bundle and $\pi _{7}\func{Riem}$ as its
curvature leads to natural questions such as, what is the analogue of a
Yang-Mills connection in this case, and what is its interpretation in terms
of $G_{2}$-structures? This also ties in with the interpretation and
significance of divergence-free torsion which corresponds to the Coulomb
gauge. Then there is also the question of existence of divergence-free
connections, that is, critical points of the functional $\mathcal{E}\left(
V\right) .$ The equation (\ref{unitvecteq}) for the critical points is very
similar to the harmonic unit vector field equation, however in this case we
have additional structure - the equation can be split into real and
imaginary octonion parts and we also have the octonion product structure, so
it is possible that this could be exploited to give some answers regarding
existence. Another possible direction is to consider harmonic maps from $M$
to the unit octonion bundle of $M.$ One of the equations would be precisely (%
\ref{unitvecteq}) and there would also be an equation involving curvature.
This may also have further interpretation in terms of $G_{2}$-structures.

In this paper we have been using octonion-valued $1$-forms and $2$-forms,
however for further progress a more rigorous theory of octonion-valued
bundles is needed. For quaternions, which are non-commutative but
associative, a theory of quaternion-valued modules and bundles has been
developed by Joyce \cite{JoyceHmodules,JoyceHmodules2} and Widdows \cite%
{Widdows1,Widdows2}. For octonion-valued modules and bundles the
corresponding theory would necessarily be even more subtle due to the added
non-associativity.

Due to the relationship with octonions, manifolds with $G_{2}$-structure
have an intrinsic non-associativity, therefore it is likely that the
enigmatic nature of $G_{2}$-structure can only be truly understood by
embracing the non-associativity and using it to define new mathematical
structures.

\appendix

\section{Proofs of identities}

\setcounter{equation}{0}\label{appProofs}

\begin{proof}[Proof of Lemma \protect\ref{lemOctoexp}]
Let $A=\left( 0,\alpha \right) \in \func{Im}\Gamma \left( \mathbb{O}M\right)
,$ then the exponential of $A$\ is defined to be $e^{A}=\sum_{k=1}^{\infty }%
\frac{1}{k!}A^{k}$. From the definition (\ref{octoproddef}) of octonion
multiplication, we have 
\begin{eqnarray*}
A &=&\alpha \\
A^{2} &=&-\left\vert \alpha \right\vert ^{2} \\
A^{3} &=&-\left\vert \alpha \right\vert ^{2}\alpha \\
A^{4} &=&\left\vert \alpha \right\vert ^{4} \\
&&...
\end{eqnarray*}%
Therefore, 
\begin{eqnarray}
e^{A} &=&\left( 1-\frac{1}{2}\left\vert \alpha \right\vert ^{2}+\frac{1}{4!}%
\left\vert \alpha \right\vert ^{4}+...\right)  \notag \\
&&+\left( 1-\frac{1}{3!}\left\vert \alpha \right\vert ^{2}+\frac{1}{5!}%
\left\vert \alpha \right\vert ^{4}-...\right) \alpha  \label{octoexp2} \\
&=&\cos \left\vert \alpha \right\vert +\alpha \frac{\sin \left\vert \alpha
\right\vert }{\left\vert \alpha \right\vert }  \notag
\end{eqnarray}%
Note that this converges for any $\alpha $.
\end{proof}

\begin{proof}[Proof of Corollary \protect\ref{corrOctopow}]
Suppose $B=\left( b,\beta \right) \in \Gamma \left( \mathbb{O}M\right) .$
Since we are proving a pointwise identity, without loss of generality, we
may assume $\beta $ (and hence $B$ itself) is nowhere vanishing, since
whenever $\beta $ does vanish, the identity (\ref{octopower}) is satisfied
trivially. Then, 
\begin{equation*}
B=\left\vert B\right\vert \left( \hat{b}+\hat{\beta}\right)
\end{equation*}%
where $\hat{b}=\frac{b}{\left\vert B\right\vert }$ and $\hat{\beta}=\frac{%
\beta }{\left\vert B\right\vert }$. Since now $\hat{b}^{2}+\left\vert \hat{%
\beta}\right\vert ^{2}=1$, we can find a non-negative real number $\theta $
such that $\cos \theta =\hat{b}$ and $\sin \theta =\left\vert \hat{\beta}%
\right\vert .$ Therefore, we write 
\begin{equation*}
B=\left\vert B\right\vert \left( \cos \theta +\alpha \frac{\sin \theta }{%
\theta }\right)
\end{equation*}%
where $\alpha =\frac{\hat{\beta}}{\left\vert \hat{\beta}\right\vert }\theta
, $ so that $\left\vert \alpha \right\vert =\theta $. From Lemma \ref%
{lemOctoexp}, we can then rewrite $B=\left\vert B\right\vert e^{A}$ where $%
A=\left( 0,\alpha \right) .$ Hence, for $k\in \mathbb{Z}$ 
\begin{equation*}
B^{k}=\left\vert B\right\vert ^{k}\left( e^{A}\right) ^{k}
\end{equation*}%
It is clear from the expansion (\ref{octoexp2}) that $\left( e^{A}\right)
^{k}=e^{kA}=\cos k\theta +\alpha \frac{\sin k\theta }{\theta }.$ Hence the
result.
\end{proof}

\begin{proof}[Proof of Lemma \protect\ref{lemAssocIds}]
Let $A,B,C\in \Gamma \left( \mathbb{O}M\right) $. Then

\begin{enumerate}
\item From the expression (\ref{octoassoc}) of the associator in terms of
the $4$-form $\psi ,$ we see that $\left[ A,B,C\right] $ only depends on the
imaginary parts of $A,B,C$ - $\alpha ,\beta ,\gamma ,$ respectively. Hence, 
\begin{equation}
\left[ \bar{A},B,C\right] =\left[ -\alpha ,\beta ,\gamma \right] =-\left[
A,B,C\right]  \label{negaabcassoc}
\end{equation}

\item From Corollary \ref{corrOctopow}, $\func{Im}A^{k}=\tilde{a}\alpha $
for some $\tilde{a}\in \mathbb{R}$, hence, 
\begin{equation*}
\left[ A^{k},A,C\right] =\tilde{a}\left[ \alpha ,\alpha ,\gamma \right] =0%
\text{.}
\end{equation*}

\item From the expression for the inner product (\ref{octoinnerprod}), we
have 
\begin{eqnarray}
\left\langle A,\left[ A,B,C\right] \right\rangle &=&\frac{1}{2}\left( A%
\overline{\left[ A,B,C\right] }+\left[ A,B,C\right] \bar{A}\right)  \notag \\
&=&\frac{1}{2}\left( -A\left[ A,B,C\right] +\left[ A,B,C\right] \bar{A}%
\right)  \label{aaassoc}
\end{eqnarray}%
since $\left[ A,B,C\right] $ is pure imaginary. However, from (\ref%
{octoassoc}), 
\begin{equation*}
\left\langle A,\left[ A,B,C\right] \right\rangle =2\psi \left(
A,A,B,C\right) =0
\end{equation*}%
since $\psi $ is totally skew-symmetric. Therefore (\ref{aaassoc}) yields%
\begin{equation}
A\left[ A,B,C\right] =\left[ A,B,C\right] \bar{A}  \label{assocaa}
\end{equation}

\item First let us consider $\left[ A,AB,C\right] .$ Using (\ref%
{negaabcassoc}), and then the definition of the associator, we can write%
\begin{eqnarray}
\left[ A,AB,C\right] &=&-\left[ \bar{A},AB,C\right]  \notag \\
&=&-\bar{A}\left( \left( AB\right) C\right) +\left( \bar{A}\left( AB\right)
\right) C  \notag \\
&=&-\bar{A}\left( A\left( BC\right) -\left[ A,B,C\right] \right) +\left\vert
A\right\vert ^{2}\left( BC\right)  \notag \\
&=&\bar{A}\left[ A,B,C\right]  \label{aabcassoc}
\end{eqnarray}%
where we have used the fact that $\left[ \bar{A},A,\cdot \right] =0$.
Continuing by induction, we conclude that $\left[ A,A^{k}B,C\right] =\bar{A}%
^{k}\left[ A,B,C\right] .$

\item Similarly as above, consider $\left[ A,BA,C\right] .$ Then, 
\begin{equation*}
\left[ A,BA,C\right] =-\left[ A,\bar{A}\bar{B},C\right] =\left[ \bar{A},\bar{%
A}\bar{B},C\right]
\end{equation*}%
Applying (\ref{aabcassoc}) and (\ref{assocaa}) we hence get 
\begin{equation*}
\left[ A,BA,C\right] =A\left[ \bar{A},\bar{B},C\right] =\left[ A,B,C\right] 
\bar{A}.
\end{equation*}%
Continuing by induction for arbitrary $k$, we conclude that $\left[
A,BA^{k},C\right] =\left[ A,B,C\right] \bar{A}^{k}$.

\item Consider $\left[ A^{k},B,CA\right] .$ Using the definition of the
associator,%
\begin{eqnarray}
\left[ A^{k},B,CA\right] &=&\left[ B,CA,A^{k}\right]  \notag \\
&=&B\left( CAA^{k}\right) -\left( B\left( CA\right) \right) A^{k}  \notag \\
&=&B\left( CA^{k+1}\right) -\left( \left( BC\right) A+\left[ B,C,A\right]
\right) A^{k}  \notag \\
&=&\left[ B,C,A^{k+1}\right] -\left[ A,B,C\right] A^{k}  \label{akbcaassoc}
\end{eqnarray}%
However, note that $\func{Im}A^{k}=\tilde{a}\alpha $ for some $\tilde{a}\in 
\mathbb{R}$, hence using the skew-symmetry of the associator and part 5 of
this proof, we have 
\begin{equation*}
\left[ A^{k},B,CA\right] =\left[ A^{k},B,C\right] \bar{A}
\end{equation*}%
Therefore, we indeed obtain 
\begin{equation}
\left[ A^{k+1},B,C\right] =\left[ A^{k},B,C\right] \bar{A}+\left[ A,B,C%
\right] A^{k}  \label{akp1assoc}
\end{equation}
\end{enumerate}
\end{proof}

\begin{proof}[Proof of Lemma \protect\ref{lemRLmult}]
Let $A,B,C\in \Gamma \left( \mathbb{O}M\right) $, and consider 
\begin{eqnarray*}
\left\langle R_{B}A,C\right\rangle &=&\left\langle AB,C\right\rangle \\
&=&\frac{1}{2}\left( \left( AB\right) \bar{C}+C\left( \bar{B}\bar{A}\right)
\right) \\
&=&\frac{1}{2}\left( A\left( B\bar{C}\right) -\left[ A,B,\bar{C}\right]
+\left( C\bar{B}\right) \bar{A}+\left[ C,\bar{B},\bar{A}\right] \right) \\
&=&\frac{1}{2}\left( A\left( B\bar{C}\right) +\left( C\bar{B}\right) \bar{A}%
\right) \\
&=&\left\langle A,C\bar{B}\right\rangle =\left\langle A,R_{\bar{B}%
}C\right\rangle
\end{eqnarray*}%
where we have used the expression for the metric (\ref{octoinnerprod}) and
properties of the associator from Lemma \ref{lemAssocIds}. Similarly we
obtain the result for $L_{B}$.\newline
\end{proof}

\begin{proof}[Proof of Lemma \protect\ref{lemAssocIds2}]
Let $V$ be a nowhere-vanishing octonion, and $A,B$ arbitrary octonion
sections.

\begin{enumerate}
\item Using the associator, 
\begin{eqnarray*}
\left( VA\right) \left( BV^{-1}\right) &=&\left( \left( VA\right) B\right)
V^{-1}+\left[ VA,B,V^{-1}\right] \\
&=&\left( V\left( AB\right) -\left[ V,A,B\right] \right) V^{-1}-\frac{1}{%
\left\vert V\right\vert ^{2}}\left[ VA,B,V\right] \\
&=&\func{Ad}_{V}\left( AB\right) -\left[ A,B,V\right] V^{-1}-\frac{\bar{V}}{%
\left\vert V\right\vert ^{2}}\left[ A,B,V\right]
\end{eqnarray*}%
where we have used identity 4 from Lemma \ref{lemAssocIds}. Now using
identity 3 from Lemma \ref{lemAssocIds}, we conclude that, 
\begin{eqnarray}
\left( VA\right) \left( BV^{-1}\right) &=&\func{Ad}_{V}\left( AB\right) -%
\frac{1}{\left\vert V\right\vert ^{2}}\left[ A,B,V\right] \left( V+\bar{V}%
\right)  \notag \\
&=&\func{Ad}_{V}\left( AB\right) +\left[ A,B,V^{-1}\right] \left( V+\bar{V}%
\right)  \label{cabcinid}
\end{eqnarray}

\item Using the same identities from Lemma \ref{lemAssocIds}, we get 
\begin{eqnarray}
\left( AV^{-1}\right) \left( VB\right) &=&\left( \left( AV^{-1}\right)
V\right) B+\left[ AV^{-1},V,B\right]  \notag \\
&=&AB-\frac{1}{\left\vert V\right\vert ^{2}}\left[ A,B,V\right] V  \notag \\
&=&AB+\left[ A,B,V^{-1}\right] V  \label{acincbid}
\end{eqnarray}
\end{enumerate}
\end{proof}

\begin{proof}[Proof of Lemma \protect\ref{lemAssocVId}]
Let $A,B,C\in \Gamma \left( \mathbb{O}M\right) $, then%
\begin{eqnarray*}
\left[ A,B,C\right] _{V} &=&A\circ _{V}\left( B\circ _{V}C\right) -\left(
A\circ _{V}B\right) \circ _{V}C \\
&=&A\circ _{V}\left( BC+\left[ B,C,V\right] V^{-1}\right) -\left( AB+\left[
A,B,V\right] V^{-1}\right) \circ _{V}C \\
&=&A\left( BC\right) +\left[ A,BC,V\right] V^{-1}+A\left( \left[ B,C,V\right]
V^{-1}\right) +\left[ A,\left[ B,C,V\right] V^{-1},V\right] V^{-1} \\
&&-\left( AB\right) C-\left[ AB,C,V\right] V^{-1}-\left( \left[ A,B,V\right]
V^{-1}\right) C-\left[ \left[ A,B,V\right] V^{-1},C,V\right] V^{-1}
\end{eqnarray*}%
Now note that using Lemma \ref{lemAssocIds}, 
\begin{eqnarray*}
\left[ A,\left[ B,C,V\right] V^{-1},V\right] V^{-1} &=&\left[ A,\left[ B,C,V%
\right] ,V\right] \bar{V}^{-1}V^{-1} \\
&=&-\left[ A,\left[ B,C,V\right] ,V^{-1}\right]
\end{eqnarray*}%
and similarly, 
\begin{eqnarray*}
\left[ \left[ A,B,V\right] V^{-1},C,V\right] V^{-1} &=&-\left[ \left[ A,B,V%
\right] ,C,V^{-1}\right] \\
&=&\left[ \left[ A,B,V\right] ,V^{-1},C\right]
\end{eqnarray*}%
However, 
\begin{eqnarray*}
A\left( \left[ B,C,V\right] V^{-1}\right) -\left[ A,\left[ B,C,V\right]
,V^{-1}\right] &=&\left( A\left[ B,C,V\right] \right) V^{-1} \\
-\left( \left[ A,B,V\right] V^{-1}\right) C-\left[ \left[ A,B,V\right]
,V^{-1},C\right] &=&-\left[ A,B,V\right] \left( V^{-1}C\right)
\end{eqnarray*}%
Therefore, 
\begin{equation*}
\left[ A,B,C\right] _{V}=\left[ A,BC,V\right] V^{-1}+\left( A\left[ B,C,V%
\right] \right) V^{-1}-\left[ AB,C,V\right] V^{-1}+\left[ A,B,C\right] -%
\left[ A,B,V\right] \left( V^{-1}C\right)
\end{equation*}%
Expanding each of the first three associators, we get%
\begin{eqnarray*}
\left[ A,B,C\right] _{V} &=&\left[ A\left( \left( BC\right) V\right) -\left(
A\left( BC\right) \right) V-\left( AB\right) \left( CV\right) +\left( \left(
AB\right) C\right) V+A\left( B\left( CV\right) \right) -A\left( \left(
BC\right) V\right) \right] V^{-1} \\
&&+\left[ A,B,C\right] -\left[ A,B,V\right] \left( V^{-1}C\right) \\
&=&\left[ A\left( \left( BC\right) V\right) \right] V^{-1}-A\left( BC\right)
-\left[ \left( AB\right) \left( CV\right) \right] V^{-1}+\left( AB\right) C+%
\left[ A\left( B\left( CV\right) \right) \right] V^{-1} \\
&&-\left[ A\left( \left( BC\right) V\right) \right] V^{-1}+\left[ A,B,C%
\right] -\left[ A,B,V\right] \left( V^{-1}C\right) \\
&=&\left[ A\left( B\left( CV\right) \right) \right] V^{-1}-\left[ \left(
AB\right) \left( CV\right) \right] V^{-1}-\left[ A,B,V\right] \left(
V^{-1}C\right) \\
&=&\left[ A,B,CV\right] V^{-1}-\left[ A,B,V\right] \left( V^{-1}C\right)
\end{eqnarray*}
\end{proof}

\begin{proof}[Proof of Lemma \protect\ref{lemd2nap}]
Suppose $P=\left( p_{0},\rho \right) $ where $p_{0}\in \Omega ^{p}\left( 
\func{Re}\mathbb{O}M\right) \cong \Omega ^{p}\left( M\right) $ and $\rho \in
\Omega ^{p}\left( \func{Im}\mathbb{O}M\right) \cong \Omega ^{p}\left(
TM\right) $. Then, 
\begin{equation*}
d_{\nabla }P=\left( dp_{0},d_{\nabla }\rho \right)
\end{equation*}%
where $d$ is the ordinary exterior derivative on $\Omega ^{p}\left( M\right)
,$ and $d_{\nabla }\rho $ is given by%
\begin{equation*}
\left( d_{\nabla }\rho \right) _{b_{1}...b_{p+1}}^{\ \ \ \ \ \ \ \ \ \
\alpha }=\left( p+1\right) \nabla _{\lbrack b_{1}}\rho
_{b_{2}...b_{p+1}]}^{\ \ \ \ \ \ \ \ \ \ \ \alpha }
\end{equation*}%
Hence, 
\begin{equation*}
d_{\nabla }^{2}P=\left( d^{2}p_{0},d_{\nabla }^{2}\rho \right) =\left(
0,d_{\nabla }^{2}\rho \right)
\end{equation*}%
where 
\begin{eqnarray}
\left( d_{\nabla }^{2}\rho \right) _{b_{0}...b_{p+1}}^{\ \ \ \ \ \ \ \ \ \
\alpha } &=&\left( p+2\right) \left( p+1\right) \nabla _{\lbrack
b_{0}}\nabla _{b_{1}}\rho _{b_{2}...b_{p+1}]}^{\ \ \ \ \ \ \ \ \alpha } 
\notag \\
&=&-\frac{1}{2}\left( p+2\right) \left( p+1\right) p\left( \func{Riem}%
\right) _{\ [b_{0}b_{1}b_{2}}^{c}\rho _{\left\vert c\right\vert
b_{3}..b_{p+1}]}^{\ \ \ \ \ \ \ \ \ \ \ \ \alpha }  \label{d2nab1} \\
&&+\frac{1}{2}\left( p+2\right) \left( p+1\right) \left( \func{Riem}\right)
_{\ \beta \lbrack b_{0}b_{1}}^{\alpha }\rho _{b_{2}...b_{p+1}]}^{\ \ \ \ \ \
\ \ \ \beta }  \notag
\end{eqnarray}%
However, from the Bianchi identity for the Riemann tensor$,$ $\func{Riem}_{\
[b_{0}b_{1}b_{2}]}^{c}=0,$ therefore the first term in (\ref{d2nab1})
vanishes. The remaining term is then a combination of the wedge product
between the $2$-form $\func{Riem}$ and the $p$-form $\rho $, together with
the $\func{Riem}$ acting as an endomorphism on the $\func{Im}\mathbb{O}$
index of $\rho $. Therefore, indeed $d_{\nabla }^{2}P$ only has a pure
imaginary part which is given by $\func{Riem}\wedge \left( \func{Im}P\right) 
$.
\end{proof}

\begin{proof}[Proof of Lemma \protect\ref{lemNabAVB}]
Recall that 
\begin{equation}
A\circ _{V}B=AB+\left[ A,B,V\right] V^{-1}=\left( AV\right) \left(
V^{-1}B\right) .  \label{AVB2}
\end{equation}%
Since we want to rewrite $\nabla _{X}\left( A\circ _{V}B\right) $ in terms
of $T$, we will first express $\circ _{V}$ in terms of the original product
using (\ref{AVB2}), and we will evaluate the derivatives using the relation
from Proposition \ref{propOctoLC}:%
\begin{equation}
\nabla _{X}\left( AB\right) =\left( \nabla _{X}A\right) B+A\left( \nabla
_{X}B\right) -\left[ T_{X},A,B\right] .  \label{DXprod2}
\end{equation}%
Then, we will use (\ref{AVB2}) to rewrite all the products in terms of $%
\circ _{V}$ again.

Consider 
\begin{eqnarray}
\nabla _{X}\left( A\circ _{V}B\right) &=&\nabla _{X}\left( \left( AV\right)
\left( V^{-1}B\right) \right)  \notag \\
&=&\left( \nabla _{X}\left( AV\right) \right) \left( V^{-1}B\right) +\left(
AV\right) \nabla _{X}\left( V^{-1}B\right) -\left[ T_{X},AV,V^{-1}B\right]
\label{nabAVB1}
\end{eqnarray}%
Now let us expand the first term in (\ref{nabAVB1}) 
\begin{equation}
\left( \nabla _{X}\left( AV\right) \right) \left( V^{-1}B\right) =\left(
\left( \nabla _{X}A\right) V\right) \left( V^{-1}B\right) +\left( A\nabla
_{X}V\right) \left( V^{-1}B\right) -\left[ T_{X},A,V\right] \left(
V^{-1}B\right)  \label{nabAVB1a}
\end{equation}%
We can rewrite 
\begin{equation*}
\nabla _{X}V=-V\left( \nabla _{X}V^{-1}\right) V
\end{equation*}%
and the first term in (\ref{nabAVB1a}) can be rewritten in terms of $\circ
_{V}$. So that, 
\begin{eqnarray}
\left( \nabla _{X}\left( AV\right) \right) \left( V^{-1}B\right) &=&\left(
\nabla _{X}A\right) \circ _{V}B-\left[ A\left( V\left( \nabla
_{X}V^{-1}\right) V\right) \right] \left[ V^{-1}B\right] -\left[ T_{X},A,V%
\right] \left( V^{-1}B\right)  \notag \\
&=&\left( \nabla _{X}A\right) \circ _{V}B-\left[ \left( A\left( V\nabla
_{X}V^{-1}\right) \right) V\right] \left[ V^{-1}B\right] -\left[ A,V\nabla
_{X}V^{-1},V\right] \left( V^{-1}B\right)  \notag \\
&&-\left[ T_{X},A,V\right] \left( V^{-1}B\right)  \notag \\
&=&\left( \nabla _{X}A\right) \circ _{V}B-\left( A\left( V\nabla
_{X}V^{-1}\right) \right) \circ _{V}B-\left[ T_{X}+V\nabla _{X}V^{-1},A,V%
\right] \left( V^{-1}B\right)  \label{nabAVB1b}
\end{eqnarray}%
where again we keep rewriting products in terms of $\circ _{V}$. Note that
in the second term of (\ref{nabAVB1b}), 
\begin{equation*}
A\left( V\nabla _{X}V^{-1}\right) =A\circ _{V}\left( V\nabla
_{X}V^{-1}\right) -\left[ A,V\nabla _{X}V^{-1},V\right] V^{-1}
\end{equation*}%
and then, 
\begin{eqnarray*}
\left( A\left( V\nabla _{X}V^{-1}\right) \right) \circ _{V}B &=&\left(
A\circ _{V}\left( V\nabla _{X}V^{-1}\right) \right) \circ _{V}B+\left( \left[
A,V\nabla _{X}V^{-1},V\right] V^{-1}\right) \circ _{V}B \\
&=&\left( A\circ _{V}\left( V\nabla _{X}V^{-1}\right) \right) \circ _{V}B+ 
\left[ \left( \left[ A,V\nabla _{X}V^{-1},V\right] V^{-1}\right) V\right]
\left( V^{-1}B\right) \\
&=&\left( A\circ _{V}\left( V\nabla _{X}V^{-1}\right) \right) \circ _{V}B- 
\left[ V\nabla _{X}V^{-1},A,V\right] \left( V^{-1}B\right)
\end{eqnarray*}%
Thus, (\ref{nabAVB1b}) becomes 
\begin{equation}
\left( \nabla _{X}\left( AV\right) \right) \left( V^{-1}B\right) =\left(
\nabla _{X}A\right) \circ _{V}B-\left( A\circ _{V}\left( V\nabla
_{X}V^{-1}\right) \right) \circ _{V}B-\left[ T_{X},A,V\right] \left(
V^{-1}B\right)  \label{nabAVB1c}
\end{equation}%
Similarly, let us consider the second term of (\ref{nabAVB1}): 
\begin{eqnarray}
\left( AV\right) \nabla _{X}\left( V^{-1}B\right) &=&\left( AV\right) \left[
\left( \nabla _{X}V^{-1}\right) B+V^{-1}\nabla _{X}B-\left[ T_{X},V^{-1},B%
\right] \right]  \notag \\
&=&\left( AV\right) \left[ \left( V^{-1}\left( V\nabla _{X}V^{-1}\right)
\right) B+V^{-1}\nabla _{X}B-\left[ T_{X},V^{-1},B\right] \right]  \notag \\
&=&\left( AV\right) \left[ V^{-1}\left( \left( V\nabla _{X}V^{-1}\right)
B\right) -\left[ V^{-1},V\nabla _{X}V^{-1},B\right] +V^{-1}\nabla _{X}B%
\right]  \notag \\
&&-\left( AV\right) \left[ T_{X},V^{-1},B\right]  \notag \\
&=&A\circ _{V}\left( \left( V\nabla _{X}V^{-1}\right) B\right) +A\circ
_{V}\nabla _{X}B+\left( AV\right) \left[ T_{X}-V\nabla _{X}V^{-1},B,V^{-1}%
\right]  \label{nabAVB2a}
\end{eqnarray}%
and in the first term in (\ref{nabAVB2a}) we can write 
\begin{eqnarray}
\left( V\nabla _{X}V^{-1}\right) B &=&\left( V\nabla _{X}V^{-1}\right) \circ
_{V}B-\left[ V\nabla _{X}V^{-1},B,V\right] V^{-1}  \notag \\
A\circ _{V}\left( \left( V\nabla _{X}V^{-1}\right) B\right) &=&A\circ
_{V}\left( \left( V\nabla _{X}V^{-1}\right) \circ _{V}B\right) -A\circ
_{V}\left( \left[ V\nabla _{X}V^{-1},B,V\right] V^{-1}\right)  \notag \\
&=&A\circ _{V}\left( \left( V\nabla _{X}V^{-1}\right) \circ _{V}B\right)
-\left( AV\right) \left( V^{-1}\left[ V\nabla _{X}V^{-1},B,V\right]
V^{-1}\right)  \notag \\
&=&A\circ _{V}\left( \left( V\nabla _{X}V^{-1}\right) \circ _{V}B\right)
-\left( AV\right) \left[ V\nabla _{X}V^{-1},B,V\right] \left\vert
V^{-1}\right\vert ^{2}  \notag \\
&=&A\circ _{V}\left( \left( V\nabla _{X}V^{-1}\right) \circ _{V}B\right)
+\left( AV\right) \left[ V\nabla _{X}V^{-1},B,V^{-1}\right]
\label{nabAVB2a2}
\end{eqnarray}%
where we have used Lemma \ref{lemAssocIds} in the last two lines. Therefore,
(\ref{nabAVB2a}) becomes 
\begin{equation}
\left( AV\right) \nabla _{X}\left( V^{-1}B\right) =A\circ _{V}\left( \left(
V\nabla _{X}V^{-1}\right) \circ _{V}B\right) +A\circ _{V}\nabla _{X}B+\left(
AV\right) \left[ T_{X},B,V^{-1}\right]  \label{nabAVB2b}
\end{equation}%
Thus, replacing the first term in (\ref{nabAVB1}) with (\ref{nabAVB1c}) and
the second term in (\ref{nabAVB1})\ with (\ref{nabAVB2b}), we get 
\begin{eqnarray}
\nabla _{X}\left( A\circ _{V}B\right) &=&\left( \nabla _{X}A\right) \circ
_{V}B+A\circ _{V}\nabla _{X}B-\left[ V\nabla _{X}V^{-1},A,B\right] _{V}
\label{nabAVB3} \\
&&-\left[ T_{X},A,V\right] \left( V^{-1}B\right) +\left( AV\right) \left[
T_{X},B,V^{-1}\right]  \notag \\
&&-\left[ T_{X},AV,V^{-1}B\right]  \notag
\end{eqnarray}%
Consider the terms containing $T_{X}$ in (\ref{nabAVB3}).\ We reorder the
associators, so that when we expand them, up to parentheses, the order in
each term is $AVT_{X}V^{-1}B$%
\begin{eqnarray}
&&-\left[ T_{X},A,V\right] \left( V^{-1}B\right) -\left( AV\right) \left[
V^{-1},B,T_{X}\right] -\left[ T_{X},AV,V^{-1}B\right]  \notag \\
&=&-\left[ A,V,T_{X}\right] \left( V^{-1}B\right) -\left( AV\right) \left[
T_{X},V^{-1},B\right] +\left[ AV,T_{X},V^{-1}B\right]  \notag \\
&=&-\left[ A\left( VT_{X}\right) -\left( AV\right) T_{X}\right] \left(
V^{-1}B\right) -\left( AV\right) \left[ T_{X}\left( V^{-1}B\right) -\left(
T_{X}V^{-1}\right) B\right]  \notag \\
&&+\left( AV\right) \left( T_{X}\left( V^{-1}B\right) \right) -\left( \left(
AV\right) T_{X}\right) \left( V^{-1}B\right)  \notag \\
&=&\left( AV\right) \left( \left( T_{X}V^{-1}\right) B\right) -\left(
A\left( VT_{X}\right) \right) \left( V^{-1}B\right)  \label{nabAVB3a}
\end{eqnarray}%
Now note that 
\begin{eqnarray}
\left( T_{X}V^{-1}\right) B &=&\left( V^{-1}\left( VT_{X}V^{-1}\right)
\right) B  \notag \\
&=&V^{-1}\left( \left( \func{Ad}_{V}T_{X}\right) B\right) -\left[ V^{-1},%
\func{Ad}_{V}T_{X},B\right]  \label{nabAVB3a2}
\end{eqnarray}%
Therefore, the first term in (\ref{nabAVB3a}) becomes 
\begin{eqnarray}
\left( AV\right) \left( \left( T_{X}V^{-1}\right) B\right) &=&\left(
AV\right) \left( V^{-1}\left( \left( \func{Ad}_{V}T_{X}\right) B\right)
\right) -\left( AV\right) \left[ V^{-1},\func{Ad}_{V}T_{X},B\right]  \notag
\\
&=&A\circ _{V}\left( \left( \func{Ad}_{V}T_{X}\right) B\right) -\left(
AV\right) \left[ V^{-1},\func{Ad}_{V}T_{X},B\right]  \label{nabAVB3a3}
\end{eqnarray}%
Similarly, 
\begin{eqnarray}
A\left( VT_{X}\right) &=&A\left( \left( \func{Ad}_{V}T_{X}\right) V\right) 
\notag \\
&=&\left( A\left( \func{Ad}_{V}T_{X}\right) \right) V+\left[ A,\func{Ad}%
_{V}T_{X},V\right]  \label{nabAVB3b2}
\end{eqnarray}%
Hence, the second term in (\ref{nabAVB3a}) becomes 
\begin{equation}
\left( A\left( VT_{X}\right) \right) \left( V^{-1}B\right) =\left( A\left( 
\func{Ad}_{V}T_{X}\right) \right) \circ _{V}B+\left[ A,\func{Ad}_{V}T_{X},V%
\right] \left( V^{-1}B\right)  \label{nabAVB3b3}
\end{equation}%
Using (\ref{OctoVAB}), we rewrite 
\begin{eqnarray}
\left( \func{Ad}_{V}T_{X}\right) B &=&\left( \func{Ad}_{V}T_{X}\right) \circ
_{V}B-\left[ \func{Ad}_{V}T_{X},B,V\right] V^{-1}  \label{nabAVB3c1} \\
A\left( \func{Ad}_{V}T_{X}\right) &=&A\circ _{V}\func{Ad}_{V}T_{X}-\left[ A,%
\func{Ad}_{V}T_{X},V\right] V^{-1}  \label{nabAVB3c2}
\end{eqnarray}%
Thus, using (\ref{nabAVB3c1}) to rewrite $\left( \func{Ad}_{V}T_{X}\right)
B, $ the first term in (\ref{nabAVB3a3}) is now 
\begin{equation}
A\circ _{V}\left( \left( \func{Ad}_{V}T_{X}\right) B\right) =A\circ
_{V}\left( \left( \func{Ad}_{V}T_{X}\right) \circ _{V}B\right) -A\circ
_{V}\left( \left[ \func{Ad}_{V}T_{X},B,V\right] V^{-1}\right)
\label{nabAVB3d1}
\end{equation}%
But, 
\begin{eqnarray*}
A\circ _{V}\left( \left[ \func{Ad}_{V}T_{X},B,V\right] V^{-1}\right)
&=&\left( AV\right) \left( V^{-1}\left[ \func{Ad}_{V}T_{X},B,V\right]
V^{-1}\right) \\
&=&\left( AV\right) \left( \left[ \func{Ad}_{V}T_{X},B,V\right] \left\vert
V^{-1}\right\vert ^{2}\right) \\
&=&\left( AV\right) \left[ \func{Ad}_{V}T_{X},B,\bar{V}^{-1}\right] \\
&=&-\left( AV\right) \left[ \func{Ad}_{V}T_{X},B,V^{-1}\right]
\end{eqnarray*}%
Therefore, (\ref{nabAVB3a3}) becomes 
\begin{equation}
\left( AV\right) \left( \left( T_{X}V^{-1}\right) B\right) =A\circ
_{V}\left( \left( \func{Ad}_{V}T_{X}\right) \circ _{V}B\right)
\label{nabAVB3a4}
\end{equation}%
Similarly, using (\ref{nabAVB3c2}) to rewrite $A\left( \func{Ad}%
_{V}T_{X}\right) $, the first term in (\ref{nabAVB3b3}) is now 
\begin{eqnarray}
\left( A\left( \func{Ad}_{V}T_{X}\right) \right) \circ _{V}B &=&\left(
A\circ _{V}\func{Ad}_{V}T_{X}\right) \circ _{V}B-\left( \left[ A,\func{Ad}%
_{V}T_{X},V\right] V^{-1}\right) \circ _{V}B  \notag \\
&=&\left( A\circ _{V}\func{Ad}_{V}T_{X}\right) \circ _{V}B-\left[ A,\func{Ad}%
_{V}T_{X},V\right] \left( V^{-1}B\right)  \label{nabAVB3b4}
\end{eqnarray}%
Thus, (\ref{nabAVB3b3}) becomes: 
\begin{equation}
\left( A\left( VT_{X}\right) \right) \left( V^{-1}B\right) =\left( A\circ
_{V}\func{Ad}_{V}T_{X}\right) \circ _{V}B  \label{nabAVB3a5}
\end{equation}%
Using (\ref{nabAVB3a4}) and (\ref{nabAVB3a5}) in (\ref{nabAVB3a}), we get
that 
\begin{eqnarray}
\left( AV\right) \left( \left( T_{X}V^{-1}\right) B\right) -\left( A\left(
VT_{X}\right) \right) \left( V^{-1}B\right) &=&\left[ A,\func{Ad}_{V}T_{X},B%
\right] _{V}  \notag \\
&=&-\left[ \func{Ad}_{V}T_{X},A,B\right] _{V}  \label{nabAVB4}
\end{eqnarray}%
So overall, the terms in (\ref{nabAVB3}) that contain $T_{X}$ simplify to (%
\ref{nabAVB4}). Therefore, overall, (\ref{nabAVB3}) simplifies to 
\begin{equation}
\nabla _{X}\left( A\circ _{V}B\right) =\left( \nabla _{X}A\right) \circ
_{V}B+A\circ _{V}\nabla _{X}B-\left[ \func{Ad}_{V}T_{X}+V\nabla
_{X}V^{-1},A,B\right] _{V}  \label{nabAVB5}
\end{equation}%
which is precisely (\ref{NabAVB}) which we were trying to prove.
\end{proof}

\bibliographystyle{jhep-a}
\bibliography{refs2}

\providecommand{\href}[2]{#2}\begingroup\raggedright\begin{thebibliography}{10}

\bibitem{AgricolaSrni}
I.~Agricola, {\it The {S}rn\'\i\ lectures on non-integrable geometries with
  torsion},  {\em Arch. Math. (Brno)} {\bf 42} (2006), no.~suppl. 5--84.

\bibitem{AgricolaSpinors}
I.~Agricola, S.~G. Chiossi, T.~Friedrich and J.~H{\"o}ll, {\it Spinorial
  description of {${\rm SU}(3)$}- and {${\rm G}_2$}-manifolds},  {\em J. Geom.
  Phys.} {\bf 98} (2015) 535--555 [\href{http://arXiv.org/abs/1411.5663}{{\tt
  1411.5663}}].

\bibitem{AgricolaFriedrich1}
I.~Agricola and T.~Friedrich, {\it On the holonomy of connections with
  skew-symmetric torsion},  {\em Math. Ann.} {\bf 328} (2004), no.~4 711--748.

\bibitem{AmmannWeissWitt1}
B.~Ammann, H.~Weiss and F.~Witt, {\it A spinorial energy functional: critical
  points and gradient flow},  \href{http://arXiv.org/abs/1207.3529}{{\tt
  1207.3529}}.

\bibitem{BaezOcto}
J.~Baez, {\it The {O}ctonions},  {\em Bull. Amer. Math. Soc. (N.S.)} {\bf 39}
  (2002) 145--205.

\bibitem{BaumFriedrichSpinors}
H.~Baum, T.~Friedrich, R.~Grunewald and I.~Kath, {\em Twistors and {K}illing
  spinors on {R}iemannian manifolds}, vol.~124 of {\em Teubner-Texte zur
  Mathematik [Teubner Texts in Mathematics]}.
\newblock B. G. Teubner Verlagsgesellschaft mbH, Stuttgart, 1991.

\bibitem{BoeckxVanhecke1}
E.~Boeckx and L.~Vanhecke, {\it Harmonic and minimal vector fields on tangent
  and unit tangent bundles},  {\em Differential Geom. Appl.} {\bf 13} (2000),
  no.~1 77--93.

\bibitem{BonanG2}
E.~Bonan, {\it Sur les vari\'et\'es riemanniennes \`a groupe d'holonomie $g\sb
  2$ our $spin(7)$},  {\em C. R. Acad. Sci. Paris} {\bf 262} (1966) 127--129.

\bibitem{Bryant-1987}
R.~L. Bryant, {\it Metrics with exceptional holonomy},  {\em Ann. of Math. (2)}
  {\bf 126} (1987), no.~3 525--576.

\bibitem{bryant-2003}
R.~L. Bryant, {\it Some remarks on {G}\_2-structures},  in {\em Proceedings of
  {G}\"okova {G}eometry-{T}opology {C}onference 2005}, pp.~75--109, G\"okova
  Geometry/Topology Conference (GGT), G\"okova, 2006.
\newblock \href{http://arXiv.org/abs/math/0305124}{{\tt math/0305124}}.

\bibitem{BryantXu}
R.~L. Bryant and F.~Xu, {\it {Laplacian Flow for Closed $G_2$-Structures: Short
  Time Behavior}},  \href{http://arXiv.org/abs/1101.2004}{{\tt 1101.2004}}.

\bibitem{StruweChen}
Y.~M. Chen and M.~Struwe, {\it Existence and partial regularity results for the
  heat flow for harmonic maps},  {\em Math. Z.} {\bf 201} (1989), no.~1
  83--103.

\bibitem{CleytonIvanovCurv}
R.~Cleyton and S.~Ivanov, {\it Curvature decomposition of {$G_2$}-manifolds},
  {\em J. Geom. Phys.} {\bf 58} (2008), no.~10 1429--1449.

\bibitem{DonaldsonGauge}
S.~K. Donaldson, {\it Gauge theory: Mathematical applications},  in {\em
  Encyclopedia of {M}athematical {P}hysics}, pp.~468--481.
\newblock Academic Press/Elsevier Science, Oxford, 2006.

\bibitem{FernandezGray}
M.~Fern{\'a}ndez and A.~Gray, {\it Riemannian manifolds with structure group
  {$G\sb{2}$}},  {\em Ann. Mat. Pura Appl. (4)} {\bf 132} (1982) 19--45 (1983).

\bibitem{FriedrichNPG2}
T.~Friedrich, I.~Kath, A.~Moroianu and U.~Semmelmann, {\it On nearly parallel
  {$G_2$}-structures},  {\em J. Geom. Phys.} {\bf 23} (1997), no.~3-4 259--286.

\bibitem{GilMedranoVF1}
O.~Gil-Medrano, {\it Relationship between volume and energy of vector fields},
  {\em Differential Geom. Appl.} {\bf 15} (2001), no.~2 137--152.

\bibitem{Gray-VCP}
A.~Gray, {\it Vector cross products on manifolds},  {\em Trans. Amer. Math.
  Soc.} {\bf 141} (1969) 465--504.

\bibitem{GrigorianG2Review}
S.~Grigorian, {\it {Moduli spaces of G2 manifolds}},  {\em Rev. Math. Phys.}
  {\bf 22} (2010), no.~9 1061--1097 [\href{http://arXiv.org/abs/0911.2185}{{\tt
  0911.2185}}].

\bibitem{GrigorianCoflow}
S.~Grigorian, {\it Short-time behaviour of a modified {L}aplacian coflow of
  {G2}-structures},  {\em Adv. Math.} {\bf 248} (2013) 378--415
  [\href{http://arXiv.org/abs/1209.4347}{{\tt 1209.4347}}].

\bibitem{GrigorianG2Torsion1}
S.~Grigorian, {\it Deformations of {$G_2$}-structures with torsion},  {\em
  Asian J. Math.} {\bf 20} (2016), no.~1 123--155
  [\href{http://arXiv.org/abs/1108.2465}{{\tt 1108.2465}}].

\bibitem{GrigorianYau1}
S.~Grigorian and S.-T. Yau, {\it {Local geometry of the {G}2 moduli space}},
  {\em Comm. Math. Phys.} {\bf 287} (2009) 459--488
  [\href{http://arXiv.org/abs/0802.0723}{{\tt 0802.0723}}].

\bibitem{HanYim}
D.-S. Han and J.-W. Yim, {\it Unit vector fields on spheres, which are harmonic
  maps},  {\em Math. Z.} {\bf 227} (1998), no.~1 83--92.

\bibitem{HarveyBook}
F.~R. Harvey, {\em Spinors and calibrations}, vol.~9 of {\em Perspectives in
  Mathematics}.
\newblock Academic Press Inc., Boston, MA, 1990.

\bibitem{Hitchin:2000jd}
N.~J. Hitchin, {\it The geometry of three-forms in six dimensions},  {\em J.
  Differential Geom.} {\bf 55} (2000), no.~3 547--576
  [\href{http://arXiv.org/abs/math/0010054}{{\tt math/0010054}}].

\bibitem{Ishihara1}
T.~Ishihara, {\it Harmonic sections of tangent bundles},  {\em J. Math.
  Tokushima Univ.} {\bf 13} (1979) 23--27.

\bibitem{JoyceHmodules}
D.~Joyce, {\it Hypercomplex algebraic geometry},  {\em Quart. J. Math. Oxford
  Ser. (2)} {\bf 49} (1998), no.~194 129--162.

\bibitem{JoyceHmodules2}
D.~Joyce, {\it A theory of quaternionic algebra, with applications to
  hypercomplex geometry},  in {\em Quaternionic structures in mathematics and
  physics ({R}ome, 1999)}, pp.~143--194 (electronic).
\newblock Univ. Studi Roma ``La Sapienza'', Rome, 1999.

\bibitem{Joycebook}
D.~D. Joyce, {\em Compact manifolds with special holonomy}.
\newblock Oxford Mathematical Monographs. Oxford University Press, 2000.

\bibitem{karigiannis-2005-57}
S.~Karigiannis, {\it Deformations of {G}\_2 and {S}pin(7) {S}tructures on
  {M}anifolds},  {\em Canadian Journal of Mathematics} {\bf 57} (2005) 1012
  [\href{http://arXiv.org/abs/math/0301218}{{\tt math/0301218}}].

\bibitem{karigiannis-2007}
S.~Karigiannis, {\it Flows of ${G_2}$-{S}tructures, {I}},  {\em Q. J. Math.}
  {\bf 60} (2009), no.~4 487--522
  [\href{http://arXiv.org/abs/math/0702077}{{\tt math/0702077}}].

\bibitem{karigiannis-2006notes}
S.~Karigiannis, {\it Some {N}otes on {G}\_2 and {S}pin(7) {G}eometry},  in {\em
  Recent advances in geometric analysis}, vol.~11 of {\em Adv. Lect. Math.
  (ALM)}, pp.~129--146.
\newblock Int. Press, Somerville, MA, 2010.
\newblock \href{http://arXiv.org/abs/math/0608618}{{\tt math/0608618}}.

\bibitem{KasteMinasianFlux}
P.~Kaste, R.~Minasian and A.~Tomasiello, {\it Supersymmetric {M}-theory
  compactifications with fluxes on seven-manifolds and {$G$}-structures},  {\em
  J. High Energy Phys.} (2003), no.~7 4--21
  [\href{http://arXiv.org/abs/hep-th/0303127}{{\tt hep-th/0303127}}].

\bibitem{ConanChapter}
N.~C. Leung, {\it Geometric structures on {R}iemannian manifolds},  in {\em
  Surveys in differential geometry. {V}olume {XVI}. {G}eometry of special
  holonomy and related topics}, vol.~16 of {\em Surv. Differ. Geom.},
  pp.~161--263.
\newblock Int. Press, Somerville, MA, 2011.

\bibitem{ManogueOctos}
C.~A. Manogue and J.~Schray, {\it Finite {L}orentz transformations,
  automorphisms, and division algebras},  {\em J. Math. Phys.} {\bf 34} (1993),
  no.~8 3746--3767 [\href{http://arXiv.org/abs/hep-th/9302044}{{\tt
  hep-th/9302044}}].

\bibitem{SchaferNonassoc}
R.~D. Schafer, {\it Structure and representation of nonassociative algebras},
  {\em Bull. Amer. Math. Soc.} {\bf 61} (1955) 469--484.

\bibitem{TaoBook}
T.~Tao, {\em Nonlinear dispersive equations}, vol.~106 of {\em CBMS Regional
  Conference Series in Mathematics}.
\newblock Published for the Conference Board of the Mathematical Sciences,
  Washington, DC; by the American Mathematical Society, Providence, RI, 2006.
\newblock Local and global analysis.

\bibitem{UhlenbeckConnection}
K.~K. Uhlenbeck, {\it Connections with {$L^{p}$} bounds on curvature},  {\em
  Comm. Math. Phys.} {\bf 83} (1982), no.~1 31--42.

\bibitem{WeissWitt2}
H.~Weiss and F.~Witt, {\it Energy functionals and soliton equations for
  {$G_2$}-forms},  {\em Ann. Global Anal. Geom.} {\bf 42} (2012), no.~4
  585--610 [\href{http://arXiv.org/abs/1201.1208}{{\tt 1201.1208}}].

\bibitem{WeissWitt1}
H.~Weiss and F.~Witt, {\it A heat flow for special metrics},  {\em Adv. Math.}
  {\bf 231} (2012), no.~6 3288--3322
  [\href{http://arXiv.org/abs/0912.0421}{{\tt 0912.0421}}].

\bibitem{Widdows1}
D.~Widdows, {\it A {D}olbeault-type double complex on quaternionic manifolds},
  {\em Asian J. Math.} {\bf 6} (2002), no.~2 253--275.

\bibitem{Widdows2}
D.~Widdows, {\it Quaternionic algebra described by {$\rm Sp(1)$}
  representations},  {\em Q. J. Math.} {\bf 54} (2003), no.~4 463--481.

\bibitem{WiegminkBending}
G.~Wiegmink, {\it Total bending of vector fields on {R}iemannian manifolds},
  {\em Math. Ann.} {\bf 303} (1995), no.~2 325--344.

\bibitem{WoodUnitVF}
C.~M. Wood, {\it On the energy of a unit vector field},  {\em Geom. Dedicata}
  {\bf 64} (1997), no.~3 319--330.

\bibitem{XuYe}
F.~Xu and R.~Ye, {\it {Existence, Convergence and Limit Map of the Laplacian
  Flow}},  \href{http://arXiv.org/abs/0912.0074}{{\tt 0912.0074}}.

\bibitem{CalabiYau}
S.-T. Yau, {\it On the {R}icci curvature of a compact {K}aehler manifold and
  the complex monge-amp\`ere equation. {I}},  {\em Comm. Pure Appl. Math.} {\bf
  31} (1978) 339--411.

\end{thebibliography}\endgroup

\end{document}